\documentclass{article}
\usepackage[utf8]{inputenc}
\usepackage{amsmath,amsthm,amssymb}
\usepackage{mathrsfs}
\usepackage{amsmath} 
\usepackage{authblk}
\usepackage[comma,sort&compress, numbers]{natbib}
\usepackage{graphicx}
\usepackage{graphics}
\usepackage[font=small,labelfont=bf]{caption}
\usepackage[dvipsnames]{xcolor}  
\usepackage{bm}
\usepackage[margin = 1.2in]{geometry}
\usepackage{hyperref}
\usepackage{pgfplots}
\pgfplotsset{width = 5cm, compat = 1.17}
\numberwithin{equation}{section} 
\usepackage{enumitem} 
\usepackage{soul}

\newtheorem{prop}{Proposition}[section]
\newtheorem{thm}{Theorem}[section]
\newtheorem{lemma}{Lemma}[section]

\theoremstyle{definition}
\newtheorem{definition}{Definition}
\newtheorem{example}{Example}
\newtheorem{rem}{Remark}

\renewcommand{\P}{\mathbb{P}}
\newcommand{\F}{\mathcal{F}}

\newcommand{\R}{\mathbb{R}}
\newcommand{\E}{\mathbb{E}}

\newcommand{\var}{\mathrm{Var}}
\newcommand{\indpt}{\perp \!\!\! \perp}
\newcommand{\tikzcircle}[2][red,fill=red]{\tikz[baseline=-0.5ex]\draw[#1,radius=#2] (0,0) circle ;}

\definecolor{darkForestGreen}{rgb}{0.0,0.5,0.0}
\definecolor{darkyellow}{rgb}{0.255,0.2,0.0}
\definecolor{darkgreen}{rgb}{0.0,0.5,0.0}

\allowdisplaybreaks 

\title{A Ball Divergence Based Measure For Conditional Independence Testing} 
\author[1]{Bilol Banerjee}
\author[2]{Bhaswar B. Bhattacharya}
\author[1]{Anil K. Ghosh}
\affil[1]{Theoretical Statistics and Mathematics Unit, Indian Statistical Institute, Kolkata, India }
\affil[2]{Department of Statistics and Data Science, University of Pennsylvania, USA }
\date{\null}

\begin{document}

\maketitle
\begin{abstract} 

In this paper we introduce a new measure of conditional dependence between two random vectors  ${\bm X}$ and ${\bm Y}$ given another random vector $\bm Z$ using the ball divergence. Our measure characterizes conditional independence and does not require any moment assumptions. We propose a consistent estimator of the measure using a kernel averaging technique and derive its asymptotic distribution. Using this statistic we construct two tests for conditional independence, one in the model-${\bm X}$ framework and the other based on a novel local wild bootstrap algorithm. In the model-${\bm X}$ framework, which assumes the knowledge of the distribution of ${\bm X}|{\bm Z}$, applying the conditional randomization test we obtain a method that controls Type I error in finite samples and is asymptotically consistent, even if the distribution of ${\bm X}|{\bm Z}$ is incorrectly specified up to distance preserving transformations. More generally, in situations where ${\bm X}|{\bm Z}$ is unknown or hard to estimate,  we design a double-bandwidth based local wild bootstrap algorithm that asymptotically controls both Type I error and power. We illustrate the advantage of our method, both in terms of Type I error and power, in a range of simulation settings and also in a real data example. A consequence of our theoretical results is a general framework for studying the asymptotic properties of a 2-sample conditional $V$-statistic, which is of independent interest.  \\

\end{abstract}

\section{Introduction}
Conditional independence testing plays a central role in a variety of statistical applications such as graphical modeling \citep{lauritzen1996graphical,maathuis2018handbook}, causal discovery \citep{peters2017elements}, variable selection \citep{candes2018panning,azadkia2021simple}, dimensionality reduction \citep{li2018sufficient}, and computational biology \citep{markowetz2007inferring,dobra2004sparse}. Formally, given random vectors $\bm X, \bm Y, \bm Z$, the conditional independence problem involves testing the hypothesis:   
\begin{align}\label{eq:H0H1}
H_0: {\bm X\indpt \bm Y | \bm Z} \text{ versus } H_1:{\bm X\not\indpt \bm Y | \bm Z} ,
\end{align}
based on i.i.d. samples $\{(\bm X_i, \bm Y_i, \bm Z_i)\}_{1 \leq i \leq n}$ from the joint distribution of $(\bm X, \bm Y, \bm Z)$. This problem is well understood for parametric models, such as the multivariate Gaussian (where the classical partial correlation characterizes conditional independence) and for discrete data \citep{cochran1954some,mantel1959statistical}. However, testing conditional independence is a particularly challenging  problem for continuous distributions in the nonparametric setting. For instance,  
\citet{Bergsma2004TestingCI} showed that if ${\bm Z}$ is continuous, then there is no unbiasedly estimable measure that characterizes conditional independence. Moreover, \citet{shah2020hardness} (see also \citet{neykov2021minimax} and \citet{kim2021local}) showed that for absolutely continuous (with respect to the Lebesgue measure) random variables it is not possible to have a non-trivial test with level asymptotically uniformly bounded by some given $\alpha \in (0, 1)$ over the entire null hypothesis space $H_0$. Hence, pointwise control of Type I error is the best one can hope for, unless additional assumptions are made on the distribution of $(\bm X, \bm Y, \bm Z)$. 

Numerous nonparametric conditional independence measures, with different moment and smoothness assumptions on the underlying distribution, have been proposed over the years. These include measures based on estimating conditional cumulative distribution functions \cite{linton1996conditional,patra2016nonparametric}, conditional characteristic functions \cite{su2007}, conditional probability density functions \cite{su2008}, empirical likelihood \cite{su2014}, kernel methods \cite{fukumizu2007kernel,zhang2012kernel,doran2014permutation,scetbon2022asymptotic,strobl2019approximate,sheng2023distance,kernel2024practical}, mutual information and entropy \cite{runge2018conditional,li2024k}, Rosenblatt transformations  \cite{cai2022,song2009}, copulas \cite{Bergsma2004TestingCI,veraverbeke2011estimation}, maximal nonlinear conditional correlation \cite{huang2010}, distance correlation \cite{fan2020,szekely2014,wang2015}, estimating regression functions \cite{shah2020hardness,petersen2021,Scheidegger2022}, geometric graph-based measures \cite{azadkia2021simple,huang2020kernel,shi2024}, among several others (see \cite{li2020} for a more detailed review). 

To use a conditional dependence measure for testing the hypothesis \eqref{eq:H0H1} one has to choose a rejection threshold for the corresponding  test statistic, which, as mentioned above, is a delicate matter that usually requires additional assumptions. An approach that has gained popularity recently is the model-${\bm X}$ framework \cite{candes2018panning}, where one assumes that  $\P_{\bm X| \bm Z}$, the conditional distribution of $\bm X| \bm Z$, is known. In this case, the conditional randomization test (CRT) can be implemented by repeatedly resampling from $\P_{\bm X| \bm Z}$ to uniformly control the finite sample Type I error of any conditional dependence measure. This, in particular, has been applied to calibrate the Azadkia-Chatterjee conditional dependence measure \cite{shi2024} and the Kernel Partial Correlation (KPC) of \citet{huang2020kernel}. \citet{berrett2019} characterized the Type I error inflation of the CRT when an estimate of $\P_{\bm X| \bm Z}$ is considered, and also demonstrate situations where this estimation error is asymptotically negligible. They also introduced a variant based on non-uniform permutations which is known as the Conditional Permutation Test (CPT). The CRT, however, is inapplicable in applications where $\P_{\bm X| \bm Z}$ is unknown. In such situations there are two common approaches: (1) approximating the asymptotic null distribution \cite{shah2020hardness,zhang2012kernel,Scheidegger2022,scetbon2022asymptotic,strobl2019approximate,zhou2020test} and (2) local bootstrap algorithms \cite{su2007,su2008,su2014,wang2015,huang2010}. The former provides pointwise asymptotic Type I error control, although strong assumptions on the  estimation errors of the nuisance parameters are sometimes required for the results to hold. The latter, in its discrete variant  \citep{paparoditis2000local}, proceeds by sampling, for each $\bm Z_j$, the observations in $\{\bm X_i\}_{1 \leq i \leq n}$ with probability proportional to their `distance' (in terms of some kernel) from $\bm Z_j$, for $1 \leq j \leq n$. However, theoretical understanding of the discrete local bootstrap for conditional independence testing is limited. Another approach that circumvents the estimation of $\P_{\bm X|\bm Z}$ are local permutation methods \cite{doran2014permutation,fukumizu2007kernel,kim2021local,li2024k}, where one randomly perturbs the observations in $\{ \bm Y_i \}_{1 \leq i \leq n}$ that are `close' to the conditioning variable $\bm Z_j$, for $1 \leq j \leq n$.

In this paper, we propose a new measure of conditional dependence for Euclidean random vectors using the ball divergence \citep{pan2018ball}, a recently introduced robust measure of difference between 2 probability distributions. Our construction relies on the observation that $\bm X \indpt \bm Y| \bm Z$ if and only if the distributions of $\bm X | \bm Y, \bm Z$ and $\bm X | \bm Z$ are the same almost surely. Hence, a natural approach to constructing a conditional dependence measure is to quantify the difference between $\P_{\bm X| \bm Y, \bm Z}$ and $\P_{\bm X| \bm Z}$, and averaging it (with a possible weight function) over the marginal distribution of ${(\bm Y, \bm Z)}$. In this paper, we operationalize this approach using the ball divergence to measure the difference between the corresponding conditional distributions. We refer to this measure as the {\it conditional ball divergence}  (cBD) which is zero if and only if $\bm X \indpt \bm Y| \bm Z$, that is, it characterizes conditional independence. Moreover, the cBD, unlike several other conditional dependence measures, such as the conditional distance correlation \cite{wang2015} and the generalized covariance measure \cite{Scheidegger2022}, does not require any moment assumptions, hence it is robust against outliers and extreme values. We illustrate the efficacy of the cBD in measuring and testing conditional independence through the following results: 

\begin{itemize} 

\item In Section \ref{sec:balldivergence}, we propose an estimate of the cBD using a kernel averaging technique and study its large sample properties. To begin with we show that the cBD estimate is consistent, that is, the estimate converges to the population measure as the sample size increases (Theorem \ref{consistency}). Next, we derive the asymptotic distribution of both the pointwise and the averaged cBD estimates (see Theorem \ref{large-sam-dist-1} and \ref{large-sam-dist-2}, respectively). The latter, in particular, shows that the cBD estimate with an appropriate weight function is asymptotically normal under $H_0$. 

\item In Section \ref{sec:conditionalindependence}, we calibrate the cBD estimate to develop tests for conditional independence in the model-$\bm{X}$ framework and using a novel local wild bootstrap algorithm. In the model-$\bm{X}$ framework, by repeatedly resampling from the known distribution of $\bm X | \bm Z$ one can implement the conditional randomization test (CRT) which controls level in finite samples and is also consistent with only finite number of resamples (Proposition \ref{CRT-test}). Furthermore, since the cBD estimate depends only on the ordering of the pairwise distances between the observations, our test remains consistent even if $\P_{\bm X | \bm Z}$ is misspecified up to distance preserving transformations, such as, location shift, rotation, or homogeneous scaling.  The same also holds for the conditional permutation test (CPT) (see Proposition \ref{CPT-test}). In general, when $\P_{\bm X| \bm Z}$ is not known, we propose a novel local wild bootstrap for calibrating the cBD statistic. Here, given $\bm Z_j$, instead of directly resampling an $\bm X_i$ (based on their distance from $\bm Z_j$), for $1 \leq i, j \leq n$, we sample from a Gaussian distribution with mean $\bm X_j$ and a small variance (bandwidth), for $1 \leq j \leq n$. This allows us to design a double-bandwidth strategy, where a larger bandwidth is used for computing the test statistic and a smaller bandwidth is used for the sampling, such that both Type I error and power are asymptotically controlled (Theorem \ref{LB-test}). 

\item In Section \ref{sec:simulations}, we compare the performance of the cBD test (calibrated using the local wild bootstrap) with other competing state-of-the-art methods both in simulations and on a benchmark data set. The cBD outperforms the other tests in a variety of simulation settings. Our method is especially powerful in detecting non-linear geometric dependencies and negative correlation between the pairwise distances. Our experiments also show that the discrete local bootstrap for calibrating conditional independence tests do not always control the Type I error in finite samples. In the real data analysis we are also able to achieve higher power with a smaller sample size compared to competing methods. 

\end{itemize}

A consequence of our theoretical analysis, is a general result about the asymptotic properties of a 2-sample conditonal $V$-statistic (Section \ref{sec:empiricalprocess}). This framework can be applied more generally for studying the limiting properties of kernel density based estimates for conditional inference. As an illustration, we outline in Remark \ref{remark:cdc-limit-dist} the derivation of the pointwise limit of the conditional distance correlation. We discuss possible extensions and directions for future research in Section \ref{sec:conclusion}. The proofs of the results are deferred to the Appendix. 

\section{Conditional Ball Divergence}
\label{sec:balldivergence} 

Suppose $\bm X, \bm Y$, $\bm Z$ are random vectors in $\mathbb R^{d_X}$, $\mathbb R^{d_Y}$, and $\mathbb R^{d_Z}$, respectively. Throughout we will assume that the joint distribution of $(\bm X, \bm Y, \bm Z)$ has density $p_{\bm X, \bm Y, \bm Z}$ in $\mathbb R^{d_X+d_Y+d_Z}$, with marginal densities denoted by $p_{\bm X}, p_{\bm Y}, p_{\bm Z}$, $p_{\bm X, \bm Y}$, $p_{\bm Y, \bm Z}$, and $p_{\bm X, \bm Z}$, with respect to Lebesgue measures in the appropriate dimensions. We begin by recalling the definition of the ball divergence, specialized to the case of probability measures on Euclidean spaces \cite{pan2018ball}. Throughout we will assume that the Euclidean space is equipped with a metric $\rho$ such that it becomes a separable metric space. 

\begin{definition}[\citet{pan2018ball}]
The ball divergence between two probability measures $P$ and $Q$ in $\R^d$ is defined as
\begin{align}
\Theta^2(P, Q)=\int\int (P-Q)^2 \bar{B}( \bm u, \rho( \bm v, \bm u)) \{ \mathrm dP( \bm u) \mathrm d P (\bm v)+ \mathrm d Q(\bm u) \mathrm d Q( \bm v)\}, 
\label{eq:uv}
\end{align}
where $\bar{B}( \bm u, \varepsilon)$ is the closed ball of radius $\varepsilon$ centered at $ \bm u$ in $\R^d$. (Here, for any measurable set $A \subseteq \R^d$, $(P-Q)^2 (A) = (P(A) - Q(A))^2$.)
\end{definition} 

A key property of the ball divergence is that $\Theta^2(P, Q) = 0$ if and only if $P=Q$ \citep[Theorem 1]{pan2018ball}. Now, note that ${\bm X\indpt \bm Y | \bm Z}$ if and only if the distributions of $\bm X | \bm Y, \bm Z$ and $\bm X| \bm Z$ are the same almost surely. This motivates the following measure of conditional dependence.  
 
\begin{definition} 
Suppose $\bm X, \bm Y$, $\bm Z$ are random vectors in $\mathbb R^{d_X}, \mathbb R^{d_Y}$, and $\mathbb R^{d_Z}$, respectively, and $a : \R^{d_Y} \times \R^{d_Z} \rightarrow [0, \infty)$ is a bounded, measurable weight function which is positive on the support of $p_{\bm Y,\bm Z}$. Then the {\it conditional ball divergence} (cBD) between $\bm X$ and $\bm Y$ given $\bm Z$, with weight function $a$, is defined as 
\begin{align}\label{eq:XYZ}
\zeta_{a}(\bm X, \bm Y| \bm Z)
& =\E\left [ \Theta^2(\P_{\bm X| \bm Y, \bm Z}, \P_{\bm X| \bm Z}) a({\bm Y, \bm Z}) \right] , 
\end{align}
where the expectation is with respect to the joint distribution of $(\bm Y, \bm Z)$. 
 \label{def:cBD}
 \end{definition} 
 
The following proposition shows that the cBD characterizes conditional dependence without any additional assumptions on the underlying distribution.

\begin{prop} \label{nice-theo-prop} For any weight function $a$ as in Definition \ref{def:cBD}, $\zeta_a(\bm X, \bm Y| \bm Z)\geq 0$. Moreover, the equality holds if and only if ${\bm X\indpt \bm Y | \bm Z}$.
\end{prop}

\begin{proof}
    The non-negativity of $\zeta_a(\bm X, \bm Y| \bm Z)$ is obvious. Also, note that if ${\bm X\indpt \bm Y | \bm Z}$, then the distributions $\P_{\bm X| \bm Y, \bm Z}$ and $\P_{\bm X| \bm Z}$ are the same almost surely. Hence, by the distribution characterization property of the ball divergence \citep[Theorem 1]{pan2018ball}, $\Theta^2(\P_{\bm X| \bm Y, \bm Z},\P_{\bm X| \bm Z}) = 0$ almost surely. This implies, recalling \eqref{eq:XYZ}, $\zeta_a(\bm X, \bm Y | \bm Z)=0$, whenever ${\bm X\indpt \bm Y | \bm Z}$. 
    
    Conversely, if $\zeta_a(\bm X, \bm Y | \bm Z)=0$, then by the positivity of $a$ we have, $\Theta^2(\P_{\bm X| \bm Y, \bm Z}, \P_{\bm X| \bm Z})=0$ almost surely. This implies, by the property of the ball divergence, $\P_{\bm X| \bm Y, \bm Z} = \P_{\bm X| \bm Z}$ almost surely, hence, ${\bm X\indpt \bm Y | \bm Z}$.
\end{proof}

\begin{rem}
Note that the measure $\zeta_a(\bm X, \bm Y| \bm Z)$ is asymmetric in ${\bm X}$ and ${\bm Y}$. It is easy to symmetrize the measure by considering the average/maximum of $\zeta_a(\bm X, \bm Y| \bm Z)$ and $\zeta_a(\bm Y, \bm Z| \bm Z)$. The subsequent results will continue hold for the symmetrized variants, however, for analytic simplicity, we study the asymmetric measure in this article. 
\end{rem}

\begin{rem} 
Similar to the Chatterjee-Azadkia conditional dependence measure \cite{azadkia2021simple}  and the Kernel Partial Correlation (KPC)  \cite{huang2020kernel}, the cBD can be normalized to be within $[0, 1]$, where 0 characterizes conditional independence and 1 characterizes complete conditional dependence (see Section \ref{sec:conclusion} and Proposition \ref{prop:normalized-cBD} for further details).  Hence, the cBD, unlike covariance-based measures such as the conditional distance correlation \cite{wang2015}, has the ability to measure the `strength' of (possibly) non-linear dependence between $\bm X$ and $\bm Y$ given $\bm Z$. 
\end{rem}

\subsection{Consistent Estimation} 
\label{sec:estimation}

We now describe a procedure for estimating the cBD based on i.i.d. samples $\{(\bm X_i, \bm Y_i, \bm Z_i)\}_{1 \leq i \leq n}$ from the joint distribution $p_{\bm X, \bm Y, \bm Z}$, using a kernel averaging method. Hereafter, we will choose  $\rho$ to be the Euclidean metric, to be denoted by $\| \cdot \|$, although most of our results will continue to hold for general $L_p$ metrics. Also, suppose $K: [0,\infty) \rightarrow [0,\infty)$ is a non-negative symmetric kernel function and $h_{1}, h_{2} > 0$ are bandwidth parameters. Throughout, unless specified otherwise, vectors will be denoted in row format. Then for any Borel set $A \subset \mathbb R^{d_X}$ define, 
\begin{align}\label{eq:Pyzestimate}
& \Tilde{\mathbb P}_{\bm X| \bm Y = \bm y, \bm Z = \bm z}(A) :=\frac{\sum_{i=1}^n K\left(\frac{\bm \|(\bm y, \bm z) -(\bm Y_i, \bm Z_i) \|}{h_{1}}\right)I_{A}\{{\bm X}_i\}}{\sum_{i=1}^n K\left(\frac{\|(\bm y, \bm z)  - (\bm Y_i, \bm Z_i)  \|}{h_{1}}\right)} 
\end{align}
and 
\begin{align}
\label{eq:Pzestimate}
\Tilde{\mathbb P}_{\bm X| \bm Z = \bm z}(A) :=\frac{\sum_{i=1}^n K\left(\frac{\| \bm z- \bm Z_i\|}{h_{2}}\right)I_{A}\{{\bm X}_i\}}{\sum_{i=1}^n K\left(\frac{\bm \| \bm z - \bm Z_i\|}{h_{2}}\right)}  , 
\end{align} 
as the sample analogues of $\mathbb P_{\bm X| \bm Z = \bm z}(A)$ and $\mathbb P_{\bm X\mid \bm Y = \bm y, \bm Z = \bm z}(A)$, respectively. Here, $I_A(\bm t) = \bm 1 \{\bm t \in A\}$ denotes the indicator function of the event $A$. Then the cBD measure (recall \eqref{eq:XYZ}) can be estimated as follows:  
\begin{align}
\hat{\zeta}_{a, n} = \frac{1}{n}\sum_{s=1}^n \Theta^2(\Tilde{ \mathbb P}_{\bm X| \bm  Y_s, \bm Z_s},\Tilde{\mathbb P}_{\bm X| \bm Z_s})a({\bm   Y_s, \bm Z_s}). 
\label{eq:XYZestimate}
\end{align}
The estimate \eqref{eq:XYZestimate} can be expanded out using the definition of the ball divergence (recall \eqref{eq:uv}). For this, define 
\begin{align}\label{eq:kw}
w_{(\bm Y_i, \bm Z_i)}(\bm y, \bm z) := K\left(\frac{\|(\bm y, \bm z)  - (\bm Y_i, \bm Z_i)  \|}{h_{1}}\right) \text{ and } w_{\bm Z_i}(\bm z) := K\left(\frac{\| \bm z - \bm Z_i \|}{h_{2}}\right) , 
\end{align}
and $w(\bm y, \bm z) := \sum_{i=1}^n w_{(\bm Y_i, \bm Z_i)}(\bm y, \bm z)$ and $w(\bm z) := \sum_{i=1}^n w_{\bm Z_i}(\bm z)$. 
Then denoting $\delta(\bm x, \bm y, \bm z) = \bm 1 \{ \bm z \in \bar B ( \bm x, \| \bm x - \bm y \|)\}$, the indicator that $\bm z$ lies inside in the ball  $\bar B ( \bm x, \| \bm x - \bm y \|)$,  we can express \eqref{eq:XYZestimate} as:  
\begin{align}\label{eq:AB} 
\hat{\zeta}_{a, n}& = \frac{1}{n} \sum_{s=1}^n \left( A_{n, s} + C_{n, s} \right) a(\bm Y_s, \bm Z_s) , 
\end{align}
where 
$$A_{n, s} := \sum_{1 \leq u, v \leq n} 
   \left( \sum_{r = 1}^n \left\{ \frac{w_{(\bm Y_r, \bm Z_r)}(\bm Y_s, \bm Z_s)}{w(\bm Y_s, \bm Z_s)}   - \frac{w_{\bm Z_r}(\bm Z_s)}{w(\bm Z_s)} \right\} \delta(\bm X_u, \bm X_v, \bm X_r) \right)^2
 \frac{w_{ (\bm Y_u, \bm Z_u) } (\bm Y_s, \bm Z_s) w_{ (\bm Y_v, \bm Z_v) }(\bm Y_s, \bm Z_s) }{ w(\bm Y_s, \bm Z_s)^2 }, $$ 
 and 
 $$C_{n, s} := \sum_{1 \leq u, v \leq n} 
   \left( \sum_{r = 1}^n \left\{ \frac{w_{(\bm Y_r, \bm Z_r)}(\bm Y_s, \bm Z_s)}{w(\bm Y_s, \bm Z_s)}   - \frac{w_{\bm Z_r}(\bm Z_s)}{w(\bm Z_s)} \right\} \delta(\bm X_u, \bm X_v, \bm X_r) \right)^2
 \frac{w_{ \bm Z_u } (\bm Z_s) w_{ \bm Z_v}(\bm Z_s) }{ w(\bm Z_s)^2 }.$$ 

To derive the asymptotic properties of $\hat{\zeta}_{a, n}$ we will need the underlying distributions to be smooth in a certain sense. This is formalized in the following definition:

\begin{definition}
Let $f(\bm x, \bm y)$ be a density function with marginal densities denoted as $f_{\bm X}$ and $f_{\bm Y}$, respectively. For any sequence $\delta_{n}$ that converges to zero define
$$R_n(\bm x, \bm y; \bm u) := \frac{1}{\delta_n^2} \left\{\frac{f(\bm x, \bm y+\delta_n\bm u)}{f(\bm x, \bm y)}-1-\delta_n \left( \frac{\frac{\partial}{\partial \bm y} f(\bm x,\bm y)}{f(\bm x,\bm y)} \right) \bm u ^\top  \right\}.$$
We say $f$ is {\it nice with respect to $\bm Y$}, if 
\begin{align}\label{eq:condition}
\sup\limits_{\bm u:\|\bm u\|=1}\limsup\limits_{n\rightarrow\infty} \int \big(R_{n}(\bm x, \bm y;\bm u)\big)^2 f(\bm x,\bm y)\mathrm d \bm x  \text{ is uniformly bounded in } \bm y. 
\end{align}
\label{definition:condition}
\end{definition}

A sufficient condition for $f(\bm x, \bm y)$ to be nice in $\bm Y$ is that the variation of the first and second order partial derivatives of the logarithm of $f_{\bm Y|\bm X}$ are bounded (see Lemma \ref{lm:condition} in Appendix \ref{sec:conditionpf} for details). 

To derive the large sample properties of $\hat{\zeta}_{a, n}$ we will make the following assumptions on the distributions, the kernels, and the bandwidths: 

\begin{enumerate}[label= (A\arabic*)]
    \item  \label{YZ} The joint densities $p_{\bm X, \bm Z}$ and $p_{\bm X, \bm Y, \bm Z}$ are nice with respect to $\bm Z$ and $(\bm Y,\bm Z)$, respectively.  Moreover, $p_{\bm X, \bm Y, \bm Z}$ and its marginals are uniformly bounded.
    
    \item \label{K} The kernel $K$ is uniformly bounded, non-negative, symmetric about zero, and compactly supported. Moreover,
       $$\int_{\mathbb R^d} K(\|{\bm u}\|) \mathrm d{\bm u} = 1 \text{ and } \int_{\mathbb R^d} \bm u^\top  \bm u  K(\|{\bm u}\|) \mathrm d{\bm u}= \tau \bm I_d , $$ 
    where $\bm I_d$ is denotes the $d \times d$ identity matrix and $\tau > 0$ is some constant. Here, $d=d_Y+d_Z$ or $d=d_Z$ depending on the context. 

    \item \label{h} The bandwidths $h_{1},h_{2}$ depend on the sample size $n$ in such a way that 
    $$h_{1}^{d_{1}}/h_{2}^{d_{2}} \rightarrow 0,~nh_i^{d_i}\rightarrow \infty, \text{ and } nh_i^{d_i+4} \rightarrow 0,~~\textnormal{ for $i = 1,2$}$$ 
    as $n \rightarrow \infty$, where $d_{1} = d_Y+d_Z$ and $d_{2} = d_Z$.  (Here, $d_Y$ and $d_Z$ denote the dimensions of ${\bm Y}$ and ${\bm Z}$, respectively). 
\end{enumerate}

Assumption \ref{K} is standard in the nonparametric kernel regression and density estimation literature. For example, it is satisfied by the commonly used Epanechnikov kernel which we will also use in our experiments. Assumption \ref{h} ensures that $h_{1}$ and $h_{2}$ are neither too large nor too small, and they maintain a relative ordering between themselves. Under these assumptions we have the following result (see Appendix \ref{sec:consistencypf} for the proof). 

\begin{thm} Suppose Assumptions \ref{YZ}, \ref{K}, and \ref{h} hold. Then, $n \rightarrow \infty$, 
$$\hat{\zeta}_{a, n}\stackrel{P}{\rightarrow}\zeta_{a}(\bm X, \bm Y|\bm Z),$$  that is, $\hat{\zeta}_{a, n}$ is a consistent estimator of $\zeta_{a}(\bm X, \bm Y|\bm Z)$. 
\label{consistency}
\end{thm}

\subsection{Asymptotic Distribution} 
\label{sec:asymptoticdistribution}

We now derive the asymptotic distribution of the cBD estimate. We begin by  reviewing some properties of the ball divergence. Towards this, let $\delta(\bm x, \bm y, \bm z) = \bm 1 \{ \bm z \in \bar B ( \bm x, \rho(\bm x, \bm y)\} = \bm 1 \{ \rho( \bm x, \bm z) \leq \rho(\bm x, \bm y) \}$ and define $$\eta(\bm x, \bm  y, \bm z, \bm z') := \delta(\bm x, \bm y, \bm z) \delta(\bm x, \bm y, \bm z').$$ 
Then given 2 distributions $P$ and $Q$ and i.i.d. samples $\bm U_1, \bm U_2, \bm U_3, \bm U_4$ and $\bm V_1, \bm V_2, \bm V_3, \bm V_4$ from $P$ and $Q$, respectively, one can write the ball divergence \eqref{eq:uv} as (from the proof of \cite[Theorem 3]{pan2018ball}): 
\begin{align}\label{eq:UV}
\Theta^2(P, Q) = \mathbb E [ \phi(\bm U_1, \bm U_2, \bm U_3, \bm U_4; \bm V_1, \bm V_2, \bm V_3, \bm V_4)  ] , 
\end{align} 
where 
\begin{align}\label{eq:ballphi} 
\phi(\bm u_1, \bm u_2, \bm u_3, \bm u_4; \bm v_1, \bm v_2, \bm v_3, \bm v_4) = \phi_A(\bm u_1, \bm  u_2, \bm u_3, \bm u_4, \bm v_3, \bm v_4) + \phi_C(\bm v_1, \bm v_2, \bm  v_3, \bm v_4, \bm u_3, \bm u_4) , 
\end{align} 
with 
\begin{align}
\phi_A(\bm u_1, \bm  u_2, \bm u_3, \bm u_4, \bm v_3, \bm v_4) &:= \eta(\bm u_1, \bm  u_2, \bm u_3, \bm u_4) + \eta(\bm u_1, \bm  u_2, \bm v_3, \bm v_4) - \eta(\bm u_1, \bm  u_2, \bm u_3, \bm v_3) - \eta(\bm u_1, \bm  u_2, \bm u_4, \bm v_4) , \nonumber \\ 
\phi_C(\bm v_1, \bm v_2, \bm  v_3, \bm v_4, \bm u_3, \bm u_4) & := \eta(\bm v_1, \bm  v_2, \bm v_3, \bm v_4) + \eta(\bm v_1, \bm  v_2, \bm u_3, \bm u_4) - \eta(\bm v_1, \bm  v_2, \bm v_3, \bm u_3) - \eta(\bm v_1, \bm  v_2, \bm v_4, \bm u_4). \nonumber 
\end{align} 
Note that $|\phi| \leq 2$. This shows that $\Theta^2(F, G)$ is a 2-sample divergence measure with a bounded core function $\phi$ of degree $(4,4)$. Although function $\phi$ is not symmetric, it can be easily symmetrized as follows: 
\begin{align}\label{eq:thetauv}
\Theta^2(\mu,\nu) = \mathbb E [ \phi'(\bm U_1, \bm U_2, \bm U_3, \bm U_4; \bm V_1, \bm V_2, \bm V_3, \bm V_4)  ], 
\end{align}
where $\phi'$ is the symmetrized version of the core function $\phi$ defined as, 
\begin{align} \label{eq:core-function}
\phi'( \underline{\bm u}; \underline{\bm v}) = \frac{1}{4!^2} \sum_{\tau \in S_4} \sum_{\sigma \in S_4} \phi(\underline{\bm u}_{\tau}; \underline{\bm v}_{\kappa} ) . 
\end{align}
Here, $S_4$ denotes the set of all permutations of $\{1, 2, 3, 4\}$, $\underline{\bm u}_{\tau} = ( \bm u_{\tau(1)}, \bm u_{\tau(2)}, \bm u_{\tau(3)}, \bm u_{\tau(4)} )$, and $\underline{\bm v}_{\sigma} =(\bm v_{\sigma(1)}, \bm v_{\sigma(2)}, \bm v_{\sigma(3)}, \bm v_{\sigma(4)} )$, for $\tau, \sigma \in S_4$. \citet{pan2018ball} showed that when $P=Q$, the function $\phi'$ is a first-order degenerate kernel, that is, 
\begin{align}
\mathbb E[\phi'(\bm U_1, \bm U_2, \bm U_3, \bm U_4; \bm V_1, \bm V_2, \bm V_3, \bm V_4)| \bm U_1] = 0 \text{ and } \mathbb E[\phi'(\bm U_1, \bm U_2, \bm U_3, \bm U_4; \bm V_1, \bm V_2, \bm V_3, \bm V_4)| \bm V_1] = 0 , 
\label{eq:degenercy-ball-core}
\end{align} 
almost surely, but $\E[\phi^\prime(\bm U_1,\ldots,\bm U_4; \bm V_1,\ldots, \bm V_4)|\bm U_1,\bm U_2] \not = 0$.


Using \eqref{eq:thetauv}, the cBD measure \eqref{eq:XYZ} can be expressed as: 
$$\zeta_{a, n} = \E \left[
\phi'(\bm X_1, \bm X_2, \bm X_3, \bm X_4; \bm X'_1, \bm X'_2, \bm X'_3, \bm X'_4)  a(\bm Y, \bm Z) \right],$$
where $(\bm X, \bm Y,  \bm Z)$ is distributed according to $\bm p_{\bm X, \bm Y, \bm Z}$, $\bm X_1, \bm X_2, \bm X_3, \bm X_4$ are i.i.d. $\P_{\bm X| \bm Y , \bm Z}$, and $\bm X_1', \bm X_2', \bm X_3', \bm X_4'$ are i.i.d. $\P_{\bm X| \bm Z}$. 
We know from Theorem \ref{consistency} that $\zeta_{a, n}$ can be consistently estimated using   
\begin{align}\label{eq:zetasum}
\hat\zeta_{a, n} = \frac{1}{n}\sum_{s=1}^n \Theta^2(\tilde \P_{\bm X| \bm Y_s,\bm Z_s}, \tilde \P_{\bm X| \bm Z_s}) a(\bm Y_s,\bm Z_s).
\end{align}
In the following theorem we obtain the limiting distribution of the individual terms in \eqref{eq:zetasum}, conditioned on $\bm Y= \bm y$ and $\bm Z = \bm z$,  both under $H_0$ and $H_1$. Throughout, we interpret stochastic integrals in the Wiener sense \cite{integralhomogeneous} (see also \cite[Chapter 9]{stochasticintegralbook}), where, unlike in the Wiener-It\^{o} integral, the `diagonal' is not removed.

\begin{thm}
  Assume the conditions of Theorem \ref{consistency} holds. Then we have the following results:

  \begin{itemize}
      \item[\textnormal{(1)}] Under $H_0,$ that is, $\bm X\indpt \bm Y|\bm Z$, for any fixed ${\bm y} \in \mathbb R^{d_Y}$ and ${\bm z} \in \mathbb R^{d_Z}$  with $p_{\bm Y, \bm Z}(\bm y, \bm z) > 0$,
       \begin{align} \label{eq:cibdyzH0}
           & nh_{1}^{d_Y+d_Z}\Theta^2(\Tilde{\mathbb P}_{\bm X\mid  \bm Y = \bm y, \bm Z = \bm z}, \Tilde{\mathbb P}_{\bm X | \bm Z= \bm z}) \nonumber \\ 
           & \stackrel{D}{\rightarrow} {4\choose 2} \int Q_0 (\bm x, \bm x'; \bm y, \bm z) \mathrm d\mathbb{G}_{\mathbb{P}_{\bm X| \bm Y= \bm y, \bm Z = \bm z}}({\bm x}) \mathrm d\mathbb{G}_{\mathbb{P}_{\bm X| \bm Y= \bm y, \bm Z = \bm z}}(\bm x') , 
       \end{align} 
       where 
       \begin{itemize} 
       
       \item $Q_0 (\bm x, \bm x'; \bm y, \bm z) = \E\left[\phi'(\bm x,\bm x', \bm X_1,\bm X_2 ; \bm X_3,\bm X_4, \bm X_5,\bm X_6)\right]$, 
       where the expectations are taken with respect to $\bm X_1,\ldots, \bm X_6 \sim \mathbb P_{\bm X\mid  \bm Y = \bm y, \bm Z = \bm z}$ and $\phi'$ is defined in \eqref{eq:core-function}.
       
       \item $\mathbb{G}_{\mathbb{P}_{\bm X| \bm Y= \bm y, \bm Z = \bm z}}$ is a Gaussian process indexed by $L^2$ functions in $\mathbb R^{d_X}$ with covariance function as in \eqref{eq:covarianceyz}. 
       \end{itemize}

       \item[\textnormal{(2)}] Under $H_1$, that is, $\bm X\not\indpt \bm Y|\bm Z$, for any fixed ${\bm y} \in \mathbb R^{d_Y}$ and ${\bm z} \in \mathbb R^{d_Z}$  with $p_{\bm Y, \bm Z}(\bm y, \bm z) > 0$, we have 
       \begin{align} \label{eq:cibdyzH1}
           & \sqrt{nh_{1}^{d_Y+d_Z}}\left(\Theta^2(\Tilde{\mathbb P}_{\bm X\mid  \bm Y = \bm y, \bm Z = \bm z}, \Tilde{\mathbb P}_{\bm X | \bm Z= \bm z})-\Theta^2(\mathbb P_{\bm X\mid  \bm Y = \bm y, \bm Z = \bm z}, \mathbb P_{\bm X | \bm Z= \bm z})\right) \nonumber \\ 
           & \stackrel{D}{\rightarrow} 4\int Q_1 (\bm x; \bm y, \bm z) \mathrm d\mathbb{G}_{\mathbb{P}_{\bm X| \bm Y= \bm y, \bm Z = \bm z}}({\bm x}) \stackrel{D} = 4 N(0, \mathrm{Var}_{\P_{\bm X|\bm Y= \bm y, \bm Z= \bm z}}[Q_1 (\bm X; \bm y, \bm z)) ] , 
       \end{align} 
       where $\mathbb{G}_{\mathbb{P}_{\bm X| \bm Y= \bm y, \bm Z = \bm z}}$ is as defined before and 
       \begin{equation*}
           Q_1 (\bm x; \bm y, \bm z) = \E\left[\phi'(\bm x,\bm X_1,\bm X_2,\bm X_3; \bm X_1',\bm X_2',\bm X_3',\bm X_4')\right]   , 
       \end{equation*}
       with $\bm X_1, \bm X_2,\bm X_3 \sim \mathbb P_{\bm X\mid  \bm Y = \bm y, \bm Z = \bm z}$, $\bm X_1',\ldots, \bm X_4' \sim \mathbb P_{\bm X| \bm Z=\bm z}$. 
  \end{itemize}
  \label{large-sam-dist-1}
\end{thm}

The proof of Theorem \ref{large-sam-dist-1} is given in Appendix \ref{sec:asymptoticpf}. The proof follows from a more general result about the limiting distribution of conditional $V$-statistics that might be of independent interest (see Section \ref{sec:empiricalprocess}).

\begin{rem}\label{remark:distributionH0H1} 
Theorem \ref{large-sam-dist-1} shows that the pointwise limit of the cBD estimate has a Gaussian distribution under $H_1$ (see \eqref{eq:cibdyzH1}). 
However, under $H_0$, the asymptotic variance of the Gaussian is zero and the limit in \eqref{eq:cibdyzH1} is degenerate. This is because the ball divergence core function $\phi'$ is conditionally first-order degenerate under $H_0$ (see Definition \ref{def:Vphi}). As a result, the limiting distribution in this case is given by a bivariate stochastic integral  (see \eqref{eq:cibdyzH0}). To express the limit in \eqref{eq:cibdyzH0} in an alternative form, note that by the spectral theorem, for $\bm y \in \mathbb R^{d_Y}, \bm z \in \mathbb R^{d_Z}$ fixed, 
\begin{align}\label{eq:spectralQ}
Q_0 (\bm x, \bm x'; \bm y, \bm z) = \sum_{i=1}^\infty \lambda_i \phi_i (\bm x) \phi_i(\bm x'), 
\end{align}
where $\{\lambda_i\}_{i \geq 1}$ and $\{\phi_i\}_{i \geq 1}$ are the eigenvalues and the eigenvectors of the operator 
$$\mathcal{H}(f)(\bm x):= \int Q_0 (\bm x, \bm x'; \bm y, \bm z) f( \bm x') \mathrm d \P_{\bm X|\bm Y= \bm y, \bm Z= \bm z}(\bm x') , $$
for $f \in L_2(\R^{d_X}, \P_{\bm X|\bm Y= \bm y, \bm Z= \bm z})$. 
Denote $U_i := \int \phi_i (\bm x) \mathrm d\mathbb{G}_{\mathbb{P}_{\bm X| \bm Y= \bm y, \bm Z = \bm z}}({\bm x})$, for $i \geq 1$. By the orthogonality of the eigenvectors, $\{U_i\}_{i \geq 1}$ is a collection of independent Gaussian random variables with mean zero and variance $\mathrm{Var}_{\mathbb{P}_{\bm X| \bm Y= \bm y, \bm Z = \bm z}}[\phi_i (\bm X)]$. Hence, from \eqref{eq:spectralQ}, 
$$ \int Q_0 (\bm x, \bm x'; \bm y, \bm z) \mathrm d\mathbb{G}_{\mathbb{P}_{\bm X| \bm Y= \bm y, \bm Z = \bm z}}({\bm x}) \mathrm d\mathbb{G}_{\mathbb{P}_{\bm X| \bm Y= \bm y, \bm Z = \bm z}}(\bm x') \stackrel{D} = \sum_{i=1}^n \lambda_i U_i^2 , $$ 
which is the familiar representation of the distribution of a degenerate $V$-statistics as an infinite weighted sum of independent $\chi^2_1$ random variables. 
\end{rem}

Now, we proceed to derive the limiting distribution of the cBD estimate \eqref{eq:XYZestimate}. One challenge towards this is the random denominators in \eqref{eq:XYZestimate} arising from the kernel density estimates. This issue is usually handled either by assuming the density of $\bm Y, \bm Z$ is bounded below or by choosing an appropriate weight function. Here, we implement the latter; as in \cite{ke2020expected,su2007,wang2015}. The idea is to choose a weight function such that the random denominator in the corresponding estimate is cancelled. This leads us to define the following cBD measure with weight function $a(\bm y, \bm z) = a^*(\bm y, \bm z) := p_{\bm Y, \bm Z}(\bm y, \bm z)^4 p_{\bm Z}(\bm z)^4$: 
\begin{align}\label{eq:weightXYZ}
\zeta^*(\bm X,\bm Y|\bm Z):=\zeta_{a^*}(\bm X,\bm Y|\bm Z) = \E[ \Theta^2(\P_{\bm X| \bm Y, \bm Z}, \P_{\bm X| \bm Z}) p_{\bm Y, \bm Z}(\bm Y, \bm Z)^4 p_{\bm Z}(\bm Z)^4 ].
\end{align}
The natural empirical estimate of $\zeta^*$ is given by: 
\begin{align*}
\tilde{\zeta}_n^{*} & := \frac{1}{n}\sum_{s=1}^n \Theta^2(\Tilde{ \mathbb P}_{\bm X| \bm  Y_s, \bm Z_s},\Tilde{\mathbb P}_{\bm X| \bm Z_s})\hat p_{\bm Y,\bm Z}^4({\bm Y_s, \bm Z_s})\hat p_{\bm Z}^4 (\bm Z_s) , 
\end{align*}
where $\phi'$ is as defined in \eqref{eq:core-function} and $\hat p_{\bm Y,\bm Z}(\bm y,\bm z) =w({\bm y,\bm z})/(nh_1^{d_Y+d_Z})$ and $\hat p_{\bm Z} = w({\bm z})/(nh_2^{d_Z})$ are the kernel density estimators of $p_{\bm Y,\bm Z}$ and $p_{\bm Z}$, respectively (recall the discussion after \eqref{eq:kw}). 
This estimate can be expanded out as a $V$-statistic, which incurs some bias  in the estimation procedure, due to the repetition of the indices in the summation. To avoid this, for the asymptotic results it is convenient to work with the following estimator, which takes the form of an $U$-statistic, 
\begin{align}
\hat{\zeta}_{n}^{*} := \frac{1}{(n)_9  } \sum_{ \bm i \in ([n])_9} \varphi_{n}\big((\bm X_{i_1},\bm Y_{i_1},\bm Z_{i_1}),(\bm X_{i_8},\bm Y_{i_8},\bm Z_{i_8}),\ldots,(\bm X_{i_9},\bm Y_{i_9},\bm Z_{i_9})\big) , 
\label{eq:XYZestimateU}
\end{align}
where $(n)_9 = n(n-1)\cdots(n-8)$, $(n)_9$ is the set of all elements in $[n]^9$ with distinct indices, 
\begin{align}\label{eq:wXYZ}
    & \varphi_{n}\big((\bm X_1,\bm Y_1,\bm Z_1),\ldots,(\bm X_9,\bm Y_9,\bm Z_9)\big) \nonumber \\ 
    & = \frac{1}{h_{1}^{4(d_Y+d_Z)} h_{2}^{4d_Z}} \prod_{\ell=2}^5 w_{(\bm Y_{\ell},\bm Z_{\ell})}(\bm Y_1,\bm Z_1) \prod_{\ell=6}^9w_{\bm Z_{\ell}}(\bm Z_1) \phi' (\bm X_{2},\ldots, \bm X_{5}; \bm X_{6}, \ldots \bm X_{9}) , 
\end{align} 
and $\phi'$ is the ball divergence core function \eqref{eq:core-function}. The following theorem establishes the large sample properties of $\hat{\zeta}_n^{*}$.

\begin{thm} 
  Assume the conditions of Theorem~\ref{consistency} holds. Then, we have the following results: 
 \begin{itemize} 
 \item[\textnormal{(1)}] $\zeta^*(\bm X,\bm Y|\bm Z)\geq 0$ and the equality holds if and only if $\bm X\indpt \bm Y|\bm Z$.
 \item[\textnormal{(2)}] $\hat{\zeta}_{n}^{*}$ is a consistent estimator of $\zeta^*(\bm X,\bm Y|\bm Z)$, that is, 
 $\hat{\zeta}_{n}^{*} \stackrel{P}{\rightarrow} \zeta^*(\bm X,\bm Y|\bm Z)$, as $n\rightarrow\infty$.
\item[\textnormal{(3)}] Under $H_0: \bm X\indpt \bm Y| \bm Z$, if additionally we have $n^2h_{1}^{d_Y+d_Z+4}\rightarrow 0$ and $h_{2}/h_{1}\rightarrow 0$ as $n\rightarrow\infty$, then   
    \begin{align}\label{eq:distributionXYZ}
        nh_{1}^{(d_Y+d_Z)/2}\hat{\zeta}_{n}^{*} \stackrel{D} \rightarrow N(0, 144  {\mathcal C}(K) \sigma^2), 
    \end{align}
    where $\mathcal{C}(K) := \int \left(\int K(\|({\bm s, \bm t})\|)K(\|{(\bm u, \bm v) - (\bm s, \bm t)}\|) \mathrm d\bm s \mathrm d{\bm t} \right)^2 \mathrm d{\bm u}\mathrm d\bm v$, 
    \begin{align}\label{eq:varianceH0} 
    \sigma^2= \int Q_0^2({\bm x_1, \bm x_2;\bm y, \bm z}) p_{\bm X, \bm Y, \bm Z}({\bm x_1, \bm y, \bm z })p_{\bm X, \bm Y, \bm Z}({\bm x_2, \bm y, \bm z }) p_{\bm Y, \bm Z}^6(\bm y, \bm z) p_{\bm Z}^8(\bm z) \mathrm  d{\bm x_1} \mathrm d{\bm x_2} \mathrm d{\bm y} \mathrm d{\bm z}, 
    \end{align}  
    with $Q_0$ as defined in Theorem \ref{large-sam-dist-1}. 
    \end{itemize} 
    \label{large-sam-dist-2}
\end{thm}

The proof of Theorem \ref{large-sam-dist-2} is given in Appendix \ref{sec:asymptoticpf}. Note from \eqref{eq:XYZestimateU} that $\hat{\zeta}_{n}^{*}$ is a $U$-statistics of order 9 with core function $\varphi_n$. The proof of Theorem \ref{large-sam-dist-2} proceeds by writing out the Hoeffding's projections of $\hat{\zeta}_{n}^{*}$. Under $H_0$, the first-order projection turns out to be asymptotically negligible and the limiting distribution is determined by terms arising from the second-order projection, which we analyze using results from \cite{hall1984} on degenerate $U$-statistics of order 2. For large samples, by estimating $\sigma^2$, we can use \eqref{eq:distributionXYZ} to obtain a test for conditional independence with precise asymptotic level. However, this approximation might not be very accurate for small samples. Moreover, although the $U$-statistic based estimate \eqref{eq:XYZestimateU} has a tractable limiting distribution, its computational complexity is very high. On the other hand, the  estimator proposed in Section \ref{sec:estimation}, which can be expanded as a $V$-statistic, is computationally more efficient and it can be calibrated using a bootstrap or permutation method as discussed in Section \ref{sec:conditionalindependence}, leading to a test that can detect conditional dependence even in moderately large samples.

\begin{rem}
Another approach to reduce the computational complexity of the $U$-statistic based estimate \eqref{eq:XYZestimateU} is to consider the following estimator of $\zeta^*(\bm X,\bm Y|\bm Z)$ which averages $\varphi_{n}$ only over disjoint $9$-tuples:
\begin{align*}
\hat\zeta_n^L = \frac{1}{\lfloor n/9 \rfloor}\sum_{i=1}^{\lfloor n/9 \rfloor} \varphi_{n}\big((\bm X_{9i-8},\bm Y_{9i-8},\bm Z_{9i-8}),\ldots,(\bm X_{9i},\bm Y_{9i},\bm Z_{9i})\big). 
\end{align*}
Clearly, $\hat\zeta_n^L$ can be computed in $\mathcal{O}(n)$ time. Moreover, using similar arguments as in the proof of Theorem \ref{consistency} it can be shown that $\hat\zeta_n^L$ is a consistent estimator of $\zeta^*(\bm X,\bm Y|\bm Z)$ (under the same set of assumptions). However, the resulting test for conditional independence will be less powerful than the tests based on \eqref{eq:AB} or \eqref{eq:XYZestimateU}.  
\end{rem}

\section{ Limiting Distribution of Conditional $V$-Statistics}
\label{sec:empiricalprocess} 
As is evident from the discussion in the previous sections, to derive the asymptotic properties in the cBD estimate \eqref{eq:zetasum} one has to understand the limiting distribution of $U$/$V$-statistics that involve empirical estimates of the conditional distributions $\Tilde{\mathbb P}_{\bm X| \bm Z = \bm z}$ and $\Tilde{\mathbb P}_{\bm X| \bm Y = \bm y, \bm Z = \bm z} $. Similar quantities appear in estimation of other nonparametric measures for conditional inference, such as the conditional distance covariance \citep{wang2015}, the tests in \cite{su2007,ke2020expected}, the conditional energy distance \citep{chen2022paired,yan2022nonparametric}, among others. In this section we distill the main techniques from our proofs and present a general result about the limiting distribution of a 2-sample conditional $V$-statistic, which can be broadly applied for deriving large sample properties of kernel density based estimates for conditional inference.

We begin with a convergence result about a conditional empirical process. As before, we assume $(\bm X_1, \bm Y_1, \bm Z_1), \ldots, (\bm X_n, \bm Y_n, \bm Z_n)$ are i.i.d. samples with joint distribution $\P_{\bm X, \bm Y, \bm Z}$ with density $p_{\bm X, \bm Y, \bm Z}$. Fixing $\bm y \in \R^{d_Y}$ and $\bm z \in \R^{d_Z}$, let $\Tilde{\mathbb P}_{\bm X| \bm Z = \bm z}$ and $\Tilde{\mathbb P}_{\bm X| \bm Y = \bm y, \bm Z = \bm z} $ be the estimates of the conditional distributions $\mathbb P_{\bm X| \bm Z = \bm z}$ and $\mathbb P_{\bm X\mid \bm Y = \bm y, \bm Z = \bm z }$, respectively, as defined in Section \ref{sec:estimation}. For bandwidths $h_{1}$ and $h_{2}$ satisfying Assumption \ref{h}, consider the {\it conditional empirical processes}, 
\begin{align}
    S_n^{\bm z} := &  \Big\{\sqrt{nh_{2}^{ d_Z } }\left(\int f(\bm x) \mathrm d\Tilde{\mathbb{P}}_{\bm X| \bm Z = \bm z}( \mathrm d \bm x)-\E_{\mathbb{P}_{\bm X| \bm Z = \bm z}} [ f(\bm X) ] \right)\mid f\in \mathcal{F}_{\bm z}\Big\} , 
    \label{eq:conditional-emp-processz} \\
    S_n^{\bm y,\bm z} := &  \Big\{\sqrt{nh_{1}^{d_Y+d_Z} }\left(\int f(\bm x) \mathrm d\Tilde{\mathbb P}_{\bm X | \bm Y=\bm y, \bm Z = \bm z}( \mathrm d \bm x)-\E_{\mathbb{P}_{\bm X| \bm Y= \bm y, \bm Z = \bm z}} [ f(\bm X) ] \right)\mid f\in \mathcal{F}_{\bm y,\bm z}\Big\},
    \label{eq:conditional-emp-processyz}
\end{align}
where $\mathcal F_{\bm z} := \{f : \mathbb R^{d_X} \rightarrow \mathbb R | \int f (\bm x )^2 \mathrm d \mathbb P_{\bm X| \bm Z= z} (\bm x) < \infty \}$ and 
$\mathcal F_{\bm y, \bm z} := \{f : \mathbb R^{d_X} \rightarrow \mathbb R | \int f (\bm x )^2 \mathrm d \mathbb P_{\bm X| \bm Y = \bm y, \bm Z= z} (\bm x) < \infty \}$. The following proposition establishes the finite-dimensional convergence of these processes.

\begin{prop}
Suppose Assumptions \ref{YZ}, \ref{K}, and \ref{h} hold. Fix $\bm y \in \R^{d_Y}$ and $\bm z \in \R^{d_Z}$ such that $p_{\bm Z}({\bm z})>0$ and $p_{\bm Y,\bm Z}({\bm y,\bm z})>0$. Then the finite dimensional distributions of the processes
$(S_n^{\bm z}, S_n^{\bm y, \bm z})$ jointly converge to the finite dimensional distributions of 
$$\left( \left\{\int f(\bm x) \mathrm d\mathbb{G}_{\mathbb{P}_{\bm X| \bm Z = \bm z}}(\bm x)\mid f\in \mathcal{F}_{\bm z} \right\} , \left\{\int f(\bm x) \mathrm d\mathbb{G}_{\mathbb{P}_{\bm X| \bm Y= \bm y, \bm Z = \bm z}}(\bm x)\mid f\in \mathcal{F}_{\bm y, \bm z} \right\} \right) , 
$$
where $\mathbb{G}_{\mathbb{P}_{\bm X| \bm Z = \bm z}}$ and $\mathbb{G}_{\mathbb{P}_{\bm X| \bm Y= \bm y, \bm Z = \bm z}}$ are independent Gaussian processes defined as follows: 

\begin{itemize}

\item $\mathbb{G}_{\mathbb{P}_{\bm X| \bm Z = \bm z}}$ is a centered Gaussian process with covariance kernel, for $f, g \in \mathcal F_{\bm z}$, 
$$\mathrm{Cov}\left[\int f \mathrm d\mathbb{G}_{\mathbb{P}_{\bm X| \bm Z = \bm z}},\int g \mathrm d\mathbb{G}_{\mathbb{P}_{\bm X| \bm Z = \bm z}} \right] = c_{K} (\bm z) \mathrm{Cov}_{\mathbb{P}_{\bm X| \bm Z = \bm z}}\big[ f({\bm X}), g({\bm X})\big] , $$  
and $c_{K}(\bm z) := \frac{\int_{\R^{d_Z}} K^2(\| \bm u\|) \mathrm d \bm u}{p_{\bm Z}(\bm z)}$.  

\item $\mathbb{G}_{\mathbb{P}_{\bm X| \bm Y= \bm y, \bm Z = \bm z}}$ is a centered Gaussian process with covariance kernel, for $f, g \in \mathcal F_{\bm y, \bm z}$, 
\begin{align}\label{eq:covarianceyz}
& \mathrm{Cov}\left[\int f \mathrm d\mathbb{G}_{\mathbb{P}_{\bm X| \bm Y= \bm y, \bm Z = \bm z}},\int g \mathrm d\mathbb{G}_{\mathbb{P}_{\bm X| \bm Y= \bm y, \bm Z = \bm z}} \right] & = c_{K} (\bm y, \bm z) \mathrm{Cov}_{\mathbb{P}_{\bm X| \bm Y= \bm y, \bm Z = \bm z}}\big[ f({\bm X}), g({\bm X})\big] , 
\end{align}
where $c_{K}(\bm y,\bm z) := \frac{\int_{\R^{d_Y+d_Z}} K^2(\| \bm u\|) \mathrm d \bm u}{p_{\bm Y,\bm Z}(\bm y,\bm z)}$. 
\end{itemize}
\label{prop:main-1}   
\end{prop}

The proof of Proposition \ref{prop:main-1} is given in Appendix \ref{sec:empiricalprocesspf}. 
Invoking this result we can derive the asymptotic properties of functionals of the conditional empirical process.  Specifically, given a core function $\phi:(\R^{d_X})^{r_{1}}\times (\R^{d_X})^{r_{2}}\rightarrow \R$, for integers $r_{1}, r_{2} \geq 0$, consider the problem of estimating 
\begin{align}\label{eq:theta}
\theta := \mathbb E[\phi( \underline{\bm X}_{[r_{1}+r_{2}]}) \mid \{ \bm Y_s = \bm y \}_{ 1  \leq s \leq r_{1} }, \{ \bm Z_s = \bm z \}_{1 \leq s \leq r_{1}+r_{2}}  ] , 
\end{align} 
where $\underline{\bm X}_{[r_{1}+r_{2}]} = ( \bm X_1, \ldots, \bm X_{r_{1}+r_{2}})$. 
     The plug-in empirical estimator of $\theta$ is given by: 
     \begin{align}\label{eq:thetaestimate}
         \hat\theta_n & = \int \phi(x_1,\ldots, x_{r_{1}} ;  x_{r_{1}+1},\ldots, x_{r_{1}+r_{2}}) \prod_{i=1}^{r_{1}}  \mathrm d \tilde\P_{\bm X\mid \bm Y = \bm y,\bm Z = \bm z} (\bm x_i) \prod_{j=r_{1}+1}^{r_{1}+r_{2}} \mathrm d \tilde\P_{\bm X| \bm Z=\bm z } (\bm x_j) . 
     \end{align}
     This is a 2-sample conditional $V$-statistic with a core function of order $(r_{1}, r_{2})$. To derive the asymptotic distribution of $\hat \theta$ we need to compute its Hoeffding's decomposition. For this, denote the symmetrization $\phi$ by: 
    \begin{align}
        \phi'(\bm x_1,\ldots, \bm x_{r_{1}+r_{2}}) = \frac{1}{r_{1}! r_{2}!}\sum_{ \tau \in S_{r_{1}}} \sum_{ \sigma \in S_{r_{2}}} \phi( \underline{\bm x}_{\tau}; \underline{\bm x}_{ r_{1}+ \sigma}) , 
        \label{eq:symver-2} 
    \end{align} 
    where, for any integer $r \geq 1$, $S_r$ is the set of all permutations of $\{1, 2, \ldots, r\}$, $\underline{\bm x}_{\tau} = ( \bm x_{\tau(1)}, \ldots \bm x_{\tau(r_{1})} )$, for $\tau \in S_{r_{1}}$, 
    and $\underline{\bm x}_{r_{1}+ \sigma} = ( \bm x_{r_{1}+\sigma(1)}, \ldots \bm x_{r_{1}+ \sigma(r_{2})} )$, for $\sigma \in S_{r_{2}}$. (Note that if $\phi$ is symmetric in the first $r_{1}$ and last $r_{2}$ coordinates, then $\phi' = \phi$.)  

\begin{definition} For $0 \leq i \leq r_{1}$ and $0 \leq j \leq r_{2}$, the $(i,j)$-th order (conditional) Hoeffding's projection of $\phi$  is defined as: 
\begin{align}\label{eq:xrab}
& \phi_{i, j} (\bm x_1, \ldots, \bm x_{i}; \bm x_{1}', \ldots, \bm x_{j}') \nonumber \\ 
& := \E[\phi'(\bm x_1, \ldots, \bm x_{i}, \bm X_{i+1}, \ldots, \bm X_{r_{1}}; \bm x_{1}', \ldots, \bm x_{j}', \bm X_{r_{1}+j+1}, \ldots \bm X_{r_{1}+r_{2}})| \mathcal G_{\bm y, \bm z}^{(i, j)}] ,
\end{align}
where $\mathcal G_{\bm y, \bm z}^{(i, j)} := \{ \{\bm Y_{s} = \bm y, \bm Z_{s} = \bm z \}_{i+1 \leq s \leq r_{1}}, \{\bm Z_{s} = \bm z \}_{ r_{1}+j+1 \leq s \leq r_{1} + r_{2}}  \}$ and $\phi'$ is the symmetrization of $\phi$. 

\begin{itemize} 

\item The function $\phi$ is said to be {\it conditionally non-degenerate}, if either $\phi_{1, 0} \ne 0 $ or $\phi_{1, 0} \ne 0 $, almost surely. 

\item For $k \geq 1$, the function $\phi$ is said to be {\it conditionally degenerate of order $k$}, if $\phi_{i, j} = 0 $, almost surely, for all $a+b \leq k$ and there exists $a_0, b_0$ such that $a_0+b_0 = k+1$ and $\phi_{a_0, b_0}\not=0$. 
 
\item The function $\phi$ is said to be {\it conditionally completely degenerate}, if $\phi_{i, j} = 0 $, almost surely, for all $a+b < r_{1}+r_{2}$. 

\end{itemize} 
\label{def:Vphi}
\end{definition}

The following result gives the asymptotic distribution of $\hat \theta$, depending on the order of degeneracy of $\phi$. Analogous results for one-sample conditional  $U$-statistics are derived in \cite{prakasa1995limit,stute1991conditional}. 

\begin{prop} Let $\theta$ and $\hat \theta$ be as defined in \eqref{eq:theta} and \eqref{eq:thetaestimate}. Further, suppose 
\begin{align}\label{eq:phiyz}
\mathbb E[\phi^2( \underline{\bm X}_{[r_{1}+r_{2}]}) \mid \{ \bm Y_s = \bm y \}_{ 1  \leq s \leq r_{1} }, \{ \bm Z_s = \bm z \}_{1 \leq s \leq r_{1}+r_{2}}  ] < \infty,
\end{align}
and Assumptions \ref{YZ}, \ref{K}, \ref{h} hold. Then we have the following: 

\begin{itemize}

\item[$(1)$] $\hat\theta_n = \theta + R_n$, where $R_n=O_P((nh_{1}^{d_Y+d_Z})^{-1/2})$. 
          
          \item[$(2)$] If $\phi$ is conditionally non-degenerate, then 
              \begin{align*}
          (nh_{1}^{d_Y+d_Z})^{1/2} (\hat{\theta}_n - \theta ) \stackrel{D}  \rightarrow  r_{1} \int \phi_{1, 0}(\bm x)  \mathrm d \mathbb{G}_{\mathbb{P}_{\bm X| \bm Y= \bm y, \bm Z = \bm z}}(\bm x)  \stackrel{D} = N(0, r_{1}^2 \sigma^2) ,  
     \end{align*} 
     where $\sigma^2 = c_{K} (\bm y, \bm z) \mathrm{Var}_{\mathbb{P}_{\bm X| \bm Y= \bm y, \bm Z = \bm z}} [ \phi_{1, 0}(\bm X)]$  and $c_{K}(\bm y,\bm z)$ is defined in Proposition \ref{prop:main-1}.

          \item[$(3)$] If $\phi$ conditionally degenerate of order $k$, for some $k \geq 1$, then 
     \begin{align*}
          (nh_{1}^{d_Y+d_Z})^{(k+1)/2} \hat{\theta}_n    \stackrel{D}  \rightarrow  {r_{1}\choose k+1}  \int \phi_{k+1, 0}({\bm x_1\ldots, \bm x_{k+1}})\prod_{i=1}^{k+1} \mathbb{G}_{\mathbb{P}_{\bm X| \bm Y= \bm y, \bm Z = \bm z}}(\bm x_i) , 
     \end{align*} 
          where $\phi_{k+1, 0}$ is the $(k+1, 0)$-th order Hoeffding's projection of $\phi$. 
    \end{itemize}

    \label{prop:main-8}
\end{prop}

The proof of Proposition \ref{prop:main-8} is given in Appendix \ref{sec:empiricalprocesspf}. Proposition \ref{prop:main-8}, in particular, gives us the asymptotic distribution of pointwise ball divergence $\Theta^2(\Tilde{\mathbb P}_{\bm X\mid  \bm Y = \bm y, \bm Z = \bm z}, \Tilde{\mathbb P}_{\bm X | \bm Z= \bm z})$, both under the null and the alternative, as in Theorem \ref{large-sam-dist-1}. Moreover, as mentioned before, we expect this result to be useful in obtaining the pointwise limiting distribution of other nonparametric conditional independence or 2-sample estimates. We illustrate this in the following remark for the conditional distance covariance \cite{wang2015}. 

\begin{rem}
\label{remark:cdc-limit-dist}
    The conditional distance covariance between 2 random vectors $\bm X$ and $\bm Y$ given another random vector $\bm Z = \bm z$ is defined as the square root of (see \cite[Definition 1]{wang2015}): 
    \begin{align}
        \mathcal{D}(\bm X,\bm Y|\bm Z=\bm z) = \frac{1}{c_{d_X} c_{d_Y}} \int_{\mathbb{R}^{d_X + d_Y}} \frac{|\phi_{\bm X,\bm Y |\bm Z = \bm z}(\bm t,\bm s) - \phi_{\bm X|\bm Z = \bm z}(\bm t)\phi_{\bm Y|\bm Z = \bm z}(\bm s)|^2}{\|\bm t\|^{p+1}\|\bm s\|^{q+1}} \mathrm d\bm t \mathrm d \bm s , 
        \label{eq:cdc-population}
    \end{align} 
    where $\phi_{\bm X | \bm Z = \bm z } (\cdot)$, $\phi_{ \bm Y| \bm Z = \bm z } (\cdot)$, and $\phi_{\bm X, \bm Y| \bm Z = \bm z } (\cdot)$,  
  are the conditional characteristic functions of $\bm X| \bm Z= \bm z$, $\bm Y| \bm Z= \bm z$, and $(\bm X, \bm Y)| \bm Z= \bm z$, respectively; and $c_{d_X} = \frac{\pi^{d_X/2}}{\Gamma((d_X+1)/2)}$ and $c_{d_Y} = \frac{\pi^{d_Y/2}}{\Gamma((d_Y+1)/2)}$. Following the derivation of \citet{szekely2007measuring}, \eqref{eq:cdc-population} can be alternatively expressed as:
    \begin{align*} 
        \mathcal{D}(\bm X,\bm Y|\bm Z=\bm z) = \E\left[\varphi\big((\bm X_1,\bm Y_1),(\bm X_2,\bm Y_2),(\bm X_3,\bm Y_3),(\bm X_4,\bm Y_4)\big)\Big|\{\bm Z_i = \bm z\}_{1\leq i\leq 4}\right] , 
    \end{align*} 
    where $(\bm X_1,\bm Y_1, \bm Z_1),(\bm X_2,\bm Y_2, \bm Z_2),(\bm X_3,\bm Y_3, \bm Z_3)$ are i.i.d. $\P_{\bm X,\bm Y, \bm Z}$ and 
   \begin{align*}
        \varphi\big((\bm X_i,\bm Y_i)_{1\leq i\leq 4}\big) = \|\bm X_1-\bm X_2\|\|\bm Y_1-\bm Y_2\| + \|\bm X_1-\bm X_2\| \|\bm Y_3-\bm Y_4\| - 2\|\bm X_1-\bm X_2\|\|\bm Y_1-\bm Y_3\| . 
    \end{align*}
Hence, the natural plug-in estimate of \eqref{eq:cdc-population} is: 
  \begin{align}\label{eq:thetaestimate}
          \mathcal{D}_n^{\bm z}& = \int \varphi((\bm x_i,\bm y_i)_{1\leq i\leq 4}) \prod_{i=1}^{4}  \mathrm d \tilde\P_{\bm X, \bm Y | \bm Z = \bm z} (\bm x_i, \bm y_i)  . 
     \end{align} 
It is known that $\mathcal{D}_n^{\bm z}$ is a consistent estimator of $\mathcal{D}(\bm X,\bm Y|\bm Z=\bm z)$ (see \cite[Theorem 4]{wang2015}), however, to the best of our knowledge, the asymptotic distribution of $\mathcal{D}_n^{\bm z}$ has not been previously explored. This can be addressed using Proposition \ref{prop:main-8}.  
To  this end, note that the symmetrization $\varphi^\prime$ of $\varphi$ (as in \eqref{eq:symver-2} with $r_1=4$ and $r_2=0$) is first-order degenerate under $H_0$ and is non-degenerate under $H_1$. Hence, using Proposition \ref{prop:main-8} we have the following results. For this, we assume 
    \begin{align*}
        \mathbb E[\varphi^2\big((\bm X_i,\bm Y_i)_{1\leq i\leq 4}\big) \mid \{ \bm Z_s = \bm z \}_{1 \leq s \leq 4}  ] < \infty 
    \end{align*}
    and $\bm z \in \mathbb R^{d_Z}$ is such that $p_{\bm Z}(\bm z) > 0 $.  
    
    \begin{itemize}
        \item Under $H_0:\bm X\indpt \bm Y\mid \bm Z = \bm z$ we have,
        $$nh_1^{d_Y+d_Z}\mathcal{D}_n^{\bm z} \stackrel{D}{\rightarrow} {4\choose 2} \int \varphi_{2}((\bm x_1, \bm y_1), (\bm x_2, \bm y_2)) \prod_{i=1}^2 \mathrm d\mathbb G_{\P_{ \bm X,\bm Y | \bm Z=\bm z}} (\bm x_i, \bm y_i) \stackrel{D}{=}6\sum_{i=1}^\infty \lambda_i U_i^2,$$
        where $\varphi_2$ is the second-order Hoeffding's projection of $\varphi^\prime$ with respect to the distribution $\P_{\bm X,\bm Y|\bm Z=\bm z}$ as in Definition \ref{def:Vphi}, $\{\lambda_i\}_{i \geq 1}$ is a square integrable sequence, and $\{U_i\}_{i \geq 1}$ is a sequence of i.i.d standard normal random variables. The distributional equality of the limiting random variable follows from spectral decomposition theorem as in Remark \ref{remark:distributionH0H1}.  

        \item Under $H_1:\bm X\indpt\bm Y| \bm Z = \bm z$ we have,
        $$\sqrt{nh_1^{d_Y+d_Z}}\big(\mathcal{D}_n^{\bm z} - \mathcal{D}(\bm X,\bm Y|\bm Z=\bm z)\big)\stackrel{D}{\rightarrow} 4 \int \varphi_1(\bm x , \bm y )\mathrm d \mathbb G_{\P_{\bm X,\bm Y|\bm Z=\bm z}}(\bm x , \bm y )\stackrel{D}{=}\mathcal{N}(0,16\sigma^2),$$
        where $\varphi_1$ is the first-order Hoeffding's projection of $\varphi^\prime$ with respect to the distribution $\P_{\bm X,\bm Y|\bm Z=\bm z}$ and $\sigma^2= c_K(\bm z) \mathrm{Var}_{\P_{\bm X,\bm Y|\bm Z=\bm z}}[\varphi_1( \bm X , \bm Y )]$.   
    \end{itemize}
\end{rem}

\section{Testing Conditional Independence} 
\label{sec:conditionalindependence}

In this section, we will discuss methods to calibrate the cBD estimate for testing conditional independence. Throughout, we fix a weight function $a$ and, suppressing the dependence on $a$,  denote the corresponding cBD measure by $\zeta(\bm X, \bm Y|\bm Z)$ and the estimate by $\hat\zeta_{n}$. 

\subsection{Tests under the Model-\textbf{\textit{X}} framework}
\label{model-X}

In the model-$\bm{X}$ framework \citep{candes2018panning} since the distribution of ${\bm X| \bm Z}$ is assumed to be known, one can implement the conditional randomization test (CRT) by repeatedly resampling from the distribution ${\bm X| \bm Z}$ to approximate the null distribution of the cBD estimate $\hat\zeta_{n}$. To describe this formally, denote the observed sample by $\mathcal{D}_0 := \{ (\bm{X}_i, \bm {Y}_i, \bm{Z}_i)\}_{i=1}^n$.   The CRT calibration of the cBD estimate then proceeds as follows: 
\begin{itemize}

\item Generate new observations ${\bm X_1', \bm X_2',\ldots, \bm X_n'}$, where ${\bm X_i^{\prime}}\sim {\bm X| \bm Z_i}$, for $1 \leq i \leq n$, and construct a new data set $\mathcal{D}' = \{ (\bm{X}_i^{\prime},\bm  Y_i, \bm Z_i)\}_{i=1}^n$. Repeat this $M$ times to obtain resampled data sets $\mathcal{D}'_1,\mathcal{D}'_2\ldots,\mathcal{D}'_M$.  

\item Report the $p$-value 
$$p_{CRT} = \frac{1+\sum_{j=1}^M \bm{1}\{\hat\zeta_{n}(\mathcal{D}_j')\geq \hat\zeta_{n}(\mathcal{D}_0)\}}{1+M},$$
where $\hat\zeta_{n}(\mathcal{D}_j')$ denotes the test statistic computed on the $j$-th resampled data set $\mathcal{D}_j'$, for  $1 \leq j \leq M$. 
\end{itemize}
This procedure controls Type I error in finite samples and is also asymptotically consistent with a finite number of resamples as shown below (see Appendix \ref{sec:rtestpf} for the proof). Throughout, we fix $\alpha \in (0, 1)$. 

\begin{prop}
Suppose the assumptions of Theorem \ref{consistency} hold. Then the test function $\phi_{CRT} = \bm{1}\{p_{CRT}<\alpha\}$, controls Type I error in finite samples, that is, $\P_{H_0}(\phi_{CRT}) \leq \alpha$. Moreover, $\lim_{n \rightarrow \infty} \P_{H_1}(\phi_{CRT}) =1$,  for $M>(1-\alpha)/\alpha$. 
\label{CRT-test}
\end{prop}

We can also calibrate the test using the conditional permutation test \citep{berrett2019}, a variant of the CRT where one permutes the observations $\bm X_1, \bm X_2, \ldots, \bm X_n$ non-uniformly in such a way that the dependence between $\bm X$ and $\bm Z$ is preserved. Specifically, we generate $M$ permutations $\pi_1,\pi_2,\ldots,\pi_M$ of $\{1,2,\ldots,n\}$ independently from the distribution: 
$$\P(\Pi = \pi|{\bm X, \bm Y, \bm Z}) = \frac{\prod_{\ell = 1}^n p({\bm X_{\pi(\ell)}| \bm Z_{\ell}})}{\sum_{\pi'\in S_n}\prod_{\ell=1}^n p({\bm X_{\pi'(\ell)}| \bm Z_{\ell}})} , $$
where $S_n$ is the set of all permutations of $\{1,2,\ldots,n\}$, and $p({\bm x| \bm z})$ is the density function of the distribution of ${\bm X| \bm Z=\bm z}$. Then we construct resampled data sets $\mathcal{D}^{\pi_j} = \{{(\bm X_{\pi_j(i)},\bm  Y_i, \bm Z_i)}\}_{i=1}^n$, for each $1 \leq j \leq M$ and compute the $p$-value: 
$$p_{CPT} = \frac{1+\sum_{j=1}^M \bm{1}\{ \hat\zeta_{n}(\mathcal{D}^{\pi_j})\geq \hat\zeta_{n}(\mathcal{D}_0)\}}{1+M}. $$
Here, as before, $\mathcal D_0 := \{ (\bm{X}_i, \bm {Y}_i, \bm{Z}_i)\}_{i=1}^n$ denotes the original data set and $\hat\zeta_{n}(\mathcal{D}^{\pi_j})$ is the test statistic computed on $\mathcal{D}^{\pi_j}$, for $1 \leq j \leq M$. 
This method also has the desired level and power properties as stated in the following proposition (see Appendix \ref{sec:permutationpf} for the proof).  

\begin{prop}
Suppose the assumptions of Theorem \ref{consistency} hold. Then the test function $\phi_{CPT} = \bm{1}\{p_{CPT}<\alpha\}$, controls Type I error in finite samples, that is, $\P_{H_0}(\phi_{CPT}) \leq \alpha$. Moreover, $\lim_{n \rightarrow \infty} \P_{H_1}(\phi_{CPT}) =1$,  for $M>(1-\alpha)/\alpha$. 
\label{CPT-test}
\end{prop}

\begin{rem}[Understanding model misspecification] 
\label{remark:CRT-MS}
In the model-${\bm X}$ framework the knowledge of ${\bm X| \bm Z}$ plays a crucial role in controlling the Type I and Type II error rates. Naturally, misspecification of the distribution of ${\bm X| \bm Z}$ may lead to incorrect inferences. To investigate the effect of model misspecification on the CRT, suppose one resamples from the model $\P^*_{\bm X| \bm Z}$, where $\P_{\bm X| \bm Z}^* \not= \P_{\bm X| \bm Z}$. Note that even if the model is misspecified, the CRT will have the desired error control properties if the distribution of $\hat\zeta_{n}(\mathcal{D}')$ matches with the null distribution of $\hat\zeta_{n}(\mathcal{D}_0)$. Towards this, note that the function $\phi'$ in \eqref{eq:core-function} depends only on the ordering of the pairwise distances of the observations $\{ \bm X_1, \bm X_2, \ldots, \bm X_n\}$. Hence, the resampled test statistic and the observed test statistic will have the same distribution if $\P_{ \bm {X} | \bm Z}$ can be obtained from $\P_{ \bm {X} | \bm Z}^*$  using a distance preserving transformation, that is, a translation, rotation or homogeneous scale transformation (the same change in scale along all coordinate axes). This is because the joint distribution of $\bm 1\left\{\|\bm X_i-\bm X_j\|\leq \|\bm X_k-\bm X_j\|\right\}$, for $1\leq i< j< k\leq n$, $w_{(\bm Y_i,\bm Z_i)}(\bm Y_j, \bm Z_j)$, for $1\leq i< j\leq n$, and $w_{\bm Z_i}(\bm Z_j)$, for $1\leq i< j\leq n$,  will remain unchanged under such transformations. To illustrate this consider the following numerical example: Suppose 
$$Z\sim N(0,1), ~ Y = Z + \varepsilon_1, \text{ and } X = Z + rY+ \varepsilon_2,$$ 
where $\varepsilon_1,\varepsilon_2$ are i.i.d $N(0,1)$ random variables. It is easy to see that when $r=0$, $X\indpt Y\mid Z$ and when $r\not=0$, $X \not\indpt Y\mid Z$. Here, we consider the cases $r=0$ and $r=1$. Note that for $r=0$, $X\mid Z\sim Z + N(0,1)$ and for $r=1$, $X\mid Z\sim 2Z+N(0,2)$. 
To compute our test statistic we set the bandwidths as: $h_{1} = c_1 n^{-1/(d_Y+d_Z+2)}$ and $h_{2} = c_2n^{-1/(d_Z+2)}$, where $c_1$ and $c_2$ are the averages of the coordinate-wise IQR of the observations $\{ (Y_1, Z_1), \ldots, (Y_n, Z_n) \}$ and $\{ Z_1,  \ldots, Z_n\}$, respectively. Throughout the nominal level is set to $\alpha= 0.05$. 

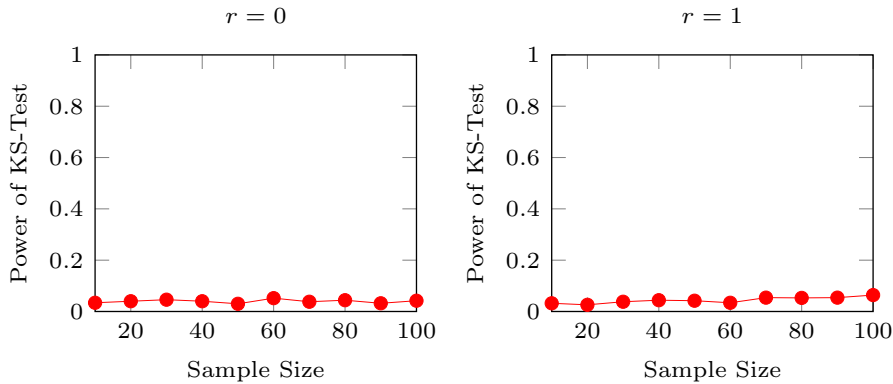
\begin{figure}[h]
\centering
\begin{tikzpicture}[scale = 1.25,font=\scriptsize]
\begin{axis}[xmin = 10, xmax = 100, ymin = 0, ymax = 1, xlabel = {Sample Size}, ylabel = {Power of KS-Test}, title = {$r = 0$}]
\addplot[color = red,   mark = *, step = 1cm,very thin]coordinates{(10,0.034)(20,0.04)(30,0.046)(40,0.04)(50,0.03)(60,0.052)(70,0.038)(80,0.044)(90,0.032)(100,0.042)
};

\end{axis}
\end{tikzpicture}
\begin{tikzpicture}[scale = 1.25,font=\scriptsize]
\begin{axis}[xmin = 10, xmax = 100, ymin = 0, ymax = 1, xlabel = {Sample Size}, ylabel = {Power of KS-Test}, title = {$r = 1$},]
\addplot[color = red,   mark = *, step = 1cm,very thin]coordinates{(10,0.032)(20,0.026)(30,0.038)(40,0.044)(50,0.042)(60,0.034)(70,0.054)(80,0.053)(90,0.054)(100,0.064)
};

\end{axis}
\end{tikzpicture}
\caption{ Power of the KS test for checking the difference between $\hat{\zeta}_n(\mathcal D_0)$ and $\hat{\zeta}_n(\mathcal D')$ in Example \ref{example:normalXZ} for $r=0$ and $r=1$. } 
    \label{fig:MX-1}
\end{figure}

\begin{example}\label{example:normalXZ} Here, we consider $\P_{ X\mid Z}^* \sim 5Z + \eta$, where $\eta\sim N\big(10,25/(r+1)\big)$, to generate the resampled data when $r=0$ and $r=1$. Note that this is a scale and location shift of the actual distribution $\P_{ X\mid Z}$, hence the distributions of $\hat{\zeta}_n(\mathcal D_0)$ and $\hat{\zeta}_n(\mathcal D')$ should be the same. To check this we use the Kolmogorov-Smirnov (KS) test to evaluate the closeness of the distributions of $\hat{\zeta}_n(\mathcal D_0)$ and $\hat{\zeta}_n(\mathcal D')$. The empirical power of the KS-test (computed based on $500$ trials) for different sample sizes is reported in Figure \ref{fig:MX-1}. The power was close to the size $0.05$ for all sample size for both $r=0$ and $r=1$, confirming that the distributions of $\hat{\zeta}_n(\mathcal D_0)$ and $\hat{\zeta}_n(\mathcal D')$ are indeed the same in the case, despite model misspecification. 
\end{example}

\begin{example}\label{example:sampleXZ} 
In this case we consider $\P^*_{ X \mid Z} =\text{Unif}(-|Z|,|Z|)$ to generate the resampled data. As before the empirical power of the KS-test (computed based on $500$ trials) for different sample sizes is reported in Figure \ref{fig:MX-2}. The plots clearly shows that here the resampled test statistic had a significantly different distribution than that obtained from correctly specified $\P_{ X\mid Z}$. 
\end{example}

\label{model_X}
\end{rem}

\begin{figure}[h]
\centering
\begin{tikzpicture}[scale = 1.25,font=\scriptsize]
\begin{axis}[xmin = 10, xmax = 100, ymin = 0, ymax = 1, xlabel = {Sample Size}, ylabel = {Power of KS-Test}, title = {$r = 0$}]
\addplot[color = red,   mark = *, step = 1cm,very thin]coordinates{(10,0.34)(20,0.254)(30,0.234)(40,0.256)(50,0.232)(60,0.248)(70,0.25)(80,0.29)(90,0.284)(100,0.322)
};

\end{axis}
\end{tikzpicture}
\begin{tikzpicture}[scale = 1.25,font=\scriptsize]
\begin{axis}[xmin = 10, xmax = 100, ymin = 0, ymax = 1, xlabel = {Sample Size}, ylabel = {Power of KS-Test}, title = {$r = 1$}]
\addplot[color = red,   mark = *, step = 1cm,very thin]coordinates{(10,0.372)(20,0.264)(30,0.236)(40,0.24)(50,0.24)(60,0.238)(70,0.264)(80,0.304)(90,0.304)(100,0.338)
};

\end{axis}
\end{tikzpicture}
\caption{ Power of the KS test for checking the difference between $\hat{\zeta}_n(\mathcal D_0)$ and $\hat{\zeta}_n(\mathcal D')$ in Example \ref{example:sampleXZ} for $r=0$ and $r=1$. }
    \label{fig:MX-2}
\end{figure}
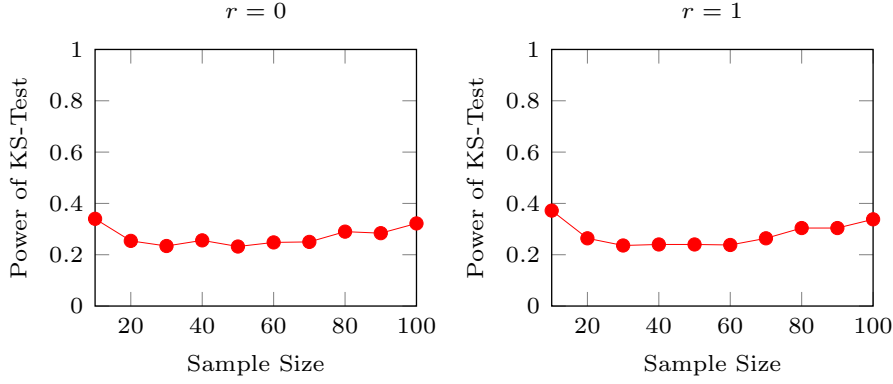

The above discussion shows that although the CRT with the cBD statistic is robust to distance preserving transformations, if the model is completely misspecified the CRT may lead to incorrect conclusions. In such cases one has to rely on a data-driven approach for approximating $\P_{ X\mid Z}$. This is discussed in the following section.

\subsection{Local Wild Bootstrap Tests}
\label{LBT}

In this section we propose a new bootstrap algorithm that approximates the density of $\P_{\bm X| \bm Z}$ using kernel smoothing methods and generates the resampled data from this estimated distribution. 
Towards this, fix bandwidths $h_0$ and $h_2'$ and consider the following estimate of the probability density function of $\bm X|\bm Z$, based on the data $\{ (\bm X_i, \bm Z_i) \}_{1 \leq i \leq n}$: 
\begin{align}\label{eq:estimatepXZ}
\hat p_{\bm X| \bm Z=\bm z}({\bm x| \bm z}) = \frac{\hat p_{\bm X,\bm Z}({\bm x,\bm z})}{\hat p_{\bm Z}({\bm z})} & = \frac{\frac{1}{nh_0^{d_X}h_{2}^{\prime d_Z}}\sum_{s=1}^n K\left(\frac{ \| \bm{z} - \bm{Z}_s \| }{h_{2}'}\right) \phi\left(\frac{\bm x-\bm X_s}{h_0}\right)}{\frac{1}{nh_{2}^{\prime d_Z}}\sum_{s=1}^n K\left(\frac{ \| \bm{z} - \bm{Z}_s \| }{h_{2}'}\right)} \nonumber \\ 
& = \sum_{s=1}^n \kappa_{s}({\bm z}) \frac{1}{h_0^{d_X}} \phi\left(\frac{\bm x-\bm X_s}{h_0}\right) , 
\end{align}
where $\phi$ denotes the probability density function of the standard normal distribution and 
\begin{align}\label{eq:kernelweights}
\kappa_{s}({\bm z}) := \frac{ K\left(\frac{ \| \bm{z} - \bm{Z}_s \| }{h_{2}'}\right)}{\sum_{j=1}^{n} K\left(\frac{\bm z-\bm Z_ j}{h_{2}'}\right) } , 
\end{align}
for $1 \leq s \leq n$. Let $\hat \P_{\bm X| \bm Z=\bm z}$ be the probability distribution function corresponding to $\hat p_{\bm X| \bm Z=\bm z}({\bm x\mid\bm z})$. Note that \eqref{eq:estimatepXZ} takes the form of a Gaussian mixture model. Specifically, the law of a random variable $\hat {\bm X} \sim \hat \P_{\bm X| \bm Z=\bm z}$ can be described as follows:  $\hat {\bm X}$ is distributed as $\mathcal N(\bm X_s, h_0^2 \bm I_{d_X})$ with probability $\kappa_s(\bm z)$, for $1 \leq s \leq n$. Now, the \textit{local wild bootstrap} method for calibrating the cBD statistic can be described as follows:

\begin{itemize}

\item Given the observed data $\mathcal{D}_0$, generate $\hat {\bm X}_i\sim \hat{\P}_{\bm X| \bm Z_i}$, independently for each $ 1 \leq i  \leq n$,  to get the resampled data set $\hat{\mathcal{D}} = \{({ \hat{\bm X}_i,\bm  Y_i, \bm Z_i})\}_{i=1}^n$. Repeat this procedure $M$ times to obtain resampled data sets $\hat{ \mathcal{D}}_1,\hat{\mathcal{D}}_2,\ldots,\hat{\mathcal{D}}_M$.

\item Report the $p$-value 
$$p_{LWB} = \frac{1+\sum_{j=1}^M \bm{1}\{ \hat\zeta_{n}(\hat{\mathcal{D}}_j)\geq \hat\zeta_{n}(\mathcal{D}_0)\}}{1+M},$$
where $\hat\zeta_{n}(\hat{\mathcal{D}}_j)$ is our test statistic computed based on the data $\hat{\mathcal{D}}_j$, for $1 \leq j \leq n$.  
\end{itemize}

Note that although the random variables $\{\hat\zeta_{n}(\hat{\mathcal{D}}_j)\}_{1 \leq j \leq M}$ are exchangeable, they do not have the same distribution as $\hat\zeta_{n}(\hat{\mathcal{D}}_0)$ under $H_0$. However, we can choose the bandwidths $h_0$ and $h_{2}'$ in a such a way that the distributions of $\hat{\zeta}_{n}(\hat{\mathcal{D}}_1)$ and $\hat{\zeta}_{n}(\mathcal{D}_0)$ become asymptotically close (in the total variation distance). As a consequence, the level of the local wild bootstrap test can be controlled at a desired level asymptotically and the test is also consistent in large samples.  This is formalized in the following theorem. 

Before proceeding with statement of the theorem, we consider a slight generalization. Note that one can consider implementing the local wild bootstrap method with any collection of (random) weight functions $\{\gamma_{s}\}_{1 \leq s \leq n}$ where $\gamma_{s}: \mathbb R^{d_Z} \rightarrow [0, 1]$ is such that $\sum_{s=1}^n \gamma_{s}(\bm z) = 1$, for all $\bm z \in \mathbb R^{d_Z}$ and $1 \leq s \leq n$, and the weight functions only depend on the data through $\{\bm Z_i\}_{1 \leq i \leq n}$. (Choosing $\gamma_s = \kappa_s$ as in \eqref{eq:kernelweights} gives the method described above.) Of course, not all weight functions $\gamma_s$ are meaningful, only those which `concentrate' around the data point $\bm Z_s$ and decay with distance from $\bm Z_s$ are statistically relevant. In the following theorem we show the asymptotic validity of the local wild bootstrap for weight functions satisfying a general condition. Implications of this condition are discussed after the theorem.

\begin{thm}
    Suppose Assumptions \ref{YZ}, \ref{K}, \ref{h} hold, $nh_0^2 \rightarrow 0$, 
    and $p_{\bm X,\bm Z}$ is nice with respect to $\bm X$. Also, let $\{\gamma_s\}_{1 \leq s \leq n}$ be a collection of weight functions such that  $\prod_{s=1}^n \gamma_{s}(\bm Z_s) \stackrel{P} \rightarrow 1$, 
    as $n \rightarrow \infty$. Then the test function $\phi_{LWB} = \bm{1}\{ p_{LWB}<\alpha \}$ satisfies the following: 
    \begin{itemize}
\item[(a)] $\lim_{n, M \rightarrow \infty}\P_{H_0}(\phi_{LWB}) \leq \alpha$, that is, $\phi_{LWB}$ controls Type I error at $\alpha$.

    \item[(b)] If $M>(1-\alpha)/\alpha$, $\lim_{n \rightarrow \infty}\P_{H_1}(\phi_{LWB}) =1$,  that is, the power of $\phi_{LWB}$ converges to $1$, as $n \rightarrow \infty$. 
  \end{itemize}
   \label{LB-test}
\end{thm}

The proof of Theorem \ref{LB-test} is given in Appendix \ref{sec:LBpf}. The result shows that the local wild bootstrap asymptotically controls Type I error and is consistent  when the weights functions satisfy $\prod_{s=1}^n \gamma_{s}(\bm Z_s) \stackrel{P} \rightarrow 1$. To understand this condition, consider, independently for each $1 \leq s \leq n$, a random variable $\hat{\bm Z_s}$ which takes value $\bm Z_i$ with probability $\gamma_s(\bm Z_i)$, for $1 \leq i \leq n$. Then $\prod_{s=1}^n \gamma_{s}(\bm Z_s)$ is the probability that $\{\hat {\bm Z}_s= \bm Z_s: 1\leq s \leq n\}$. Hence, the condition $\prod_{s=1}^n \gamma_{s}(\bm Z_s) \stackrel{P} \rightarrow 1$ ensures that the resampled and the observed configurations are asymptotically unchanged, that is, each $\bm Z_s$ is resampled exactly once with high probability, for $1 \leq s \leq n$. There are 2 natural ways to ensure the condition $\prod_{s=1}^n \gamma_{s}(\bm Z_s) \stackrel{P} \rightarrow 1$ holds: (1) choosing $\gamma_s= \kappa_s$ as in \eqref{eq:kernelweights} with the bandwidth $h_{2}'$ satisfying $h_2' = o(n^{-2/d_Z})$ (see Lemma \ref{lemma:auxillary-bandwidth} in Appendix \ref{sec:weightpf}), and (2) choosing $\gamma_s = \delta_{\bm Z_s}$, for $1 \leq s \leq n$, where $\delta_{\bm z}$ denotes the point mass at the point $\bm z$. We adopt the latter choice in our numerical studies in Section \ref{sec:simulations}.

Another interesting consequence of Theorem \ref{LB-test} is that it allows us to choose the bandwidth $h_0$ separately from the bandwidths $h_1, h_2$ which are used for computing the test statistic.  This double-bandwidth approach is necessary to resolve the dichotomy between Type I error and power that is inherent in permutation/bootstrap calibration of conditional independence measures. The finer choice of $h_0$ ensures the Type I error control, while the coarser choices of $h_1, h_2$ ensure efficient estimation of the cBD estimate, leading to better power. This also bears resemblance with the double-binning local permutation method in  \citet{kim2021local}, where a similar condition on bin-widths appears (see the following remark).

\begin{rem}\label{rem:nh} 
Recently, \citet{kim2021local} proposed a local permutation method for conditional independence testing by partitioning the space of the conditioning random variable ${\bm Z}$ and permuting the variables $(\bm X, \bm Y)$ within each bin, to approximate the null distribution. They proved asymptotic Type I error rate control of the local permutation method, in the total variation sense, when the bin width $h$ is such that $nh^2\rightarrow 0$. However, to have good power a larger bin size is required for computing the test statistic. This leads to a double-binning approach, where they used a larger bin size to estimate the test statistic and a smaller bin size to calibrate the test. However, their method requires ${\bm Z}$ to have bounded support and only  the case where $\bm Z$ is univariate is considered. In contrast, our resampling algorithm is applicable even when ${\bm Z}$ has unbounded support or is multidimensional. 
\end{rem}

To further investigate how well the local wild bootstrap performs under the null hypothesis, 
we test the closeness of the distribution of the resampled test statistic $\hat \zeta_n(\hat {\mathcal D})$ and the distribution of the test statistic obtained using the CRT method $\hat\zeta_n(\mathcal D^\prime)$ (which assumes the knowledge of $\P_{\bm X | \bm Z}$) using the following simulation. For comparison we also implement the same for the discrete local bootstrap. Specifically, we consider the same setup as in Remark \ref{remark:CRT-MS}: 

$$Z\sim N(0,1), ~ Y = Z+\epsilon_1, \text{ and }X = Z +\epsilon_2, $$ where $\epsilon_1,\epsilon_2$ are i.i.d. $N(0,1)$ random variables. Here, also we take the bandwidths $h_1 = c_1n^{-1/(d_Y+d_Z+2)}$ and $h_2 = c_2 n^{-1/(d_Z+2)}$, where $c_1$ and $c_2$ are the averages of the coordinate-wise IQR of the observations $\{(Y_1,Z_1),\ldots, (Y_n,Z_n)\}$ and $\{Z_1,\ldots, Z_n\}$, respectively. For every $n$, we repeat the resampling algorithms $200$ times and use the Kolmogorov-Smirnov (KS) test to evaluate the closeness of the respective distributions. The nominal level is kept fixed at $\alpha = 0.05$. Our findings are reported in Figure \ref{fig:CRT-LWB-DLB}. We observe that the null distribution of the local wild bootstrapped test statistic is close to the distribution of the CRT resampled test statistic. 
On the other hand, the distribution of the discrete local bootstrapped test statistic is drastically different. 
This difference diminishes with increasing sample size, but one may require a very large sample size for this difference to be negligible. 

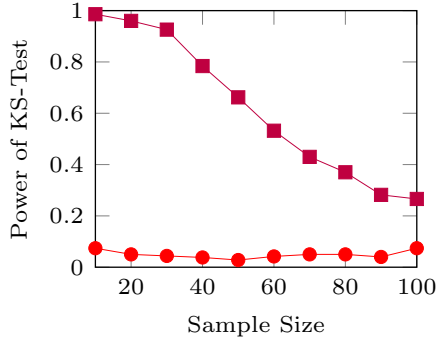
\begin{figure}[h]
\centering
\begin{tikzpicture}[scale = 1.25,font=\scriptsize]
\begin{axis}[xmin = 10, xmax = 100, ymin = 0, ymax = 1, xlabel = {Sample Size}, ylabel = {Power of KS-Test}, title = {}]
\addplot[color = red,   mark = *, step = 1cm,very thin]coordinates{(10,0.074)(20,0.05)(30,0.044)(40,0.038)(50,0.028)(60,0.042)(70,0.05)(80,0.05)(90,0.04)(100,0.074)};

\addplot[color = purple,   mark = square*, step = 1cm,very thin]coordinates{(10,0.986)(20,0.96)(30,0.926)(40,0.784)(50,0.662)(60,0.532)(70,0.43)(80,0.37)(90,0.282)(100,0.266)};

\end{axis}
\end{tikzpicture}



\caption{Power of the KS test for checking the distributional difference between the CRT resampled test statistic with the local wild bootstrapped test statistic (\textcolor{red}{$\tikzcircle{2pt}$}) and the discrete local bootstrapped test statistic ($\textcolor{purple}{\blacksquare}$) in Example \ref{example:normalXZ} for $r=0$.}
    \label{fig:CRT-LWB-DLB}
\end{figure}

\subsection{Choice of the Weight Function}

Recall that in Section \ref{sec:asymptoticdistribution} we considered the cBD measure with weight function $a(\bm y,\bm z) = p_{\bm Z}(\bm z)^4 p_{\bm Y,\bm Z}(\bm y,\bm z)^4$. This choice is convenient because the corresponding cBD estimate is asymptotically Gaussian under $H_0$. Although the limiting null distribution of the cBD estimate might be intractable for general weight functions, Theorem \ref{LB-test} provides a way to calibrate the cBD estimate that is asymptotical consistent for any choice of the weight function $a(\cdot, \cdot)$. This raises the question: How does the choice of the weight function effect the finite-sample power of the cBD based tests? We explore this through a small simulation with three possible weight functions (a) $a(\bm y,\bm z) =1$, (b) $a(\bm y,\bm z) = p_{\bm Y,\bm Z}(\bm y,\bm z)^2$, and (c) $a(\bm y, \bm z) = p_{\bm Z}(\bm z)^4p_{\bm Y,\bm Z}(\bm y,\bm z)^4$. For ease of computation, here we calibrate the test using the $V$-statistic cBD estimates. 

\begin{figure}[h]
    \centering
\begin{tikzpicture}[scale = 1.25,font=\scriptsize]
\begin{axis}[xmin = -2.1, xmax = 2.1, ymin = 0, ymax = 1, xlabel = {$r$}, ylabel = { Empirical Power }, title = {  }]

\addplot[color = red, mark = *, step = 1cm,very thin, mark size = 1.5pt]coordinates{(-2,1)(-1.5,0.974)(-1.2,0.874)(-0.9,0.658)(-0.6,0.262)(-0.3,0.074)(0,0.04)(0.3,0.116)(0.6,0.4)(0.9,0.74)(1.2,0.938)(1.5,0.986)(2,1)};

\addplot[color = purple, mark = triangle*, step = 1cm,very thin, mark size = 1.5pt]coordinates{(-2,0.992)(-1.5,0.92)(-1.2,0.792)(-0.9,0.508)(-0.6,0.196)(-0.3,0.06)(0,0.036)(0.3,0.082)(0.6,0.29)(0.9,0.606)(1.2,0.872)(1.5,0.954)(2,0.998)};

\addplot[color = magenta, mark = square*, step = 1cm,very thin, mark size = 1.5pt]coordinates{(-2,0.908)(-1.5,0.708)(-1.2,0.496)(-0.9,0.274)(-0.6,0.13)(-0.3,0.06)(0,0.04)(0.3,0.08)(0.6,0.19)(0.9,0.382)(1.2,0.596)(1.5,0.75)(2,0.942)};

\end{axis}
\end{tikzpicture}
    \caption{Empirical power of cBD the tests in Example \ref{example:weight} with weights (a) $a(\bm y,\bm z) = 1$ ({$\tikzcircle{2pt}$}), (b) $a(\bm y,\bm z) = p_{\bm Z}(\bm z)^2$ (\textcolor{purple}{$\blacktriangle$}) and (c) $a(\bm y,\bm z) = p_{\bm Z}(\bm z)^4 p_{\bm Z,\bm Y}(\bm y,\bm z)^4$ (\textcolor{magenta}{$\blacksquare$}).}
    \label{fig:ex0}
\end{figure}
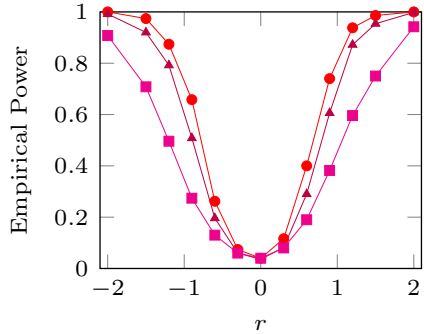

\begin{example}\label{example:weight}
Consider the model $$Y = Z + \eta_1 \text{ and } X = Z + rY+\eta_2,$$ 
where $\eta_1,\eta_2$ and $Z$ are i.i.d. $N(0, 1)$. For each of the 3 weight functions, the empirical power of the corresponding cBD test is computed for different values of $r$ keeping the sample size fixed at $n=50$. The bandwidths $h_1, h_2$ for computing the cBD statistic are set as in Example \ref{example:normalXZ}. Also, the bandwidth for the local bootstrap is chosen to be $h_0 = 20 c_2 n^{-1/1.95}$,  where $c_2$ is the IQR of $\{Z_1, \ldots , Z_n\}$, which satisfies the assumption $nh_0^2\rightarrow 0$ as in Theorem \ref{LB-test}. The results are reported in Figure \ref{fig:ex0}. We see that the test corresponding to $a(\bm Y,\bm Z) = 1$ has the highest power, followed by $a(\bm y,\bm z) = p_{\bm Y,\bm Z}(\bm y,\bm z)^2$ and $a(\bm Y,\bm Z) = p_{\bm Z}(\bm z)^4p_{\bm Y,\bm Z}(\bm y,\bm z)^4$. This shows that although the weight function $a(\bm Y,\bm Z) = p_{\bm Z}(\bm z)^4p_{\bm Y,\bm Z}(\bm y,\bm z)^4$ leads to a simple asymptotic null distribution, other choices can have better finite-sample performance. In particular, this example shows that choosing $a(\bm y,\bm z)=1$ might be beneficial in practice. 
\end{example}

\section{Analysis of Simulated and Benchmark Data Sets} 
\label{sec:simulations}

In this section, we evaluate the empirical performance of our test on simulated experiments and also on a benchmark data set. We compare our test with the tests based on the Conditional Distance Covariance \citep{wang2015}, the Generalized Covariance Measure \citep{shah2020hardness}, and the weighted Generalized Covariance Measure \citep{Scheidegger2022}, which are referred to as the cDC test, GCM tests, and wGCM test, respectively. There are multiple ways to carry out the GCM test depending upon the method used to estimate the regression function. Here, we use the generalized additive model for that purpose. \citet{Scheidegger2022} proposed two different tests using the wGCM where one can either choose a weight function apriori and use it for the entire data or estimate the optimal weight function based on a training data set and apply the test on the validation data set. Here, also we use generalized additive models to estimate the regression functions and the corresponding tests are referred to as wGCM.fix test and wGCM.est test, respectively. 

In our experiments we choose $h_{1} = c_1 n^{-1/(d_Y+d_Z+2)}$ and $h_{2} =c_2 n^{-1/(d_Z+2)}$ to estimate the test statistic, where $c_1$ and $c_2$ are the averages of the coordinate-wise IQR of the observations $\{ (\bm Y_1, \bm Z_1), \ldots, (\bm Y_n, \bm Z_n) \}$ and $\{ \bm Z_1,  \ldots,  \bm Z_n\}$, respectively.  For the local smooth bootstrap we let $\gamma_s = \delta_{\bm Z_s}$, for $1 \leq s \leq n$, which ensures $\prod_{s=1}^n \gamma_{s}(\bm Z_s) = 1$ and choose $h_0 = 20 ~ c_2 n^{-1/1.95}$. We refer to our test as the cBD test. In our simulated experiments, all $p$-values are computed based on $500$ replications and the empirical power is estimated by the proportion of times $H_0$ is rejected over $500$ repetitions. Throughout the nominal level is set at 5\%.

\subsection{Analysis of Simulated Data}

In this section we compare the performance of the proposed test with existing methods in simulation settings. We begin with a simple one dimensional regression model:  

\begin{example}\label{example:regression} Consider the regression model $$Y = Z + \eta_1 \text{ and } X = Z + rY+\eta_2,$$ where $\eta_1$ and $\eta_2$ are independent error terms and $Z$ follows a univariate distribution. Here, $X$ depends on both $Y$ and $Z$ when $r\not=0$ and the dependence becomes stronger as $|r|$ increases. We consider two cases: 
\begin{itemize}
\item[(a)] $Z,\eta_1,\eta_2$ are i.i.d. $N(0,1)$ random variables and 
\item[(b)] $Z,\eta_1,\eta_2$ are i.i.d. standard Cauchy random variables. 
\end{itemize} 
Powers of different tests are computed as a function $r$ with sample size fixed at $n=50$. The results are reported in Figure \ref{fig:ex1}. In Example \ref{example:regression} (a), the GCM test has the best performance, closely followed by the wGCM tests and the cDC test. Our test also has satisfactory performance, but its power is relatively lower compared to other tests. On the other hand, in Example \ref{example:regression} (b), our test and the cDC test have competitive performance and they significantly outperform the regression-based tests (the GCM and wGCM tests). This shows the lack of robustness of the regression-based tests against outliers and extreme values generated from heavy-tailed distributions. While the cDC test has more power when $|r|$ is small, the cBD test outperforms the cDC test for large $|r|$. 
\end{example}

\begin{figure}[h]
    \centering
\begin{tikzpicture}[scale = 1.25,font=\scriptsize]
\begin{axis}[xmin = -2.1, xmax = 2.1, ymin = 0, ymax = 1, xlabel = {$r$}, ylabel = { Empirical Power }, title = {\bf Example 4 (a)}]

\addplot[color = red, mark = *, step = 1cm,very thin, mark size = 1.5pt]coordinates{(-2,1)(-1.5,0.974)(-1.2,0.874)(-0.9,0.658)(-0.6,0.262)(-0.3,0.074)(0,0.04)(0.3,0.116)(0.6,0.4)(0.9,0.74)(1.2,0.938)(1.5,0.986)(2,1)};

\addplot[color = darkyellow, mark = triangle*, step = 1cm,very thin, mark size = 1.5pt]coordinates{(-2,1)(-1.5,1)(-1.2,1)(-0.9,0.952)(-0.6,0.532)(-0.3,0.3)(0,0.07)(0.3,0.572)(0.6,0.978)(0.9,1)(1.2,1)(1.5,1)(2,1)};

\addplot[color = blue, mark = square*, step = 1cm,very thin, mark size = 1.5pt]coordinates{(-2,1)(-1.5,1)(-1.2,1)(-0.9,0.998)(-0.6,0.956)(-0.3,0.502)(0,0.064)(0.3,0.506)(0.6,0.968)(0.9,0.998)(1.2,1)(1.5,1)(2,1)};


\addplot[color = ForestGreen, mark = diamond*, step = 1cm,very thin, mark size = 1.5pt]coordinates{(-2,0.986)(-1.5,0.982)(-1.2,0.97)(-0.9,0.93)(-0.6,0.738)(-0.3,0.294)(0,0.094)(0.3,0.304)(0.6,0.728)(0.9,0.918)(1.2,0.966)(1.5,0.988)(2,0.988)};

\addplot[color = black, mark = diamond*, step = 1cm,very thin, mark size = 1.5pt]coordinates{(-2,1)(-1.5,1)(-1.2,1)(-0.9,0.984)(-0.6,0.844)(-0.3,0.318)(0,0.062)(0.3,0.324)(0.6,0.854)(0.9,0.99)(1.2,0.998)(1.5,0.998)(2,1)};
\end{axis}
\end{tikzpicture}
\begin{tikzpicture}[scale = 1.25,font=\scriptsize]
\begin{axis}[xmin = -2.1, xmax = 2.1, ymin = 0, ymax = 1, xlabel = {$r$}, ylabel = { Empirical Power }, title = {\bf Example 4 (b)}]


\addplot[color = red, mark = *, step = 1cm,very thin, mark size = 1.5pt]coordinates{(-2,0.98)(-1.5,0.924)(-1.2,0.842)(-0.9,0.628)(-0.6,0.402)(-0.3,0.146)(0,0.054)(0.3,0.208)(0.6,0.544)(0.9,0.776)(1.2,0.892)(1.5,0.94)(2,0.984)};

\addplot[color = darkyellow, mark = triangle*, step = 1cm,very thin, mark size = 1.5pt]coordinates{(-2,0.906)(-1.5,0.872)(-1.2,0.846)(-0.9,0.776)(-0.6,0.654)(-0.3,0.406)(0,0.054)(0.3,0.438)(0.6,0.686)(0.9,0.81)(1.2,0.848)(1.5,0.88)(2,0.908)};

\addplot[color = blue, mark = square*, step = 1cm,very thin, mark size = 1.5pt]coordinates{(-2,0.256)(-1.5,0.246)(-1.2,0.234)(-0.9,0.212)(-0.6,0.168)(-0.3,0.114)(0,0.014)(0.3,0.122)(0.6,0.21)(0.9,0.232)(1.2,0.252)(1.5,0.258)(2,0.264)};


\addplot[color = ForestGreen, mark = diamond*, step = 1cm,very thin, mark size = 1.5pt]coordinates{(-2,0.382)(-1.5,0.362)(-1.2,0.332)(-0.9,0.286)(-0.6,0.232)(-0.3,0.138)(0,0.028)(0.3,0.152)(0.6,0.248)(0.9,0.302)(1.2,0.332)(1.5,0.35)(2,0.368)};

\addplot[color = black, mark = diamond*, step = 1cm,very thin, mark size = 1.5pt]coordinates{(-2,0.13)(-1.5,0.114)(-1.2,0.104)(-0.9,0.078)(-0.6,0.06)(-0.3,0.034)(0,0.004)(0.3,0.04)(0.6,0.072)(0.9,0.098)(1.2,0.108)(1.5,0.118)(2,0.138)};
\end{axis}
\end{tikzpicture}
    \caption{Empirical power of the cBD test ({$\tikzcircle{2pt}$}), cDC test (\textcolor{darkyellow}{$\blacktriangle$}), GCM test (\textcolor{blue}{$\blacksquare$}) wGCM.est test (\textcolor{ForestGreen}{$\blacklozenge$}), and wGCM.fix test (\textcolor{black}{$\blacklozenge$}) for Example \ref{example:regression} (a) and (b).}
    \label{fig:ex1}
\end{figure}
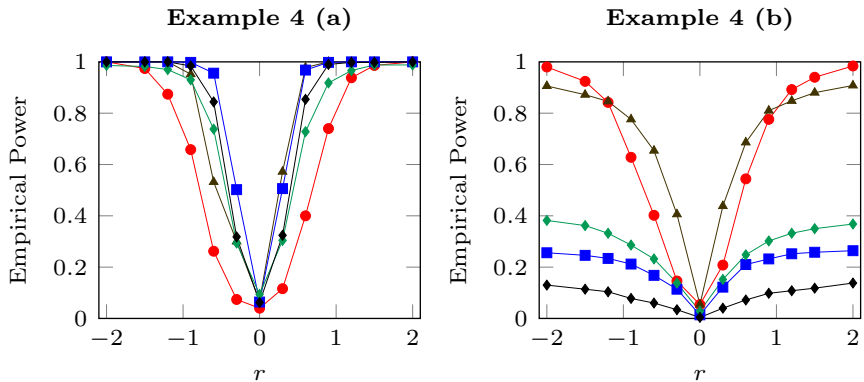

\begin{figure}[t]
    \centering
    \includegraphics[scale = 0.40]{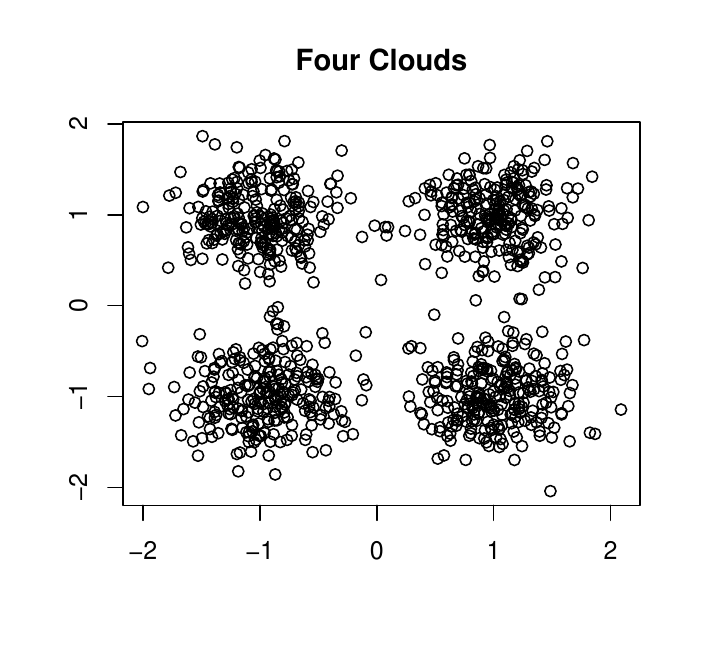}
    \includegraphics[scale = 0.40]{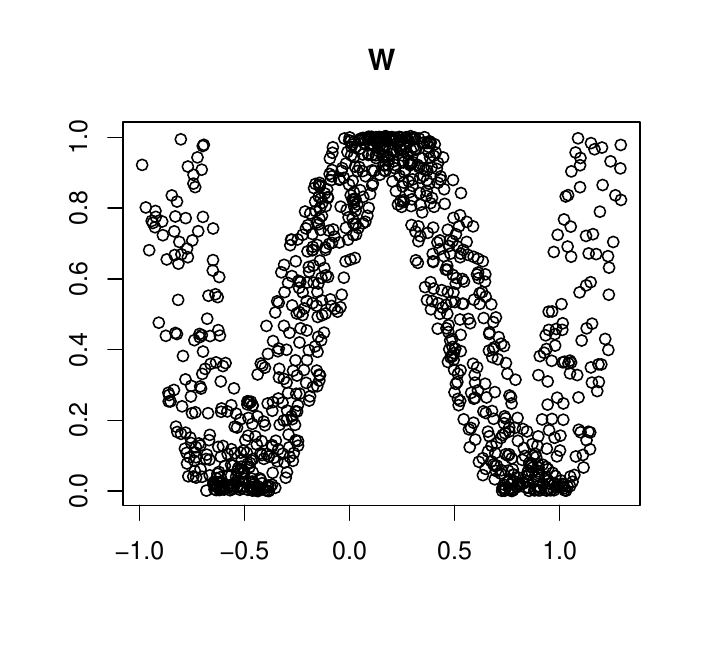}
    \includegraphics[scale = 0.40]{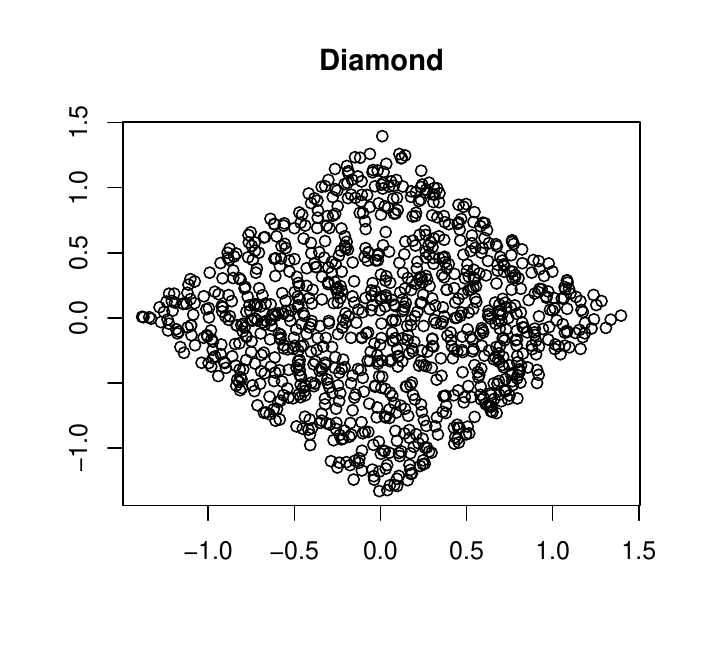}
    
    \vspace{-0.35in}
    
    \includegraphics[scale = 0.40]{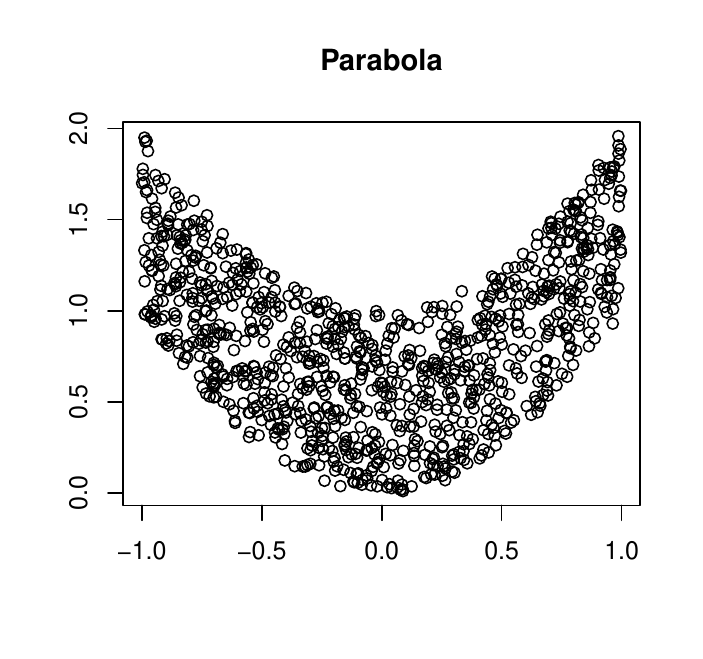}
    \includegraphics[scale = 0.40]{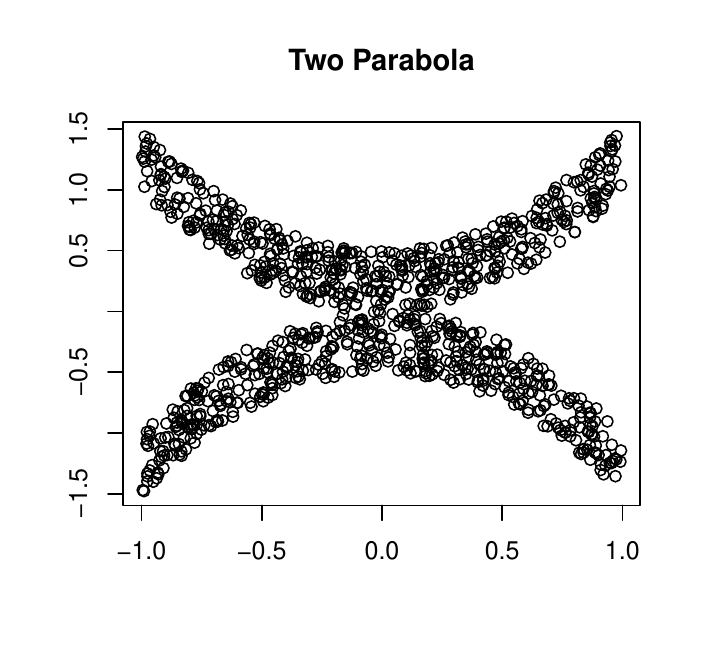}
    \includegraphics[scale = 0.40]{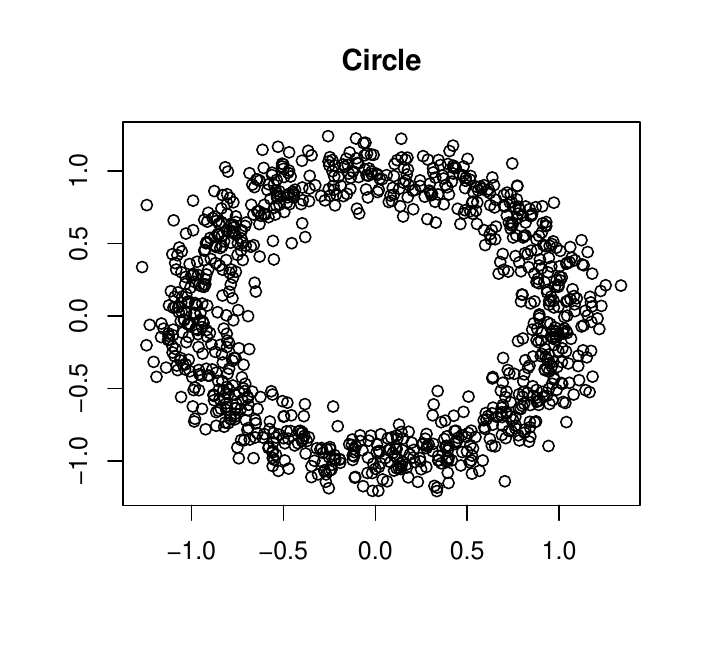}
    \caption{Scatter plots of observations from the six unusual bivariate distributions in \cite{newton2009introducing}}
    \label{fig:Newton}
\end{figure}

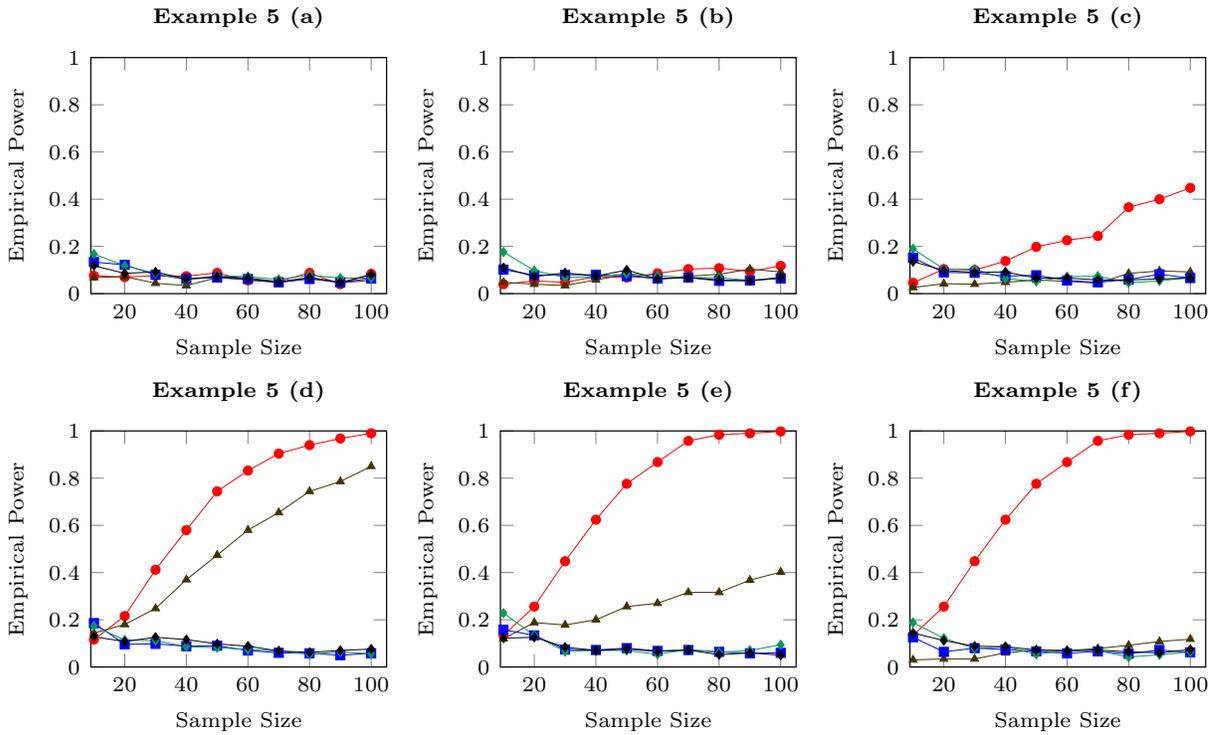
\begin{figure}[!b]
    \centering
\begin{tikzpicture}[scale = 1.15,font=\scriptsize]
\begin{axis}[xmin = 9, xmax = 105, ymin = 0, ymax = 1, xlabel = {Sample Size}, ylabel = { Empirical Power }, title = {\bf Example 5 (a)}]


\addplot[color = red, mark = *, step = 1cm,very thin, mark size = 1.5pt]coordinates{(10,0.076)(20,0.07)(30,0.076)(40,0.074)(50,0.088)(60,0.056)(70,0.052)(80,0.088)(90,0.04)(100,0.084)};

\addplot[color = darkyellow, mark = triangle*, step = 1cm,very thin, mark size = 1.5pt]coordinates{(10,0.068)(20,0.072)(30,0.044)(40,0.034)(50,0.068)(60,0.058)(70,0.046)(80,0.068)(90,0.046)(100,0.056)};

\addplot[color = blue, mark = square*, step = 1cm,very thin, mark size = 1.5pt]coordinates{(10,0.134)(20,0.122)(30,0.08)(40,0.062)(50,0.068)(60,0.062)(70,0.048)(80,0.062)(90,0.048)(100,0.064)};


\addplot[color = ForestGreen, mark = diamond*, step = 1cm,very thin, mark size = 1.5pt]coordinates{(10,0.168)(20,0.118)(30,0.086)(40,0.058)(50,0.076)(60,0.072)(70,0.06)(80,0.078)(90,0.066)(100,0.066)};

\addplot[color = black, mark = diamond*, step = 1cm,very thin, mark size = 1.5pt]coordinates{(10,0.118)(20,0.086)(30,0.092)(40,0.062)(50,0.072)(60,0.066)(70,0.05)(80,0.068)(90,0.048)(100,0.08)};
\end{axis}
\end{tikzpicture}
\begin{tikzpicture}[scale = 1.15,font=\scriptsize]
\begin{axis}[xmin = 9, xmax = 105, ymin = 0, ymax = 1, xlabel = {Sample Size}, ylabel = { Empirical Power }, title = {\bf Example 5 (b)}]


\addplot[color = red, mark = *, step = 1cm,very thin, mark size = 1.5pt]coordinates{(10,0.04)(20,0.054)(30,0.048)(40,0.072)(50,0.068)(60,0.086)(70,0.104)(80,0.108)(90,0.094)(100,0.118)};

\addplot[color = darkyellow, mark = triangle*, step = 1cm,very thin, mark size = 1.5pt]coordinates{(10,0.048)(20,0.04)(30,0.034)(40,0.058)(50,0.082)(60,0.062)(70,0.074)(80,0.082)(90,0.104)(100,0.092)};

\addplot[color = blue, mark = square*, step = 1cm,very thin, mark size = 1.5pt]coordinates{(10,0.102)(20,0.078)(30,0.08)(40,0.08)(50,0.074)(60,0.064)(70,0.068)(80,0.054)(90,0.056)(100,0.064)};


\addplot[color = ForestGreen, mark = diamond*, step = 1cm,very thin, mark size = 1.5pt]coordinates{(10,0.176)(20,0.098)(30,0.068)(40,0.072)(50,0.084)(60,0.076)(70,0.068)(80,0.068)(90,0.056)(100,0.068)};

\addplot[color = black, mark = diamond*, step = 1cm,very thin, mark size = 1.5pt]coordinates{(10,0.11)(20,0.072)(30,0.088)(40,0.074)(50,0.1)(60,0.062)(70,0.068)(80,0.058)(90,0.054)(100,0.068)};
\end{axis}
\end{tikzpicture}
\begin{tikzpicture}[scale = 1.15,font=\scriptsize]
\begin{axis}[xmin = 9, xmax = 105, ymin = 0, ymax = 1, xlabel = {Sample Size}, ylabel = { Empirical Power }, title = {\bf Example 5 (c)}]


\addplot[color = red, mark = *, step = 1cm,very thin, mark size = 1.5pt]coordinates{(10,0.046)(20,0.104)(30,0.1)(40,0.138)(50,0.198)(60,0.226)(70,0.244)(80,0.366)(90,0.4)(100,0.448)};

\addplot[color = darkyellow, mark = triangle*, step = 1cm,very thin, mark size = 1.5pt]coordinates{(10,0.026)(20,0.042)(30,0.04)(40,0.048)(50,0.058)(60,0.052)(70,0.044)(80,0.086)(90,0.096)(100,0.092)};

\addplot[color = blue, mark = square*, step = 1cm,very thin, mark size = 1.5pt]coordinates{(10,0.152)(20,0.09)(30,0.088)(40,0.074)(50,0.078)(60,0.056)(70,0.048)(80,0.06)(90,0.082)(100,0.066)};


\addplot[color = ForestGreen, mark = diamond*, step = 1cm,very thin, mark size = 1.5pt]coordinates{(10,0.19)(20,0.104)(30,0.104)(40,0.064)(50,0.052)(60,0.072)(70,0.074)(80,0.046)(90,0.054)(100,0.068)};

\addplot[color = black, mark = diamond*, step = 1cm,very thin, mark size = 1.5pt]coordinates{(10,0.136)(20,0.098)(30,0.09)(40,0.092)(50,0.066)(60,0.066)(70,0.058)(80,0.058)(90,0.062)(100,0.066)};
\end{axis}
\end{tikzpicture}
\begin{tikzpicture}[scale = 1.15,font=\scriptsize]
\begin{axis}[xmin = 9, xmax = 105, ymin = 0, ymax = 1, xlabel = {Sample Size}, ylabel = { Empirical Power }, title = {\bf Example 5 (d)}]


\addplot[color = red, mark = *, step = 1cm,very thin, mark size = 1.5pt]coordinates{(10,0.116)(20,0.216)(30,0.412)(40,0.58)(50,0.744)(60,0.832)(70,0.904)(80,0.94)(90,0.968)(100,0.99)};

\addplot[color = darkyellow, mark = triangle*, step = 1cm,very thin, mark size = 1.5pt]coordinates{(10,0.144)(20,0.18)(30,0.248)(40,0.37)(50,0.474)(60,0.58)(70,0.654)(80,0.744)(90,0.786)(100,0.85)};

\addplot[color = blue, mark = square*, step = 1cm,very thin, mark size = 1.5pt]coordinates{(10,0.186)(20,0.096)(30,0.098)(40,0.088)(50,0.09)(60,0.07)(70,0.06)(80,0.058)(90,0.05)(100,0.058)};


\addplot[color = ForestGreen, mark = diamond*, step = 1cm,very thin, mark size = 1.5pt]coordinates{(10,0.174)(20,0.114)(30,0.112)(40,0.086)(50,0.082)(60,0.074)(70,0.068)(80,0.054)(90,0.064)(100,0.056)};

\addplot[color = black, mark = diamond*, step = 1cm,very thin, mark size = 1.5pt]coordinates{(10,0.128)(20,0.108)(30,0.126)(40,0.116)(50,0.098)(60,0.088)(70,0.068)(80,0.062)(90,0.07)(100,0.076)};
\end{axis}
\end{tikzpicture}
\begin{tikzpicture}[scale = 1.15,font=\scriptsize]
\begin{axis}[xmin = 9, xmax = 105, ymin = 0, ymax = 1, xlabel = {Sample Size}, ylabel = { Empirical Power }, title = {\bf Example 5 (e)}]


\addplot[color = red, mark = *, step = 1cm,very thin, mark size = 1.5pt]coordinates{(10,0.134)(20,0.256)(30,0.448)(40,0.624)(50,0.776)(60,0.868)(70,0.958)(80,0.984)(90,0.99)(100,0.998)};

\addplot[color = darkyellow, mark = triangle*, step = 1cm,very thin, mark size = 1.5pt]coordinates{(10,0.12)(20,0.188)(30,0.178)(40,0.2)(50,0.256)(60,0.27)(70,0.316)(80,0.316)(90,0.368)(100,0.402)};

\addplot[color = blue, mark = square*, step = 1cm,very thin, mark size = 1.5pt]coordinates{(10,0.158)(20,0.134)(30,0.074)(40,0.072)(50,0.08)(60,0.068)(70,0.072)(80,0.064)(90,0.058)(100,0.06)};


\addplot[color = ForestGreen, mark = diamond*, step = 1cm,very thin, mark size = 1.5pt]coordinates{(10,0.228)(20,0.134)(30,0.066)(40,0.07)(50,0.07)(60,0.054)(70,0.074)(80,0.062)(90,0.07)(100,0.094)};

\addplot[color = black, mark = diamond*, step = 1cm,very thin, mark size = 1.5pt]coordinates{(10,0.122)(20,0.126)(30,0.084)(40,0.07)(50,0.076)(60,0.066)(70,0.072)(80,0.052)(90,0.06)(100,0.05)};
\end{axis}
\end{tikzpicture}
\begin{tikzpicture}[scale = 1.15,font=\scriptsize]
\begin{axis}[xmin = 9, xmax = 105, ymin = 0, ymax = 1, xlabel = {Sample Size}, ylabel = { Empirical Power }, title = {\bf Example 5 (f)}]


\addplot[color = red, mark = *, step = 1cm,very thin, mark size = 1.5pt]coordinates{(10,0.134)(20,0.256)(30,0.448)(40,0.624)(50,0.776)(60,0.868)(70,0.958)(80,0.984)(90,0.99)(100,0.998)};

\addplot[color = darkyellow, mark = triangle*, step = 1cm,very thin, mark size = 1.5pt]coordinates{(10,0.03)(20,0.034)(30,0.034)(40,0.06)(50,0.072)(60,0.07)(70,0.078)(80,0.092)(90,0.108)(100,0.118)};

\addplot[color = blue, mark = square*, step = 1cm,very thin, mark size = 1.5pt]coordinates{(10,0.126)(20,0.064)(30,0.08)(40,0.074)(50,0.066)(60,0.058)(70,0.066)(80,0.056)(90,0.072)(100,0.062)};


\addplot[color = ForestGreen, mark = diamond*, step = 1cm,very thin, mark size = 1.5pt]coordinates{(10,0.188)(20,0.122)(30,0.078)(40,0.086)(50,0.054)(60,0.07)(70,0.07)(80,0.042)(90,0.052)(100,0.064)};

\addplot[color = black, mark = diamond*, step = 1cm,very thin, mark size = 1.5pt]coordinates{(10,0.144)(20,0.112)(30,0.088)(40,0.086)(50,0.072)(60,0.068)(70,0.072)(80,0.064)(90,0.06)(100,0.074)};
\end{axis}
\end{tikzpicture}

    \caption{Empirical power of the cBD test ({$\tikzcircle{2pt}$}), cDC test (\textcolor{darkyellow}{$\blacktriangle$}), GCM test (\textcolor{blue}{$\blacksquare$}), wGCM.est test (\textcolor{ForestGreen}{$\blacklozenge$}) and wGCM.fix test (\textcolor{black}{$\blacklozenge$}) for Example \ref{example:clouds} with (a) Four clouds, (b) W, (c) Diamond, (d) Parabola, (e) Two parabola and (f) Circle unusual error distributions.} 
    \label{fig:ex2}
\end{figure}

\begin{example}\label{example:clouds} 
Here, we consider bivariate distributions with geometric dependence structures.  Specifically, suppose $Z\sim Unif(0,1)$, $X = Z + \eta_1$, and $Y= Z + \eta_2$, where $(\eta_1,\eta_2)$ are generated from one of the six unusual bivariate distributions described in \cite{newton2009introducing}. The scatter plot of observations from these six bivariate distributions (referred to as (a) Four Clouds, (b) W, (c) Diamond, (d) Parabola, (e) Two Parabolas, and (f) Circle)  are shown in Figure \ref{fig:Newton}. It is easy to see that in the Four Clouds example, $X\indpt Y\mid Z$ while for the rest we have $X\not\indpt Y\mid Z$. Here we compute the power of the different tests for varying sample sizes, which are reported in Figure \ref{fig:ex2}. Note that in each of these examples, the Pearson's correlation coefficient between $(\eta_1,\eta_2)$ is zero. Hence, the residuals in the GCM and wGCM tests are expected to be uncorrelated, rendering these tests powerless. This is validated in Figure \ref{fig:ex2}, where the GCM and wGCM tests have power close to the nominal level of $0.05$ for all the examples. In the Four Clouds example, the powers of all tests are close to the nominal level $\alpha=0.05$, as expected. In the W example, none of the tests have satisfactory performance. In the Diamond example, the power curve of our test gradually increases with the sample size, whereas those for the other tests remain near the nominal level. For the Parabola and Two Parabolas examples, the power curves of our test and cDC test both increase with the sample size, but our test significantly outperforms the cDC test. In the Circle example, our test also has significantly higher power compared to the other tests even when the sample size is moderately large. Overall, our test has the highest power among competing methods in detecting these unusual dependencies. 
\end{example} 

\begin{example}\label{example:dimension} Suppose $X, Y$ and $Z$ are as in Example \ref{example:clouds}, with $(\eta_1,\eta_2)$ generated from an equal mixture of two bivariate normal distributions. These bivariate distributions have mean zero and variance-covariance matrix $\Sigma_1$ and $\Sigma_2$, respectively. We consider two cases: $$\text{ (a) }\Sigma_1 = \begin{pmatrix}1 & 0\\ 0 & 1\end{pmatrix},~ \Sigma_2 = \begin{pmatrix}10 & 0\\ 0 & 10\end{pmatrix} \quad \text{ and }  \quad \text{ (b) } \Sigma_1 = \begin{pmatrix}1 & 0\\ 0 & 10\end{pmatrix}, ~ \Sigma_2 = \begin{pmatrix}10 & 0\\ 0 & 1\end{pmatrix}. $$ 
Figure \ref{fig:ex3} shows the power of the tests for varying sample sizes. Here, the regression-based tests have very poor performance, because the  covariance between $(\eta_1,\eta_2)$ is zero. In Example \ref{example:dimension} (a) our test slightly outperforms the cDC test. However, in Example \ref{example:dimension} (b) the cBD test significantly outperforms the cDC and the other tests. Note that in this case the correlation between pairwise distances of $\eta_1$ and  $\eta_2$ is negative, where distance correlation based methods tend to perform poorly. For instance, in the context of independence testing \citet{sarkar2018} noted that if the correlation between the pairwise distances is negative, the distance correlation test performs poorly even for moderately large sample size. 
\end{example} 

\begin{figure}[t]
    \centering
\begin{tikzpicture}[scale = 1.25,font=\scriptsize]
\begin{axis}[xmin = 9, xmax = 105, ymin = 0, ymax = 1, xlabel = {Sample Size}, ylabel = { Empirical Power }, title = {\bf Example 6 (a)}]


\addplot[color = red, mark = *, step = 1cm,very thin, mark size = 1.5pt]coordinates{(10,0.118)(20,0.348)(30,0.536)(40,0.646)(50,0.766)(60,0.828)(70,0.89)(80,0.92)(90,0.932)(100,0.966)};

\addplot[color = darkyellow, mark = triangle*, step = 1cm,very thin, mark size = 1.5pt]coordinates{(10,0.21)(20,0.32)(30,0.446)(40,0.556)(50,0.698)(60,0.78)(70,0.86)(80,0.916)(90,0.94)(100,0.964)};

\addplot[color = blue, mark = square*, step = 1cm,very thin, mark size = 1.5pt]coordinates{(10,0.132)(20,0.08)(30,0.064)(40,0.066)(50,0.052)(60,0.042)(70,0.05)(80,0.048)(90,0.044)(100,0.046)};


\addplot[color = ForestGreen, mark = diamond*, step = 1cm,very thin, mark size = 1.5pt]coordinates{(10,0.206)(20,0.092)(30,0.082)(40,0.066)(50,0.104)(60,0.072)(70,0.074)(80,0.058)(90,0.058)(100,0.058)};

\addplot[color = black, mark = diamond*, step = 1cm,very thin, mark size = 1.5pt]coordinates{(10,0.04)(20,0.048)(30,0.07)(40,0.036)(50,0.074)(60,0.054)(70,0.06)(80,0.054)(90,0.042)(100,0.038)};
\end{axis}
\end{tikzpicture}
\begin{tikzpicture}[scale = 1.25,font=\scriptsize]
\begin{axis}[xmin = 9, xmax = 105, ymin = 0, ymax = 1, xlabel = {Sample Size}, ylabel = { Empirical Power }, title = {\bf Example 6 (b)}]


\addplot[color = red, mark = *, step = 1cm,very thin, mark size = 1.5pt]coordinates{(10,0.06)(20,0.19)(30,0.358)(40,0.474)(50,0.64)(60,0.722)(70,0.788)(80,0.836)(90,0.91)(100,0.934)};

\addplot[color = darkyellow, mark = triangle*, step = 1cm,very thin, mark size = 1.5pt]coordinates{(10,0.002)(20,0.004)(30,0.002)(40,0.004)(50,0.014)(60,0.028)(70,0.09)(80,0.208)(90,0.392)(100,0.606)};

\addplot[color = blue, mark = square*, step = 1cm,very thin, mark size = 1.5pt]coordinates{(10,0.06)(20,0.036)(30,0.036)(40,0.02)(50,0.026)(60,0.032)(70,0.028)(80,0.026)(90,0.032)(100,0.016)};


\addplot[color = ForestGreen, mark = diamond*, step = 1cm,very thin, mark size = 1.5pt]coordinates{(10,0.186)(20,0.066)(30,0.062)(40,0.03)(50,0.08)(60,0.052)(70,0.038)(80,0.044)(90,0.032)(100,0.048)};

\addplot[color = black, mark = diamond*, step = 1cm,very thin, mark size = 1.5pt]coordinates{(10,0.028)(20,0.024)(30,0.03)(40,0.02)(50,0.028)(60,0.03)(70,0.008)(80,0.026)(90,0.026)(100,0.016)};
\end{axis}
\end{tikzpicture}
    \caption{Empirical power of the cBD test ({$\tikzcircle{2pt}$}), cDC test (\textcolor{darkyellow}{$\blacktriangle$}), GCM test (\textcolor{blue}{$\blacksquare$}), wGCM.est test (\textcolor{ForestGreen}{$\blacklozenge$}) and wGCM.fix test (\textcolor{black}{$\blacklozenge$}) for Example \ref{example:dimension} (a) and (b).}
    \label{fig:ex3}
\end{figure}

\begin{example}\label{example:regressionvector} Here, we consider the bivariate analogue of Example \ref{example:regression}. Specifically, suppose $${\bm Y = \bm Z + \bm \eta_1} \text{ and } {\bm X= \bm Z+r\bm Y+\bm \eta_2},$$ 
where $\bm \eta_1$ and $\bm \eta_2$ are independent error terms and $\bm Z$ follows a bivariate distribution. We consider two scenarios: 
\begin{itemize} 
\item[$(a)$] ${\bm Z,\bm \eta_1,\bm \eta_2}$ are i.i.d. observations from a bivariate standard normal distribution, 
\item[(b)] ${\bm Z,\bm \eta_1,\bm \eta_2}$ are i.i.d. and the components are independent standard Cauchy random variables. 
\end{itemize}
We compute the power of the tests for different values of $r$ and report them in Figure \ref{fig:ex4}. In Example \ref{example:regressionvector} (a) our test has slightly lower power compared to the other tests (similar to the univariate case considered in Example \ref{example:regression} (a)). However, in Example \ref{example:regressionvector} (b) only our test and the cDC test have satisfactory performances. Although, the cDC test has an edge over our test in this case, the probability of Type I error (the power at $r=0$) of the cDC test is 0.192, which is significantly higher than the nominal level of $0.05$. On the other hand, our test controls Type I error at the nominal level. In that sense, the power of cDC and cBD tests are not comparable in this scenario. 
\end{example}

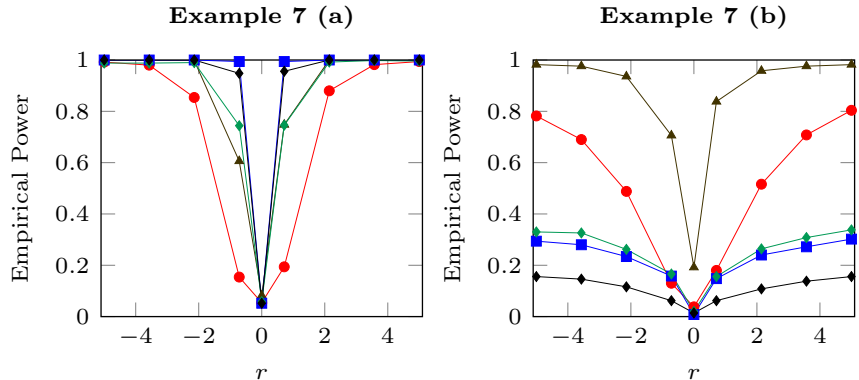
\begin{figure}[h]
    \centering
\begin{tikzpicture}[scale = 1.25,font=\scriptsize]
\begin{axis}[xmin = -5.1, xmax = 5.1, ymin = 0, ymax = 1, xlabel = {$r$}, ylabel = { Empirical Power }, title = {\bf Example 7 (a)}]


\addplot[color = red, mark = *, step = 1cm,very thin, mark size = 1.5pt]coordinates{(-5,0.992)(-3.5714286,0.98)(-2.1428571,0.854)(-0.7142857,0.154)(0,0.054)(0.7142857,0.194)(2.1428571,0.880)(3.5714286,0.982)(5,0.994)};

\addplot[color = darkyellow, mark = triangle*, step = 1cm,very thin, mark size = 1.5pt]coordinates{(-5,1)(-3.5714286,1)(-2.1428571,0.998)(-0.7142857,0.606)(0,0.082)(0.7142857,0.748)(2.1428571,1)(3.5714286,1)(5,1)};

\addplot[color = blue, mark = square*, step = 1cm,very thin, mark size = 1.5pt]coordinates{(-5,1)(-3.5714286,1)(-2.1428571,1)(-0.7142857,0.994)(0,0.052)(0.7142857,0.994)(2.1428571,1)(3.5714286,1)(5,1)};


\addplot[color = ForestGreen, mark = diamond*, step = 1cm,very thin, mark size = 1.5pt]coordinates{(-5,0.988)(-3.5714286,0.988)(-2.1428571,0.99)(-0.7142857,0.744)(0,0.058)(0.7142857,0.746)(2.1428571,0.992)(3.5714286,0.998)(5,0.996)};

\addplot[color = black, mark = diamond*, step = 1cm,very thin, mark size = 1.5pt]coordinates{(-5,1)(-3.5714286,1)(-2.1428571,1)(-0.7142857,0.948)(0,0.052)(0.7142857,0.956)(2.1428571,1)(3.5714286,1)(5,1)};
\end{axis}
\end{tikzpicture}
\begin{tikzpicture}[scale = 1.25,font=\scriptsize]
\begin{axis}[xmin = -5.1, xmax = 5.1, ymin = 0, ymax = 1, xlabel = {$r$}, ylabel = { Empirical Power }, title = {\bf Example 7 (b)}]


\addplot[color = red, mark = *, step = 1cm,very thin, mark size = 1.5pt]coordinates{(-5,0.782)(-3.5714286,0.69)(-2.1428571,0.488)(-0.7142857,0.13)(0,0.038)(0.7142857,0.18)(2.1428571,0.516)(3.5714286,0.708)(5,0.804)};

\addplot[color = darkyellow, mark = triangle*, step = 1cm,very thin, mark size = 1.5pt]coordinates{(-5,0.982)(-3.5714286,0.976)(-2.1428571,0.936)(-0.7142857,0.706)(0,0.192)(0.7142857,0.838)(2.1428571,0.958)(3.5714286,0.976)(5,0.982)};

\addplot[color = blue, mark = square*, step = 1cm,very thin, mark size = 1.5pt]coordinates{(-5,0.294)(-3.5714286,0.28)(-2.1428571,0.234)(-0.7142857,0.158)(0,0.008)(0.7142857,0.148)(2.1428571,0.24)(3.5714286,0.272)(5,0.302)};


\addplot[color = ForestGreen, mark = diamond*, step = 1cm,very thin, mark size = 1.5pt]coordinates{(-5,0.33)(-3.5714286,0.326)(-2.1428571,0.262)(-0.7142857,0.166)(0,0.02)(0.7142857,0.156)(2.1428571,0.264)(3.5714286,0.308)(5,0.338)};

\addplot[color = black, mark = diamond*, step = 1cm,very thin, mark size = 1.5pt]coordinates{(-5,0.156)(-3.5714286,0.146)(-2.1428571,0.116)(-0.7142857,0.062)(0,0.014)(0.7142857,0.062)(2.1428571,0.108)(3.5714286,0.138)(5,0.156)};
\end{axis}
\end{tikzpicture}
    \caption{Empirical power of the cBD test ({$\tikzcircle{2pt}$}), cDC test (\textcolor{darkyellow}{$\blacktriangle$}), GCM test (\textcolor{blue}{$\blacksquare$}), wGCM.est test (\textcolor{ForestGreen}{$\blacklozenge$}) and wGCM.fix test (\textcolor{black}{$\blacklozenge$}) for Example \ref{example:regressionvector} (a) and (b).}
    \label{fig:ex4}
\end{figure}

\begin{example}\label{example:parabolacircle} 
Now, we consider a multivariate analogue of Example \ref{example:clouds}. In particular, suppose ${\bm Z} = (Z_1,Z_2,Z_3)$ is generated uniformly from the unit cube $[0,1]^3$ and  $$(X,Y) = (\max\{Z_1,Z_2,Z_3\}+\varepsilon_1, \max\{Z_1,Z_2,Z_3\}+\varepsilon_2),$$ where $(\varepsilon_1,\varepsilon_2)$ are generated from the (a) Circle, (b) Parabola, and (c) Two parabola distributions as in Example \ref{example:clouds}. We compute the power of the tests for varying sample sizes and report them in Figure \ref{fig:ex5}. Here, the cDC test has extremely poor performance with power equal to zero in all instances. The regression-based tests also perform poorly, since the errors are uncorrelated. Only our test is able to successfully detect this form of conditional dependence, and its power increases with the sample size. 

\begin{figure}[h]
    \centering
\begin{tikzpicture}[scale = 1.15,font=\scriptsize]
\begin{axis}[xmin = 19, xmax = 105, ymin = 0, ymax = 0.5, xlabel = {Sample Size}, ylabel = { Empirical Power }, title = {\bf Example 8 (a)}]
\addplot[color = red, mark = *, step = 1cm,very thin, mark size = 1.5pt]coordinates{(20,0.018)(30,0.092)(40,0.112)(50,0.108)(60,0.12)(70,0.182)(80,0.224)(90,0.264)(100,0.304)};

\addplot[color = darkyellow, mark = triangle*, step = 1cm,very thin, mark size = 1.5pt]coordinates{(20,0)(30,0)(40,0)(50,0)(60,0)(70,0)(80,0)(90,0)(100,0)};

\addplot[color = blue, mark = square*, step = 1cm,very thin, mark size = 1.5pt]coordinates{(20,0.096)(30,0.088)(40,0.06)(50,0.086)(60,0.06)(70,0.074)(80,0.074)(90,0.094)(100,0.094)};

\addplot[color = ForestGreen, mark = diamond*, step = 1cm,very thin, mark size = 1.5pt]coordinates{(20,0.124)(30,0.112)(40,0.066)(50,0.062)(60,0.056)(70,0.066)(80,0.052)(90,0.052)(100,0.06)};

\addplot[color = black, mark = diamond*, step = 1cm,very thin, mark size = 1.5pt]coordinates{(20,0.116)(30,0.124)(40,0.082)(50,0.08)(60,0.066)(70,0.058)(80,0.068)(90,0.06)(100,0.074)};
\end{axis}
\end{tikzpicture}
\begin{tikzpicture}[scale = 1.15,font=\scriptsize]
\begin{axis}[xmin = 19, xmax = 105, ymin = 0, ymax = 0.5, xlabel = {Sample Size}, ylabel = { Empirical Power }, title = {\bf Example 8 (b)}]

\addplot[color = red, mark = *, step = 1cm,very thin, mark size = 1.5pt]coordinates{(20,0.078)(30,0.112)(40,0.132)(50,0.186)(60,0.214)(70,0.248)(80,0.286)(90,0.284)(100,0.348)};

\addplot[color = darkyellow, mark = triangle*, step = 1cm,very thin, mark size = 1.5pt]coordinates{(20,0)(30,0)(40,0)(50,0)(60,0)(70,0)(80,0)(90,0)(100,0)};

\addplot[color = blue, mark = square*, step = 1cm,very thin, mark size = 1.5pt]coordinates{(20,0.14)(30,0.12)(40,0.098)(50,0.08)(60,0.09)(70,0.092)(80,0.074)(90,0.094)(100,0.088)};

\addplot[color = ForestGreen, mark = diamond*, step = 1cm,very thin, mark size = 1.5pt]coordinates{(20,0.146)(30,0.096)(40,0.076)(50,0.088)(60,0.062)(70,0.062)(80,0.078)(90,0.078)(100,0.078)};

\addplot[color = black, mark = diamond*, step = 1cm,very thin, mark size = 1.5pt]coordinates{(20,0.156)(30,0.152)(40,0.112)(50,0.108)(60,0.11)(70,0.104)(80,0.08)(90,0.078)(100,0.106)};
\end{axis}
\end{tikzpicture}
\begin{tikzpicture}[scale = 1.15,font=\scriptsize]
\begin{axis}[xmin = 19, xmax = 105, ymin = 0, ymax = 0.5, xlabel = {Sample Size}, ylabel = { Empirical Power }, title = {\bf Example 8 (c)}]

\addplot[color = red, mark = *, step = 1cm,very thin, mark size = 1.5pt]coordinates{(20,0.07)(30,0.096)(40,0.154)(50,0.17)(60,0.19)(70,0.242)(80,0.242)(90,0.312)(100,0.386)};

\addplot[color = darkyellow, mark = triangle*, step = 1cm,very thin, mark size = 1.5pt]coordinates{(20,0)(30,0)(40,0)(50,0)(60,0)(70,0)(80,0)(90,0)(100,0)};

\addplot[color = blue, mark = square*, step = 1cm,very thin, mark size = 1.5pt]coordinates{(20,0.11)(30,0.096)(40,0.084)(50,0.08)(60,0.088)(70,0.06)(80,0.064)(90,0.054)(100,0.056)};

\addplot[color = ForestGreen, mark = diamond*, step = 1cm,very thin, mark size = 1.5pt]coordinates{(20,0.154)(30,0.086)(40,0.076)(50,0.062)(60,0.08)(70,0.08)(80,0.074)(90,0.06)(100,0.06)};

\addplot[color = black, mark = diamond*, step = 1cm,very thin, mark size = 1.5pt]coordinates{(20,0.158)(30,0.128)(40,0.082)(50,0.088)(60,0.072)(70,0.074)(80,0.06)(90,0.074)(100,0.068)};
\end{axis}
\end{tikzpicture}
    \caption{Empirical power of the cBD test ({$\tikzcircle{2pt}$}), cDC test (\textcolor{darkyellow}{$\blacktriangle$}), GCM test (\textcolor{blue}{$\blacksquare$}), wGCM.est test (\textcolor{ForestGreen}{$\blacklozenge$}) and wGCM.fix test (\textcolor{black}{$\blacklozenge$}) for Example \ref{example:parabolacircle} (a), (b) and (c).}
    \label{fig:ex5}
\end{figure}
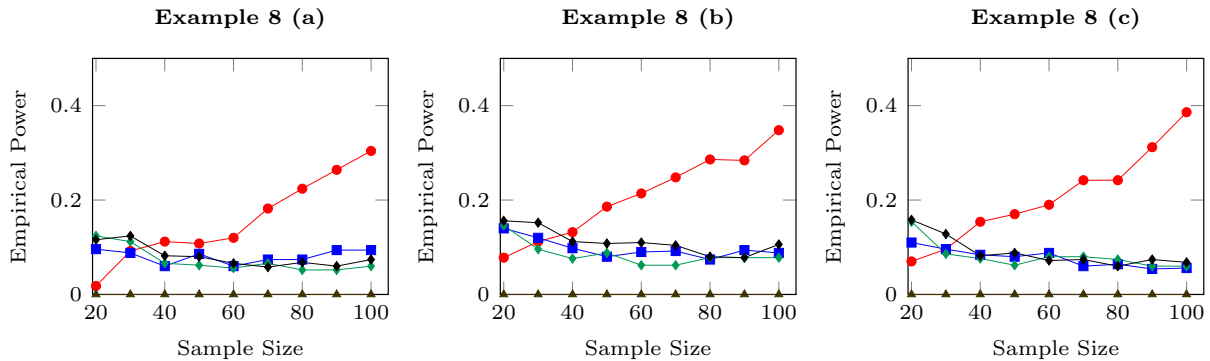

\end{example}

\subsection{Analysis of Benchmark Datasets}

In this section, we analyze the `marks' data set available in the R package `bnlearn'. This data set contains marks obtained by students in the subjects Mechanics (M), Vectors (V), Analysis (An), Algebra (Al), and Statistics (S), all measured on a scale of 0-100. This dataset has been previously analyzed in \cite{mardia,edwards} and \cite{whittaker} where the authors found that the marks obtained by the students in Statistics is conditionally dependent on that of Analysis given the marks on the rest of the subjects. They also found dependence of the marks in Mechanics and Vectors given the marks on the rest of the subjects. Therefore, considering these studies we consider testing (a) $H_0: S\indpt An \mid (M,V,Al)$ and (b) $H_0: M\indpt V\mid (S,An,Al)$ where we have evidence against the null. To compare the performance of the different methods we generate random sub-samples from the original data and carry out the tests based on those sub-samples. This procedure is repeated 500 times to estimate the powers of the tests by the proportions of times they reject  $H_0$. We evaluate the performance for different sub-sample sizes, and the results are given in Figure \ref{fig:marks}. In case (a), the GCM test has a better performance when the sample size is small, but our cBD test outperforms all other tests for relatively larger sample size. In case (b), when the sample size is small, the cDC test and GCM test has a better performance than our test, but when the sample size is large the cBD test significantly outperforms its competitors.    

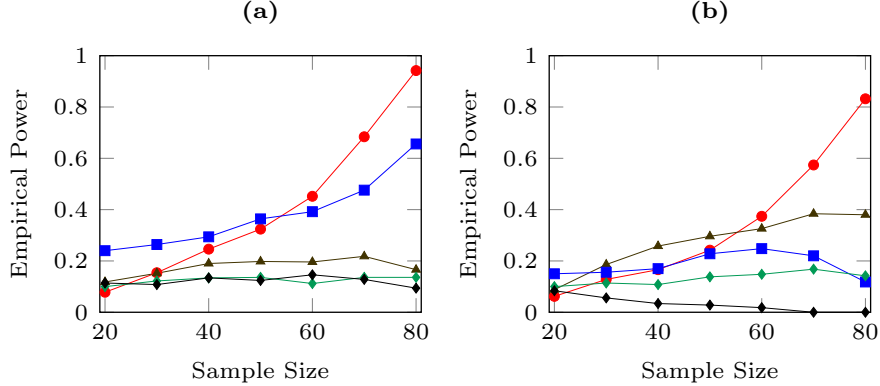
\begin{figure}[!h]
    \centering
\begin{tikzpicture}[scale = 1.25,font=\scriptsize]
\begin{axis}[xmin = 19, xmax = 81, ymin = 0, ymax = 1, xlabel = {Sample Size}, ylabel = { Empirical Power }, title = {\bf (a)}]
\addplot[color = red, mark = *, step = 1cm,very thin, mark size = 1.5pt]coordinates{(20,0.078)(30,0.154)(40,0.246)(50,0.324)(60,0.452)(70,0.684)(80,0.942)};

\addplot[color = darkyellow, mark = triangle*, step = 1cm,very thin, mark size = 1.5pt]coordinates{(20,0.118)(30,0.152)(40,0.19)(50,0.198)(60,0.196)(70,0.218)(80,0.166)};

\addplot[color = blue, mark = square*, step = 1cm,very thin, mark size = 1.5pt]coordinates{(20,0.24)(30,0.264)(40,0.294)(50,0.364)(60,0.392)(70,0.476)(80,0.656)};

\addplot[color = ForestGreen, mark = diamond*, step = 1cm,very thin, mark size = 1.5pt]coordinates{(20,0.1)(30,0.122)(40,0.134)(50,0.136)(60,0.112)(70,0.136)(80,0.136)};

\addplot[color = black, mark = diamond*, step = 1cm,very thin, mark size = 1.5pt]coordinates{(20,0.114)(30,0.108)(40,0.134)(50,0.124)(60,0.146)(70,0.128)(80,0.094)};
\end{axis}
\end{tikzpicture}
\begin{tikzpicture}[scale = 1.25,font=\scriptsize]
\begin{axis}[xmin = 19, xmax = 81, ymin = 0, ymax = 1, xlabel = {Sample Size}, ylabel = { Empirical Power }, title = {\bf (b)}]
\addplot[color = red, mark = *, step = 1cm,very thin, mark size = 1.5pt]coordinates{(20,0.062)(30,0.128)(40,0.166)(50,0.242)(60,0.374)(70,0.574)(80,0.832)};

\addplot[color = darkyellow, mark = triangle*, step = 1cm,very thin, mark size = 1.5pt]coordinates{(20,0.088)(30,0.186)(40,0.258)(50,0.296)(60,0.326)(70,0.384)(80,0.38)};

\addplot[color = blue, mark = square*, step = 1cm,very thin, mark size = 1.5pt]coordinates{(20,0.15)(30,0.156)(40,0.17)(50,0.228)(60,0.248)(70,0.22)(80,0.118)};

\addplot[color = ForestGreen, mark = diamond*, step = 1cm,very thin, mark size = 1.5pt]coordinates{(20,0.1)(30,0.114)(40,0.108)(50,0.138)(60,0.148)(70,0.168)(80,0.142)};

\addplot[color = black, mark = diamond*, step = 1cm,very thin, mark size = 1.5pt]coordinates{(20,0.084)(30,0.056)(40,0.034)(50,0.028)(60,0.018)(70,0)(80,0)};
\end{axis}
\end{tikzpicture}
    \caption{Empirical power of cBD ({$\tikzcircle{2pt}$}), cDC test (\textcolor{darkyellow}{$\blacktriangle$}), GCM test (\textcolor{blue}{$\blacksquare$}), wGCM.est test (\textcolor{ForestGreen}{$\blacklozenge$}) and wGCM.fix test (\textcolor{black}{$\blacklozenge$}) for testing (a) $S\indpt An\mid (M,V,Al)$ and (b) $S \indpt V\mid (M,An,Al)$ from the marks dataset.}
    \label{fig:marks}
\end{figure}

\section{Extensions and Future Directions}
\label{sec:conclusion}

In this paper, we propose a new measure of conditional dependence using the ball divergence. We estimate the measure using a kernel averaging method, study its asymptotic properties, and develop a test for conditional independence using a novel bootstrap approach. Our test exhibits improved performance both in simulations and real-data examples. 

Many extensions and variations of our method are possible. For instance, there has been a series of recent results, beginning with the breakthrough work of \citet{chatterjee2021}, on measures of  dependence which are normalized between $[0, 1]$, and takes the values 0 or 1, respectively, if and only if the corresponding pair of random variables is independent or one is a measurable function of the other almost surely (see \citet{chatterjee2022survey} for a survey of recent developments). For  measuring conditional dependence, \citet{azadkia2021simple} proposed a correlation coefficient with similar properties, that is, it lies between 0 and 1, takes value 0 if and only if $\bm X \indpt \bm Y| \bm Z$, and takes value 1 if and only if $\bm X$ is almost surely a measurable function of $\bm Y$ given $\bm Z$. The kernel partial correlation (KPC) of \citet{huang2020kernel} extends this to general topological spaces. In a similar vein, we can also construct a normalized measure for conditional dependence using the ball divergence. Specifically, taking the weight function $a=1$ in \eqref{eq:XYZ}, define the {\it normalized cBD} as follows: 
\begin{align}\label{eq:RXYZ}
 \mathcal{R}(\bm X, \bm Y| \bm Z) :=  \frac{\E\left [\Theta^2(\P_{\bm X| \bm Y, \bm Z}, \P_{\bm X| \bm Z}) ) \right]}{\mathbb E\left [\int \Theta^2(\P_{\delta_{\bm x}}, \P_{\bm X| \bm Z}) \mathrm d \P_{\bm X|\bm Y,\bm Z} (\bm x)   \right] }. 
\end{align}
Similar to the conditional dependence measures in \cite{azadkia2021simple} and \cite{huang2020kernel} the normalized cBD satisfies the following trifecta of properties (see Proposition \ref{prop:normalized-cBD}): (1) $\mathcal{R}(\bm X, \bm Y| \bm Z) \in [0, 1]$, (2) $\mathcal{R}(\bm X, \bm Y| \bm Z) = 0$ if and only if ${\bm X\indpt \bm Y | \bm Z}$, and (3)  $\mathcal{R}(\bm X, \bm Y| \bm Z) = 1$
if and only if $\bm X$ is a measurable function of $\bm Y$ and $\bm Z$ almost surely.

One can also consider defining a conditional dependence measure using the ball covariance of \citet{pan2019ball}. Given a pair of random vectors $(\bm X, \bm Y)$ in $\mathbb R^{d_X+d_Y}$ with joint distribution $\P_{\bm X, \bm Y}$
and marginal distributions $\P_{\bm X}$ and $\P_{\bm Y}$, respectively, the ball covariance 
with weight functions $\omega_1 : \mathbb R^{d_X} \times \mathbb R^{d_X} \rightarrow \mathbb R_{\geq 0}$ and $\omega_2 : \mathbb R^{d_Y} \times \mathbb R^{d_Y} \rightarrow \mathbb R_{\geq 0}$, is defined as: 
\begin{align*}
\mathrm{BCov}_{\omega_1, \omega_2}^2(\bm X, \bm Y) = \int (P_{\bm X, \bm Y} -P_{\bm X} \otimes P_{\bm Y} )^2 & \left( \bar{B}(\bm{u}_1, \rho(\bm{u}_1, \bm{u}_2)) \times \bar{B}(\bm{v}_1, \rho(\bm{v}_1, \bm{v}_2)) \right) \nonumber \\ 
& \omega_1(\bm{u}_1, \bm{u}_2) \omega_2(\bm{v}_1, \bm{v}_2) \{ \mathrm dP_{\bm X, \bm Y}(\bm{u}_1, \bm{v}_1) \mathrm dP_{\bm X, \bm Y}(\bm{u}_2, \bm{v}_2)  \} , 
\end{align*}
where $\bar{B}( \bm u, \varepsilon)$ is the closed ball of radius $\varepsilon$ centered at $\bm u$ in the appropriate dimension. The $\mathrm{BCov}$ characterizes (unconditional) independence under certain assumptions on the support of the distribution (see \citep[Theorem 2.1]{pan2019ball}). Using the $\mathrm{BCov}$, one can define a conditional dependence measure between $\bm X$ and $\bm Y$, given $\bm Z$ as follows: 
\begin{align*}
\mathrm{cBCov}_{\omega_1, \omega_2, a}^2(\bm X, \bm Y | \bm Z) = \E[ \mathrm{BCov}_{\omega_1, \omega_2}^2(\bm X | \bm Z, \bm Y | \bm Z) a (\bm Z) ] , 
\end{align*}
for some weight function $a : \mathbb R^{d_Z} \rightarrow \mathbb R_{\geq 0}$. It can be shown that $\mathrm{cBCov} = 0$ if and only if $\bm X\indpt \bm Y | \bm Z$, however, the conditions required on the support for this to hold can be restrictive in the context of conditional independence testing.  The cBCov can be estimated using a kernel averaging method similar to that in Section \ref{sec:estimation},  and its large sample properties can be studied using the framework developed in Section \ref{sec:empiricalprocess}. While the cBCov resembles a classical covariance measure, the cBD can
measure the strength of dependence, as discussed in the previous paragraph. Exploring how the cBCov compares with the cBD is an interesting future direction. \\

\small

\noindent\textbf{Acknowledgements}: B. B. Bhattacharya was supported by NSF CAREER grant DMS 2046393, NSF grant DMS 2113771, and a Sloan Research Fellowship.

\bibliographystyle{apalike}
\small
\bibliography{refs.bib}

\normalsize

\appendix

\section{Proof of Theorem \ref{consistency}}
\label{sec:consistencypf}

Recall, from \eqref{eq:XYZ}, that $\zeta(\bm X, \bm Y | \bm Z) = \E[\Theta^2( \mathbb P_{\bm X| \bm Y, \bm Z}, \mathbb P_{\bm X| \bm Z}) a({\bm Y, \bm Z})]$, where $\Theta^2(F, G)$ is the ball divergence between 2 probability distributions $F$ and $G$ as defined in \eqref{eq:uv}. Let $\phi$ be the core function of the ball divergence as in \eqref{eq:core-function}. For notational convenience denote $\underline{\bm u} = (\bm u_1, \bm u_2, \bm u_3, \bm u_4)$ and $\underline{\bm v} = (\bm v_1, \bm v_2, \bm v_3, \bm v_4)$, and note that,
\begin{align*}
    \Theta^2(\P_{\bm X| \bm Y= \bm y, \bm Z = \bm z}, \P_{\bm X| \bm Z = \bm z}) = \int \phi(\underline{\bm u}; \underline{\bm v}) \prod_{i=1}^4 \mathrm d \P_{\bm X| \bm Y = \bm y , \bm Z = \bm z }({\bm u}_i) \prod_{i=1}^4 \mathrm d \P_{\bm X| \bm Z = \bm z }({\bm v}_i).
\end{align*}
The corresponding estimate is given by
\begin{align*}
    \Theta^2(\tilde\P_{\bm X| \bm Y = \bm y ,\bm Z = \bm z }, \tilde\P_{\bm X| \bm Z = \bm z }) = \int \phi(\underline{\bm u}; \underline{\bm v}) \prod_{i=1}^4 \mathrm d \tilde\P_{\bm X| \bm Y = \bm y , \bm Z = \bm z }({\bm u}_i) \prod_{i=1}^4 \mathrm d \tilde\P_{\bm X| \bm Z = \bm z }({\bm v}_i), 
\end{align*}
where 
$\Tilde{\mathbb P}_{\bm X| \bm Y = \bm y , \bm Z = \bm z }(A)$ and $\Tilde{\mathbb P}_{\bm X| \bm Z = \bm z }(A)$ are as defined in \eqref{eq:Pyzestimate} and \eqref{eq:Pzestimate}, respectively. Now, using Proposition \ref{prop:main-8} we have,
\begin{align*}
    \Theta^2(\tilde\P_{\bm X| \bm Y = \bm y ,\bm Z = \bm z }, \tilde\P_{\bm X| \bm Z = \bm z }) = \Theta^2(\P_{\bm X| \bm Y = \bm y ,\bm Z = \bm z }, \P_{\bm X\mid\bm Z = \bm z }) + O_P\left(\frac{1}{\sqrt{nh_{1}^{(d_Y+d_Z)} } } \right) . 
\end{align*}
Therefore,   
\begin{align*}
\Theta^2(\Tilde{\mathbb P}_{\bm X| \bm Y = \bm y, \bm Z = \bm z},\Tilde{\mathbb P}_{\bm X| \bm Z = \bm z}) a(\bm y, \bm z) \stackrel{P} \rightarrow \Theta^2( \mathbb P_{\bm X| \bm Y = \bm y, \bm Z = \bm z}, \mathbb P_{\bm X| \bm Z = \bm z})a(\bm y, \bm z) .
\end{align*} 
Then by the Dominated Convergence Theorem (DCT),
\begin{align}\label{eq:XYZestimatepf}
\Theta^2(\Tilde{\mathbb P}_{\bm X| \bm Y_1 , \bm Z_1 },\Tilde{\mathbb P}_{\bm X| \bm Z_1}) a( \bm Y_1, \bm Z_1) \stackrel{P} \rightarrow \Theta^2( \mathbb P_{\bm X| \bm Y_1 , \bm Z_1 }, \mathbb P_{\bm X| \bm Z_1 })a(\bm Y_1, \bm Z_1) .
\end{align} 
Moreover, since $|\Theta^2(\Tilde{\mathbb P}_{\bm X| \bm Y_1, \bm Z_1},\Tilde{\mathbb P}_{\bm X| \bm Z_1})| \leq 2$ and $a(\bm y,\bm z)$ is bounded, 
$$\mathbb E \left[ \left(
\Theta^2(\Tilde{\mathbb P}_{\bm X| \bm Y_1 , \bm Z_1 },\Tilde{\mathbb P}_{\bm X| \bm Z_1 }) a(\bm Y_1, \bm Z_1) \right)^2 \right]  \leq 
4\E \left[ a^2(\bm Y_1, \bm Z_1) \right] <\infty.$$ Hence, the sequence $\Theta^2(\Tilde{\mathbb P}_{\bm X| \bm Y, \bm Z},\Tilde{\mathbb P}_{\bm X| \bm Z}) a(\bm Y, \bm Z)$ is asymptotically uniformly integrable and from \eqref{eq:XYZestimatepf} we have,   
\begin{align}\label{eq:PXYZ}
\E [\hat{\zeta}_{a, n}] = \E \left[ \Theta^2(\Tilde{\mathbb P}_{\bm X | \bm Y_1, \bm Z_1},\Tilde{\mathbb P}_{\bm X| \bm Z_1}) a(\bm Y_1, \bm Z_1) \right] \rightarrow \E\left[\Theta^2( \mathbb P_{\bm X| \bm Y_1, \bm Z_1}, \mathbb P_{\bm X| \bm Z_1})a(\bm Y_1, \bm Z_1)\right] = \zeta_a(\bm X,\bm Y| \bm Z),
\end{align} 
as $n\rightarrow\infty$. Therefore, the estimate $\hat{\zeta}_{a, n}$ is asymptotically unbiased. 

To complete the proof of consistency, we will show that $\var[\hat \zeta_{a, n} ]\rightarrow 0$, as $n\rightarrow\infty$. To begin with, observe that 
\begin{align*}
    \var[\hat\zeta_{a, n}] = T_{n}^{(1)} + T_{n}^{(2)} , 
\end{align*} 
where 
\begin{align}    
     T_{n}^{(1)} & := \frac{1}{n} \var \left[\Theta^2(\Tilde{\mathbb P}_{\bm X| \bm Y_1, \bm Z_1},\Tilde{\mathbb P}_{\bm X| \bm Z_1}) a(\bm Y_1, \bm Z_1) \right] , \nonumber \\ 
   T_n^{(2)} & : =  \frac{n(n-1)}{n^2} \text{Cov}[\Theta^2(\Tilde{\mathbb P}_{\bm X| \bm Y_1, \bm Z_1},\Tilde{\mathbb P}_{\bm X| \bm Z_1}) a(\bm Y_1, \bm Z_1),\Theta^2(\Tilde{\mathbb P}_{\bm X| \bm Y_2, \bm Z_2},\Tilde{\mathbb P}_{\bm X| \bm Z_2}) a( \bm Y_2, \bm Z_2) ] . \nonumber 
    \end{align}
Using $|\Theta^2(\Tilde{\mathbb P}_{\bm X| \bm Y_1, \bm Z_1},\Tilde{\mathbb P}_{\bm X| \bm Z_1})| \leq 2$ and the boundedness of $a(\bm y,\bm z)$, it follows that $$\var[\Theta^2(\Tilde{\mathbb P}_{\bm X| \bm Y_1, \bm Z_1},\Tilde{\mathbb P}_{\bm X| \bm Z_1}) a(\bm Y_1, \bm Z_1)] = O(1),$$ and hence, $T_{n}^{(1)} \rightarrow 0$. Therefore, it suffices to show that $T_{n}^{(2)}\rightarrow 0$. For this, using boundedness note that 
\begin{align*} 
\E\left[ \left( \Theta^2(\Tilde{\mathbb P}_{\bm X| \bm Y_1, \bm Z_1},\Tilde{\mathbb P}_{\bm X| \bm Z_1}) a(\bm Y_1, \bm Z_1) \Theta^2(\Tilde{\mathbb P}_{\bm X| \bm Y_2, \bm Z_2},\Tilde{\mathbb P}_{\bm X| \bm Z_2}) a(\bm Y_2, \bm Z_2) \right)^2 \right] \leq 16 \left( \E\left[ a^2(\bm Y_1, \bm Z_1) a^2( \bm Y_2, \bm Z_2)\right] \right) <\infty.
\end{align*}
Hence, \eqref{eq:XYZestimatepf}, uniform integrability, and DCT gives, 
\begin{align*} 
&\lim_{n\rightarrow\infty} \mathbb E\left[\Theta^2(\Tilde{\mathbb P}_{\bm X| \bm Y_1, \bm Z_1},\Tilde{\mathbb P}_{\bm X| \bm Z_1}) a(\bm Y_1, \bm Z_1) \Theta^2(\Tilde{\mathbb P}_{\bm X| \bm Y_2, \bm Z_2},\Tilde{\mathbb P}_{\bm X| \bm Z_2}) a( \bm Y_2, \bm Z_2) \right]  =  \zeta_a(\bm X, \bm Y | \bm Z)^2 . 
\end{align*}
Combining the above with \eqref{eq:PXYZ} gives, $T_{n}^{(2)} \rightarrow 0$. This implies, $\var[\hat\zeta_{a, n}] \rightarrow 0$, as $n\rightarrow\infty$, which together with \eqref{eq:PXYZ} completes the proof of the Theorem \ref{consistency}. \hfill $\Box$

\section{Proof of Theorem \ref{large-sam-dist-1}}
\label{sec:asymptoticyzpf}

\begin{itemize}
    \item[(1)] 
Note that under $H_0: \bm X\indpt \bm Y|\bm Z$, the core function of the ball divergence $\phi'$ (as in \eqref{eq:core-function}) is conditionally degenerate of order 1, in the sense of Definition \ref{def:Vphi}. Hence, by Proposition \ref{prop:main-8} (3) with $r_{1}=r_{2}=4$ and $k=1$, under $H_0$, 
    $$nh_{1}^{d_Y+d_Z}\Theta^2(\Tilde{\P}_{\bm X|\bm Y = \bm y, \bm Z = \bm z }, \Tilde{\P}_{\bm X|\bm Z=\bm z}) \stackrel{D}{\rightarrow} {4 \choose 2} \int \phi_{2, 0}(\bm x, \bm x') \mathrm d \mathbb G_{\P_{\bm X|\bm Y= \bm y, \bm Z = \bm z}}(\bm x)\mathrm d \mathbb G_{\P_{\bm X|\bm Y= \bm y, \bm Z = \bm z}}(\bm x') ,$$
    where $\phi_{2, 0}$ is the (2, 0)-th order Hoeffding's decomposition of $\phi$. Recalling \eqref{eq:xrab}, it easy to see that 
$\phi_{2, 0}(\bm x, \bm x') = Q_0(\bm x, \bm x'; \bm y,\bm z)$,  
where $Q_0$ is as defined in Theorem \ref{large-sam-dist-1}.

    \item[(2)] Under $H_1:\bm X\indpt\bm Y|\bm Z$, the core function of the ball divergence $\phi'$ is conditionally non-degenerate.  Hence, by Proposition \ref{prop:main-8} (2), 
    \begin{align*}
    & \sqrt{nh_{1}^{d_Y+d_Z}}\big(\Theta^2(\Tilde{\P}_{\bm X|\bm Y = \bm y, \bm Z = \bm z }, \Tilde{\P}_{\bm X|\bm Z=\bm z})-\Theta^2(\P_{\bm X|\bm Y = \bm y, \bm Z = \bm z}, \P_{\bm X|\bm Z=\bm z})\big) \nonumber \\ 
    & \stackrel{D}{\rightarrow} 4 \int \phi_{1, 0}(\bm x) \mathrm d \mathbb G_{\P_{\bm X|\bm Y= \bm y, \bm Z = \bm z}}(\bm x) . 
    \end{align*}
    From \eqref{eq:xrab},we have $\phi_{1, 0}(\bm x) = Q_1(\bm x; \bm y,\bm z)$, where $Q_1$ is as defined in Theorem \ref{large-sam-dist-1}. 
    
\end{itemize}

\section{Proof of Theorem \ref{large-sam-dist-2}} 
\label{sec:asymptoticpf}

\subsection{Proof of Theorem \ref{large-sam-dist-2} (1)} 

Recall from \eqref{eq:weightXYZ}    
$$\zeta^*(\bm X,\bm Y|\bm Z) = \E[\Theta^2(\P_{\bm X|\bm Y,\bm Z},\P_{\bm X|\bm Z})p_{\bm Z}^4(\bm Z)p_{\bm Y,\bm Z}^4(\bm Y,\bm Z)].$$
    Since $\Theta^2(\P_{\bm X| \bm Y,\bm Z},\P_{\bm X|\bm Z})\geq 0$ we have $\zeta^*(\bm X,\bm Y|\bm Z)\geq 0$. The equality holds if and only if $\Theta^2(\P_{\bm X| \bm Y,\bm Z},\P_{\bm X|\bm Z})=0$ almost surely, that is, if and only if $\P_{\bm X| \bm Y,\bm Z}=\P_{\bm X|\bm Z}$ almost surely. This proves the result in Theorem \ref{large-sam-dist-2} (1).

\subsection{Proof of Theorem \ref{large-sam-dist-2} (2)} 
    
   Recall the definition of $\hat{\zeta}_{n}^{*}$ from \eqref{eq:XYZestimateU}. The core function $\varphi_n$ in \eqref{eq:wXYZ} can be expressed as: 
   \begin{align*}
    & \varphi_{n}\big((\bm X_1,\bm Y_1,\bm Z_1),(\bm X_{2},\bm Y_{2},\bm Z_{2}),\ldots, (\bm X_9,\bm Y_9,\bm Z_9)\big)\\
        & = \frac{1}{ h_{1}^{4(d_Y+d_Z)} h_{2}^{4d_Z} } K_{12}\cdots K_{15}K_{16}'\cdots K_{19}' \phi'(\bm X_2,\ldots, \bm X_5; \bm X_6,\ldots, \bm X_9),
    \end{align*}
   where $K_{ts} = K(\frac{\|(\bm Y_t,\bm Z_t)-(\bm Y_s,\bm Z_s)\|}{h_{1}})$ and $K_{ts}' = K(\frac{\|\bm Z_t-\bm Z_s\|}{h_{2}})$. For notational convenience, let $\bm U_i:=(\bm X_i,\bm Y_i,\bm Z_i)$. Now note,
    $$\E[\hat{\zeta}_{n}^{*}] = \E\left[\varphi_{n}\big(\bm U_1,\bm U_{2},\ldots, \bm U_9\big)\right]$$
    and
    \begin{align}
        \label{eq:varexp}
            \var[\hat{\zeta}_{n}^{*}] = {n\choose 9}^{-1}\sum_{c=1}^9 {9\choose c}{n-9\choose 9-c} \var\left[\varphi_{n}^{(c)}(\bm U_1,\ldots, \bm U_c)\right] , 
    \end{align}
    where
    \begin{align}
        \varphi_{n}^{(c)}(\bm u_1,\ldots, \bm u_c) = \E\left[\frac{1}{9!}\sum_{\pi\in S_9}\varphi_{n}(\bm U_{\pi(1)},\ldots, \bm U_{\pi(9)})|\bm U_1=\bm u_1,\ldots, \bm U_c = \bm u_c \right] , 
        \label{eq:hoeffding-projection}
    \end{align}
    is the $c$-th order Hoeffding's projection of $\varphi_{n}$ (see Section 1.3 in \cite{lee} for the derivation of the above expressions). By Lemma \ref{lm:aux-3.5} we have, 
    \begin{align}\label{eq:expectationXYZ}
    \E[\hat{\zeta}_{n}^{*}]\rightarrow \zeta^*(\bm X,\bm Y|\bm Z).
    \end{align} 
    Therefore, the result in Theorem \ref{large-sam-dist-2} (2) will follow if we can show $\var[\hat{\zeta}_{n}^{*}] \rightarrow 0$, as $n\rightarrow\infty$. This is established in the following lemma: 
    
    \begin{lemma} For $1 \leq c \leq 9$,  
    \begin{align} 
    \frac{1}{n^{c}} \var[\varphi_{n}^{(c)}(\bm U_1,\ldots, \bm U_c)] \rightarrow 0 , 
     \label{eq:var-order}
    \end{align}
    as $n\rightarrow\infty$. Consequently, from \eqref{eq:varexp}, $\var[\hat{\zeta}_{n}^{*}] \rightarrow 0$. 
    \end{lemma}
    
    \begin{proof} Suppose $\bm U = (\bm X, \bm Y, \bm Z) $ is distributed with density $p_{\bm X, \bm Y, \bm Z}$. For $c=1$, 
    \begin{align*}
        \var[\varphi_{n}^{(1)}(\bm U)] & = \E[\varphi_{n}^{(1)}(\bm U)^2] - (\E[\varphi_{n}^{(1)}(\bm U)])^2.
    \end{align*} 
    By Lemma \ref{lm:aux-3.5}, 
    $(\E[\varphi_{n}^{(1)}(\bm U)])^2 = (\E[\hat{\zeta}_{n}^{*}])^2\rightarrow \big(\zeta^*(\bm X, \bm Y|\bm Z)\big)^2$,  as $n\rightarrow\infty$. 
    Therefore, we need to find the rate of convergence of $\E[\varphi_{n}^{(1)}(\bm U)^2]$. Note that, from \eqref{eq:hoeffding-projection}, we can write,
    \begin{align} \label{eq:varianceT}
        \varphi_{n}^{(1)}(\bm u) & = \frac{1}{9}\Big\{\E[\varphi_{n}(\bm u, \bm U_2,\ldots, \bm U_9)] + 4 \E[\varphi_{n}(\bm U_1, \bm u,\bm U_3,\ldots, \bm U_9)] + 4 \E[\varphi_{n}(\bm U_1, \bm U_2,\ldots, \bm U_5, \bm u, \bm U_7,\ldots, \bm U_9)]\Big\} \nonumber \\
        & := \frac{1}{9}\Big\{\vartheta_1^{(1)}(\bm u) + 4 \vartheta_{2}^{(1)}(\bm u)+ 4 \vartheta_3^{(1)} (\bm u)\Big\} . 
    \end{align} 
    By Lemma \ref{lm:aux-3.5},
    \begin{align}
        \E\left[\vartheta_{1}^{(1)}(\bm U)^2\right] & = \E\left[\E\big[\varphi_{n}(\bm U, \bm U_2,\ldots, \bm U_9)\big]\E\big[\varphi_{n}(\bm U, \bm U_2',\ldots, \bm U_9')\big]\right]\tag*{\text{(where $\bm U, \bm U_2,\ldots, \bm U_9,\bm U_2',\ldots, \bm U_9'$ are i.i.d.)}} \nonumber \\
        & =  \E\left[\varphi_{n}(\bm U, \bm U_2,\ldots, \bm U_9)\varphi_{n}(\bm U, \bm U_2',\ldots, \bm U_9')\right] \nonumber\\
        & \rightarrow \int \left[\E\left[\phi'(\bm X_2,\ldots, \bm X_5; \bm X_6,\ldots, \bm X_9)\right] p_{\bm Z}(\bm z)^4 p_{\bm Y,\bm Z}(\bm y,\bm z)^4\right]^2 p_{\bm Y,\bm Z}(\bm y,\bm z)\mathrm d\bm y \mathrm d\bm z , 
        \label{eq:varianceU}
    \end{align} 
    where the expectation is with respect to $\bm X_2,\ldots, \bm X_5$ i.i.d. from $\mathbb P_{\bm X | \bm Y=\bm y, \bm Z = \bm z}$ and $\bm X_6,\ldots, \bm X_9$ i.i.d. from $\P_{\bm X| \bm Z =\bm z}$. Similarly, it can shown that $\E[\vartheta_{2}^{(1)}(\bm U)^2] = O(1)$ and $\E[\vartheta_{3}^{(1)}(\bm U)^2] = O(1)$. Hence, from \eqref{eq:varianceT} and using $(a+b+c)^2\leq 3(a^2+b^2+c^2)$, it follows that     
     $\E[\varphi_{n}^{(1)}(\bm U)^2] = O(1)$. Hence, $\frac{1}{n}\var[\varphi_{n}^{(1)}(\bm U)] \rightarrow 0$, as $n\rightarrow\infty$. 
    
    Next, suppose $c=2$. Note from \eqref{eq:hoeffding-projection} that,
    \begin{align}
       \varphi_{n}^{(2)}(\bm u,\bm u') = \frac{1}{{72}} \Big\{ \vartheta_{1}^{(2)}(\bm u,\bm u') + \vartheta_{2}^{(2)}(\bm u,\bm u') \Big\} , 
 \label{eq:second-projection}
   \end{align}
    where
    \begin{align}
        \vartheta_{1}^{(2)}(\bm u,\bm u') := &  4 \vartheta_{1, 1}^{(2)}(\bm u,\bm u') + 4 \vartheta_{1, 2}^{(2)}(\bm u,\bm u') + 16 \vartheta_{1, 3}^{(2)}(\bm u,\bm u') + 6 \vartheta_{1, 4}^{(2)}(\bm u,\bm u') + 6 \vartheta_{1, 5}^{(2)}(\bm u,\bm u') ,         \label{eq:projectionU} 
    \end{align} 
    where 
    \begin{align*}
    \vartheta_{1, 1}^{(2)}(\bm u,\bm u') & := \E\left[\varphi_{n}(\bm u, \bm u', \bm U_3,\ldots, \bm U_9)\right] , \\
     \vartheta_{1, 2}^{(2)}(\bm u,\bm u') & := \E\left[\varphi_{n}(\bm u, \bm U_2,\ldots, \bm U_5; \bm u', \bm U_7,\ldots,  \bm U_9)\right] , \nonumber \\
        \vartheta_{1, 3}^{(2)}(\bm u,\bm u') & :=  \E\left[\varphi_{n}(\bm U_1, \bm u,\bm U_3,\ldots, \bm U_5; \bm u', \bm U_7,\ldots,  \bm U_9)\right] , \\ 
         \vartheta_{1, 4}^{(2)}(\bm u,\bm u') & := \E\left[\varphi_{n}(\bm U_1, \bm u, \bm u', \bm U_4, \bm U_5; \bm U_6, \ldots,  \bm U_9)\right] ,\nonumber \\
      \vartheta_{1, 5}^{(2)}(\bm u,\bm u') & :=  \E\left[\varphi_{n}(\bm U_1, \bm U_2, \ldots, \bm U_5; \bm u, \bm u', \bm U_8,  \bm U_9)\right]  ; 
\end{align*}
        and $\vartheta_{2}^{(2)}(\bm u,\bm u') = \vartheta_{1}^{(2)}(\bm u',\bm u)$. For $\bm U, \bm U'$ are i.i.d. with density $p_{\bm X, \bm Y, \bm Z}$, it follows from \eqref{eq:H2pfsmall}, \eqref{eq:H24pf}, and \eqref{eq:H35pfsmall},  
        \begin{align*}
        & h_{1}^{(d_Y+d_Z)}\var[\varphi_{n}^{(2)}(\bm U,\bm U')] \rightarrow \tau_2^2 , 
        \end{align*} 
  for some constant $\tau_2^2 > 0$. Hence, $\frac{1}{n^2}\var[\varphi_{n}^{(2)}(\bm U, \bm U')] \rightarrow 0$, as $n\rightarrow\infty$. 

    Similarly, for any $c=3,\ldots, 9$ it can be shown that there exists a constant $\tau_c^2 > 0$ such that 
    \begin{align}
     h_{1}^{(c-1)(d_Y+d_Z)}\var[\varphi_{n}^{(c)}(\bm U_1,\ldots, \bm U_c)]\rightarrow \tau_c^2 , 
     \label{eq:var-order}
    \end{align}
    as $n\rightarrow\infty$. This implies $\frac{1}{n^{c}} \var[\varphi_{n}^{(c)}(\bm U_1,\ldots, \bm U_c)] \rightarrow 0$, as $n\rightarrow\infty$. 
    \end{proof}
        
        The above lemma together with \eqref{eq:expectationXYZ} completes the proof of Theorem \ref{large-sam-dist-2} (2).

\subsection{Proof of Theorem \ref{large-sam-dist-2} (3)}

        To begin with, note that $\hat{\zeta}_{n}^{*}$ is a $U$-statistics with kernel $\varphi_{n}$ of degree $9$. Let $\delta_{n}^{(1)}(\bm u_1) = \varphi_{n}^{(1)}(\bm u_1)$ and, for $2 \leq c \leq 9 $, 
  \begin{align}\label{eq:hc}
  \delta_{n}^{(c)}(\bm u_1,\ldots, \bm u_c) = \varphi_{n}^{(c)}(\bm u_1,\ldots, \bm u_c) - \sum_{a=1}^{c-1}\sum_{1\leq i_1\leq \cdots \leq i_a\leq c} \delta_{n}^{(a)}(\bm u_{i_1},\ldots,\bm u_{i_a}) , 
  \end{align} 
   where $\varphi_{n}^{(c)}(\bm u_1,\ldots, \bm u_c)$ is as defined in \eqref{eq:hoeffding-projection}. From Section 1.6 in \cite{lee} we have,
   $$\hat{\zeta}_{n}^{*} = \sum_{c=1}^9 {9\choose c} \Delta_{n}^{(c)} , $$
   where $\Delta_{n}^{(c)} = {n\choose c}^{-1}\sum_{1\leq i_1\leq \cdots \leq i_c \leq n} \delta_{n}^{(c)}(\bm U_{i_1},\ldots, \bm U_{i_c})$ satisfies the following properties: For $1 \leq c \leq 9$, 
   \begin{itemize}
       \item[(a)] The $U$-statistics $\Delta_{n}^{(1)}, \Delta_{n}^{(2)},\ldots, \Delta_{n}^{(9)}$ are uncorrelated.
       \item[(b)] $\var[\Delta_{n}^{(c)}] = {n\choose c}^{-1} \var[\delta_{n}^{(c)}(\bm U_1,\ldots, \bm U_c)]$.
       \item[(c)] $\var[\delta_{n}^{(c)}(\bm U_1,\ldots, \bm U_c)] = \sum_{a=1}^c (-1)^{c-a}{c\choose a} \var[\varphi_{n}^{(a)}(\bm U_1,\ldots, \bm U_a)]$, for $1 \leq c \leq 9$.
   \end{itemize}
   
   We begin by showing that $\delta_{n}^{(1)}$ is asymptotically negligible. 
   
   \begin{lemma}\label{lemma:Hsmall} Under $H_0: \bm X \indpt \bm Y|\bm Z$, $\delta_n^{(1)}( \bm u) = O(h_1^4)$, almost surely. Moreover, $$nh_{1}^{(d_Y+d_Z)/2} \Delta_{n}^{(1)} \stackrel{L_2} \rightarrow 0.$$
    \end{lemma}

\begin{proof} 
Note that 
   \begin{align*}
    \delta_{n}^{(1)}(\bm u) & =\varphi_n^{(1)}(\bm u) = \frac{1}{9}\Big\{\vartheta_1^{(1)}(\bm u) + 4 \vartheta_{2}^{(1)}(\bm u)+ 4 \vartheta_3^{(1)} (\bm u)\Big\} , 
   \end{align*} 
   where $\vartheta_1^{(1)}$, $\vartheta_{2}^{(1)}$, and $\vartheta_3^{(1)}$ are as defined in \eqref{eq:varianceT}. 
    Note that 
\begin{align} 
\vartheta_1^{(1)}(\bm u) & = \E[ \varphi_n(\bm u,\bm U_2,\ldots,\bm U_5; \bm U_6,\ldots, \bm U_9)] \nonumber \\
     & = \frac{1}{ h_{1}^{4(d_Y+d_Z)} h_{2}^{4d_Z} }\E\left[\prod_{s=2}^5 w_{\bm Y_s,\bm Z_s}(\bm y,\bm z) \prod_{s=6}^9 w_{\bm Z_s}(\bm z)\phi'(\bm X_2,\ldots, \bm X_9)\right] \nonumber \\ 
    & = p_{\bm Y, \bm Z} ( \bm y, \bm z)^{4}  p_{\bm Z} (\bm z)^{4} \E[ \phi'(\bm X_2, \bm X_3,\ldots, \bm X_9) \mid \{ \bm Y_s= \bm y \}_{2 \leq s \leq 5}, \{\bm Z_s= \bm z \}_{2 \leq s \leq 9} ] +  O(h_{1}^2 h_{2}^2)  \tag*{ (by Remark \ref{remark:h})} \nonumber \\ 
     & = O(h_{1}^2 h_{2}^2) = O(h_{1}^4) , \nonumber 
\end{align} 
since $\phi'$ is conditionally degenerate of order 1 under $H_0$ and $h_{2}/h_{1}\rightarrow 0$ by assumption. Similarly, it can be shown that $\vartheta_2^{(1)}(\bm u) = O(h_{1}^4)$ and $\vartheta_3^{(1)}(\bm u) = O(h_{1}^4)$, almost surely. This implies, $\delta_n^{(1)}( \bm u) = O(h_{1}^4)$. 

Also, we have from Lemma \ref{lm:aux-3.5}, 
\begin{align} 
\E[\vartheta_1^{(1)}(\bm U)] & = \E[ \varphi_n (\bm U,\bm U_2,\ldots,\bm U_5; \bm U_6,\ldots, \bm U_9)] \nonumber \\
    & \rightarrow  \int p_{\bm Y, \bm Z} ( \bm y, \bm z)^{5}  p_{\bm Z} (\bm z)^{4} \E[ \phi'(\bm X_2, \bm X_3,\ldots, \bm X_9) \mid \{ \bm Y_s= \bm y \}_{2 \leq s \leq 5}, \{\bm Z_s= \bm z \}_{2 \leq s \leq 9} ] \mathrm d \bm y \mathrm d \bm z +  O(h_{1}^2 h_{2}^2)   \nonumber \\ 
     & = O(h_{1}^2 h_{2}^2) = O(h_{1}^4) . \nonumber 
\end{align} 
Similarly, $\E[\vartheta_2^{(1)}(\bm U)] = O(h_{1}^4)$ and $\E[\vartheta_3^{(1)}(\bm U)] = O(h_{1}^4)$. Hence, 
\begin{align}\label{eq:expectationH1}
nh_{1}^{(d_Y+d_Z)/2} \E[\Delta_n^{(1)}( \bm U)] = nh_{1}^{(d_Y+d_Z)/2} \E[ \delta_n^{(1)}( \bm U) ] = O(nh_{1}^{(d_Y+d_Z)/2 + 4}) \rightarrow 0 , 
\end{align}
since $n^2 h_{1}^{d_Y+d_Z+4}\rightarrow 0$.  
Next, from \eqref{eq:varianceU} and Lemma \ref{lm:aux-3.5} we have $\E[\vartheta_1^{(1)}(\bm U)^2] = O(h_{1}^8)$. Similarly, it can be shown that $\E[\vartheta_2^{(1)}(\bm U)^2] = O(h_{1}^8)$ and $\E[\vartheta_3^{(1)}(\bm U)^2] = O(h_{1}^8)$. This implies, 
\begin{align}\label{eq:varianceH1}
n^2h_{1}^{(d_Y+d_Z)}  \mathrm{Var}[ \Delta_n^{(1)}( \bm U)] = nh_{1}^{(d_Y+d_Z)}   \mathrm{Var}[ \delta_n^{(1)}( \bm U) ] = O(nh_{1}^{(d_Y+d_Z + 8)} )  \rightarrow 0, 
\end{align}
since $n^2 h_{1}^{d_Y+d_Z+4}\rightarrow 0$. Combining \eqref{eq:expectationH1} and \eqref{eq:varianceH1} shows that $nh_{1}^{(d_Y+d_Z)/2} \Delta_{n}^{(1)} \stackrel{L_2} \rightarrow 0$. 
 \end{proof} 
   
   The above lemma shows that $\hat{\zeta}_{n}^{*}$ is an approximately first-order degenerate $U$-statistics. Now, we consider the second-order Hoeffding's projection. 
   
      \begin{lemma}\label{lemma:hn} Under $H_0: \bm X \indpt \bm Y|\bm Z$, 
             \begin{align}\label{eq:distributionhn}
             nh_{1}^{(d_Y+d_Z)/2} \Delta_{n}^{(2)}\stackrel{D}{\rightarrow} N(0, 144  \sigma_2^2 \mathcal C_2(K)) , 
             \end{align}
              where $\sigma_2^2 = \sigma^2$ and $\mathcal C_2(K) = \mathcal C(K)$ are as in the statement of Theorem \ref{large-sam-dist-2}. 
         \end{lemma}

\begin{proof}[Proof of Lemma \ref{lemma:hn}]  
We have from \eqref{eq:hc}, $\delta_{n}^{(2)}(\bm u, \bm u') = \varphi_{n}^{(2)}(\bm u, \bm u') -  \delta_{n}^{(1)}(\bm u) - \delta_{n}^{(1)}(\bm u')$. This means, 
   \begin{align}\label{eq:H2}
   \Delta_{n}^{(2)} & = \frac{1}{{n \choose 2}} \sum_{1 \leq i < j \leq n} \varphi_{n}^{(2)}(\bm U_i, \bm U_j) -  \frac{2}{n} \sum_{i=1}^n \delta_{n}^{(1)}(\bm U_{i}) \nonumber \\ 
   & = \frac{1}{{n \choose 2}} \sum_{1 \leq i < j \leq n} \varphi_{n}^{(2)}(\bm U_i, \bm U_j) -  o_P\left(\frac{1}{nh_{1}^{(d_Y+d_Z)/2}} \right) , 
\end{align} 
by Lemma \ref{lemma:Hsmall}. Hence, it suffices to derive the limiting distribution 
  of ${n \choose 2}^{-1} \sum_{1 \leq i < j \leq n} \varphi_{n}^{(2)}(\bm U_i, \bm U_j)$. Towards this, recall from \eqref{eq:second-projection}, 
   \begin{align*}
       \varphi_{n}^{(2)}(\bm u,\bm u') = \frac{1}{{72}} \Big\{ \vartheta_{1}^{(2)}(\bm u,\bm u') + \vartheta_{2}^{(2)}(\bm u,\bm u') \Big\} ,
   \end{align*}
where $\vartheta_{1}^{(2)}$ and $\vartheta_{2}^{(2)}$ are as defined in \eqref{eq:projectionU}.  Note that $\vartheta_{1}^{(2)}(\bm u, \bm u')$ has 5 terms. The proof of Lemma \ref{lemma:hn} now proceeds as follows: 
\begin{itemize}

\item First we show that the $U$-statistics corresponding to the first 2 terms in \eqref{eq:projectionU} are asymptotically negligible:   
\begin{align}\label{eq:H2pfsmall}
\frac{nh_{1}^{(d_Y+d_Z)/2}}{n(n-1)} \sum_{1 \leq i \ne j \leq n} \vartheta_{1, 1}^{(2)}(\bm U_i, \bm U_j) = o_{L_2}(1) , ~ \frac{nh_{1}^{(d_Y+d_Z)/2}}{n(n-1)} \sum_{1 \leq i \ne j \leq n} \vartheta_{1, 2}^{(2)}(\bm U_i, \bm U_j) = o_{L_2}(1) , 
\end{align} 
where $o_{L_2}(1)$ is a term that converges to zero in $L_2$, as $n \rightarrow \infty$.

\item Next, we show the following about the $U$-statistics corresponding to the last three terms in  \eqref{eq:projectionU}: 
\begin{align}\label{eq:H24pf}
\frac{nh_{1}^{(d_Y+d_Z)/2}}{n(n-1)}\sum_{1\leq i \ne j \leq n} \vartheta_{1, 4}^{(2)} (\bm U_i,\bm U_j) & \rightarrow N(0, \sigma_2^2 \mathcal C_2(K)) , 
\end{align} 
where $\sigma_2^2$ and $\mathcal C_2(K)$ are as in Lemma \ref{lemma:hn};   
and 
\begin{align}\label{eq:H35pf}
\frac{nh_{2}^{d_Z/2}}{n(n-1)}\sum_{1\leq i \ne j \leq n} \vartheta_{1, 3}^{(2)} (\bm U_i,\bm U_j) & \rightarrow N(0, \tau_3^2), ~ \frac{nh_{2}^{d_Z/2}}{n(n-1)}\sum_{1\leq i \ne j \leq n} \vartheta_{1, 5}^{(2)} (\bm U_i,\bm U_j) \rightarrow N(0, \tau_5^2) , 
\end{align} 
for some positive constants $\tau_3^2$ and $\tau_5^2$. Here, the convergence in \eqref{eq:H24pf} and \eqref{eq:H35pf} hold in distribution and in the first 2 moments. 
\end{itemize}

The proofs of the above results are given below. To complete the proof of Lemma \ref{lemma:hn}, observe that \eqref{eq:H35pf} implies, since $h_2/h_1 \rightarrow 0$, 
\begin{align}\label{eq:H35pfsmall}
\frac{nh_{1}^{(d_Y+d_Z)/2}}{n(n-1)} \sum_{1 \leq i \ne j \leq n} \vartheta_{1, 1}^{(3)}(\bm U_i, \bm U_j) = o_{P}(1) , ~ \frac{nh_{1}^{(d_Y+d_Z)/2}}{n(n-1)} \sum_{1 \leq i \ne j \leq n} \vartheta_{1, 2}^{(5)}(\bm U_i, \bm U_j) = o_{P}(1) . 
\end{align} 
Combining this with \eqref{eq:H2} and \eqref{eq:H24pf} the result in Lemma \ref{lemma:hn} follows. 

\begin{proof}[Proof of \eqref{eq:H2pfsmall}]  
    Recall the definition of $\vartheta_{1, 1}^{(2)}$ from \eqref{eq:projectionU}. Then by Lemma \ref{lm:aux-3.5} and Remark \ref{remark:h}, 
       \begin{align*}
       & \frac{1}{n(n-1)}\sum_{1 \leq i \ne j \leq n} \E[ \vartheta_{1, 1}^{(2)}(\bm U_i, \bm U_j) ] \nonumber \\ 
       & = \E [\vartheta_{1, 1}^{(2)}(\bm U_1, \bm U_2) ] \nonumber \\
       & = \frac{1}{h_{2}^{4d_Z}h_{1}^{4(d_Y+d_Z)}} \int \E\left[  \prod_{s=2}^5 w_{\bm Y_s,\bm Z_s}(\bm y_1, \bm z_1) \prod_{s=6}^9 w_{\bm Z_s}(\bm z_1)\phi'(\bm X_2, \bm X_3,\ldots, \bm X_9)\right] p_{\bm Y, \bm Z}(\bm y_1, \bm z_1) \mathrm d \bm y_1 \mathrm d \bm z_1 \\ 
& \rightarrow  \int p_{\bm Y, \bm Z} ( \bm y, \bm z)^{5}  p_{\bm Z} (\bm z)^{4} \E[ \phi'(\bm X_2, \bm X_3,\ldots, \bm X_9) \mid \{ \bm Y_s= \bm y \}_{2 \leq s \leq 5}, \{\bm Z_s= \bm z \}_{2 \leq s \leq 9} ]  \mathrm d \bm y \mathrm d \bm z +  O(h_{1}^2 h_{2}^2)  \nonumber \\ 
     & = O(h_{1}^2 h_{2}^2) = O(h_{1}^4) . \nonumber 
\end{align*} 
This means, 
\begin{align}\label{eq:H11expectation}
\frac{nh_{1}^{(d_Y+d_Z)/2}}{n(n-1)} \sum_{1 \leq i \ne j \leq n} \E [ \vartheta_{1, 1}^{(2)}(\bm U_i, \bm U_j) ] = O(nh_{1}^{(d_Y+d_Z)/2 +4} ) \rightarrow 0.
\end{align}

Next, we compute the variance of $\sum_{1 \leq i \ne j \leq n} \vartheta_{1, 1}^{(2)}(\bm U_i, \bm U_j) $. As before, by Lemma \ref{lm:aux-3.5} and Remark \ref{remark:h}, 
\begin{align}
& \frac{h_{1}^{(d_Y+d_Z)}}{n^2}\sum_{1 \leq i \ne j \leq n} \var[ \vartheta_{1, 1}^{(2)}(\bm U_i, \bm U_j) ] \nonumber \\ 
       & \lesssim h_{1}^{(d_Y+d_Z)} \E[ \vartheta_{1, 1}^{(2)}(\bm U_1, \bm U_2) ^2 ] \nonumber \\ 
       & = h_{1}^{(d_Y+d_Z)} \E\left[\E\big[\varphi_{n}(\bm U_1, \bm U_2, \bm U_3, \ldots, \bm U_9)\big]\E\big[\varphi_{n}(\bm U_1, \bm U_2, \bm U_3', \ldots, \bm U_9')\big]\right]\tag*{\text{(where $\bm U_1,\ldots, \bm U_9,\bm U_3',\ldots, \bm U_9'$ are i.i.d.)}} \nonumber \\
        & = h_{1}^{(d_Y+d_Z)} \E\left[\varphi_{n}(\bm U_1, \bm U_2, \bm U_3, \ldots, \bm U_9)\varphi_{n}(\bm U_1, \bm U_2, \bm U_3', \ldots, \bm U_9')\right] \nonumber\\
        & \rightarrow  \left\{ C_K \int \left(\E\left[ \gamma'(\bm X_2, \bm X_3, \ldots, \bm X_9; \bm X_3', \ldots, \bm X_9') \right] p_{\bm Z}(\bm z)^8 p_{\bm Y,\bm Z}(\bm y,\bm z)^7 \right) p_{\bm Y,\bm Z}(\bm y,\bm z) \mathrm d\bm y \mathrm d\bm z  + O(h_1^2 h_2^2) \right\} \nonumber \\ 
           & = O(h_{1}^2 h_{2}^2) = O(h_{1}^4) \rightarrow 0 , 
     \label{eq:H11}
\end{align} 
where $\gamma'$ is the symmetrized version of the function $\gamma : (\R^{d_X})^{15} \rightarrow \R$, 
$$\gamma(\bm x_2, \bm x_3, \ldots, \bm x_9; \bm x_3', \ldots, \bm x_9') :=\phi^\prime(\bm x_2, \bm x_3, \ldots, \bm x_9) \phi^\prime(\bm x_2, \bm x_3', \ldots, \bm x_9'),$$
the expectation is with respect to $\bm X_2,\ldots, \bm X_5, \bm X_3', \ldots, \bm X_5'$ i.i.d. from $\mathbb P_{\bm X | \bm Y=\bm y, \bm Z = \bm z}$ and $\bm X_6,\ldots, \bm X_9, \bm X_6',\ldots, \bm X_9'$ i.i.d. from $\P_{\bm X| \bm Z =\bm z}$, and $C_K$ is some constant depending on $K$. Since $\phi^\prime$ is conditionally first order degenerate under $H_0$, 
$$\E\left[ \gamma'(\bm X_2, \bm X_3, \ldots, \bm X_9; \bm X_3', \ldots, \bm X_9') \right] = 0 , $$
which explains the last step in \eqref{eq:H11variance}. Now, we consider the covariance terms, 
\begin{align}
& \frac{h_{1}^{(d_Y+d_Z)}}{n^2}\sum_{1 \leq i \ne j \ne k \leq n} \mathrm{Cov}[ \vartheta_{1, 1}^{(2)}(\bm U_i, \bm U_j),  \vartheta_{1, 1}^{(2)}(\bm U_j, \bm U_k)  ] \nonumber \\ 
       & \lesssim n h_{1}^{(d_Y+d_Z)} \E[ \vartheta_{1, 1}^{(2)}(\bm U_1, \bm U_2) \vartheta_{1, 1}^{(2)}(\bm U_1, \bm U_3) ] \nonumber \\ 
       & = n h_{1}^{(d_Y+d_Z)} \E\left[\E\big[\varphi_{n}(\bm U_1, \bm U_2, \bm U_3', \ldots, \bm U_9')\big]\E\big[\varphi_{n}(\bm U_1, \bm U_3, \bm U_3'', \ldots, \bm U_9'')\big]\right]  \tag*{\text{(where $\bm U_1, \bm U_2, \bm U_3, \bm U_3', \ldots, \bm U_9', \bm U_3'', \ldots, \bm U_9''$ are i.i.d.)}} \nonumber \\ 
        & = n h_{1}^{(d_Y+d_Z)} \E\left[\varphi_{n}(\bm U_1, \bm U_2, \bm U_3', \ldots, \bm U_9') \varphi_{n}(\bm U_1, \bm U_3, \bm U_3'', \ldots, \bm U_9'') \right]  \nonumber \\
& = O(nh_1^{(d_Y+d_Z)+4}) \rightarrow 0, 
     \label{eq:H12}
\end{align} 
by Lemma \ref{lm:aux-3.5} and Remark \ref{remark:h}.  Collecting \eqref{eq:H11} and \eqref{eq:H12} it follows that
\begin{align}\label{eq:H11variance}
\frac{h_{1}^{(d_Y+d_Z)}}{n^2} \var\left[\sum_{1 \leq i \ne j \leq n} \vartheta_{1, 1}^{(2)}(\bm U_i, \bm U_j) \right] \rightarrow 0 . 
\end{align}
 From \eqref{eq:H11expectation} and \eqref{eq:H11variance} the first result in \eqref{eq:H2pfsmall} follows. The second result in \eqref{eq:H2pfsmall} can be proved similarly. 
\end{proof}

\begin{proof}[Proof of \eqref{eq:H24pf} and \eqref{eq:H35pf}]   

We will only prove \eqref{eq:H24pf}. The results in \eqref{eq:H35pf} can proved similarly. To begin with recall the definition of $\vartheta_{1, 4}^{(2)}$ from \eqref{eq:projectionU}. Then 
\begin{align}
& \vartheta_{1, 4}^{(2)}(\bm u, \bm u') \nonumber \\ 
& = \E\left[\varphi_{n}(\bm U_1, \bm u, \bm u', \bm U_4, \bm U_5; \bm U_6, \ldots,  \bm U_9)\right] \nonumber \\
       & = \frac{1}{h_1^{4(d_Y+d_Z)}h_2^{4d_Z}}\E\left[ w_{\bm y,\bm z}(\bm Y_1,\bm Z_1) w_{\bm y',\bm z'}(\bm Y_1,\bm Z_1)\prod_{s = 4}^5 w_{\bm Y_s, \bm Z_s}(\bm Y_1,\bm Z_1) \prod_{s = 6}^9 w_{\bm Z_s}(\bm Z_1) \phi'(\bm x,\bm x',\bm X_4, \ldots,\bm X_9)\right] , \nonumber  
       \end{align} 
       where $\bm u = (\bm x, \bm y, \bm z)$ and $\bm u' = (\bm x', \bm y', \bm z')$. Then note that, 
       \begin{align}
       & \E[\vartheta_{1, 4}^{(2)}(\bm U, \bm U')^2]\nonumber\\
       & = \frac{1}{h_1^{8(d_Y+d_Z)}h_2^{8d_Z}} \nonumber \\ 
        & ~~~~ \E\Big[ w_{\bm Y,\bm Z}(\bm Y_1,\bm Z_1) w_{\bm Y',\bm Z'}(\bm Y_1,\bm Z_1) \prod_{s = 4}^5 w_{\bm Y_s, \bm Z_s}(\bm Y_1,\bm Z_1) \prod_{s = 6}^9 w_{\bm Z_s}(\bm Z_1) \phi'(\bm X,\bm X',\bm X_4, \ldots,\bm X_9)\nonumber\\
       & ~~~~~~~~ w_{\bm Y,\bm Z}(\tilde{\bm Y}_1,\tilde{\bm Z}_1) w_{\bm Y',\bm Z'}(\tilde{\bm Y}_1,\tilde{\bm Z}_1)\prod_{s = 4}^5 w_{\tilde{\bm Y}_s, \tilde{\bm Z}_s}(\tilde{\bm Y}_1,\tilde{\bm Z}_1) \prod_{s = 6}^9 w_{\tilde{\bm Z}_s}(\tilde{\bm Z}_1) \phi'(\bm X,\bm X',\tilde{\bm X}_4, \ldots,\tilde{\bm X}_9)\Big] , \label{eq:U14pf}
       \end{align} 
       where $\{\tilde{\bm U}_i = (\tilde{\bm X}_i, \tilde{\bm Y}_i, , \tilde{\bm Z}_i) \}_{1 \leq i \leq 9}$ are i.i.d. with density $p_{\bm X, \bm Y, \bm Z}$ which is independent of  $\{ \bm U_i = ( \bm X_i, \bm Y_i, , \bm Z_i) \}_{1 \leq i \leq 9}$. To compute \eqref{eq:U14pf}, we first condition on $\bm U_1,\tilde{\bm U}_1,\bm U$, and $\bm U'$ to get, 
              \begin{align*}
          & \E[\vartheta_{1, 4}^{(2)}(\bm U, \bm U')^2|\bm U,\bm U', \bm U_1, \tilde{\bm U_1}]\\ 
          & = \frac{1}{h_1^{8(d_Y+d_Z)}h_2^{8d_Z}} w_{\bm Y,\bm Z}(\bm Y_1,\bm Z_1) w_{\bm Y',\bm Z'}(\bm Y_1,\bm Z_1) w_{\bm Y,\bm Z}(\tilde{\bm Y}_1,\tilde{\bm Z}_1) w_{\bm Y',\bm Z'}(\tilde{\bm Y}_1,\tilde{\bm Z}_1) \nonumber \\
         &  \int \Bigg\{ \prod_{s = 4}^5 w_{\bm y_s, \bm z_s}(\bm Y_1,\bm Z_1) \prod_{s = 4}^5 w_{\tilde{\bm y}_s, \tilde{\bm z}_s}(\tilde{\bm Y}_1,\tilde{\bm Z}_1) \prod_{s = 6}^9 w_{\bm z_s}(\bm Z_1) \prod_{s = 6}^9 w_{\tilde{\bm z}_s}(\bm Z_1) \phi'(\bm X,\bm X',\bm x_4, \ldots,\bm x_9) \phi'(\bm X,\bm X',\tilde{\bm x}_4, \ldots,\tilde{\bm x}_9) \\
          & ~~~~~~   \prod_{s=4}^5 p_{\bm X, \bm Y, \bm Z}(\bm x_s, \bm y_s, \bm z_s) \prod_{s=4}^5 p_{\bm X, \bm Y, \bm Z}(\tilde{\bm x}_s, \tilde{\bm y}_s, \tilde{\bm z}_s)  \prod_{s=6}^9 p_{\bm X, \bm Z}( \bm x_s, \bm z_s)  \prod_{s=6}^9 p_{\bm X, \bm Z}(\tilde{\bm x}_s, \tilde{\bm z}_s) \Bigg\} \mathrm d \mathcal V , 
          \end{align*} 
         where $\mathrm d \mathcal V := \prod_{s=4}^9 \mathrm d\bm x_s \prod_{s=4}^9 \mathrm d \tilde{\bm x}_s \prod_{s=4}^5 \mathrm d\bm y_s \prod_{s=4}^5 \mathrm d\tilde{\bm y}_s \prod_{s=4}^9 \mathrm d\bm z_s \prod_{s=4}^9 \mathrm d\tilde{\bm z}_s$. Now, by the change of variables: $$\bm y_s = \bm Y_1 + h_1 \bm u_s,~\bm z_s = \bm Z_1 + h_1 \bm v_s,~\tilde{\bm y}_s = \tilde{\bm Y}_1 + h_1 \tilde{\bm u}_s,~\tilde{\bm z}_s = \tilde{\bm Z}_1 + h_1 \tilde{\bm v}_s,$$ for $s \in \{4, 5\}$, and 
         $$\bm z_s = \bm Z_1 + h_2 \bm v_s,~\tilde{\bm z}_s = \tilde{\bm Z}_1 + h_2 \tilde{\bm v}_s,$$ 
         for $s \in \{6, 7, 8, 9\}$ we get, 
              \begin{align}
          & \E[\vartheta_{1, 4}^{(2)}(\bm U, \bm U')^2|\bm U,\bm U', \bm U_1, \tilde{\bm U_1}] \nonumber \\ 
          & = \frac{1}{h_1^{4(d_Y+d_Z)}} w_{\bm Y,\bm Z}(\bm Y_1,\bm Z_1) w_{\bm Y',\bm Z'}(\bm Y_1,\bm Z_1) w_{\bm Y,\bm Z}(\tilde{\bm Y}_1,\tilde{\bm Z}_1) w_{\bm Y',\bm Z'}(\tilde{\bm Y}_1,\tilde{\bm Z}_1) \nonumber \\
         &  \int \Bigg\{ \prod_{s = 4}^5 K( \| ( \bm u_s, \bm v_s )  \| ) \prod_{s = 4}^5 K( \| ( \tilde{\bm u}_s, \tilde{\bm v}_s )  \| ) \prod_{s = 6}^9 K( \| \bm v_s \| ) \prod_{s = 6}^9 K( \| \tilde{\bm v}_s \| )  \phi'(\bm X,\bm X',\bm x_4, \ldots,\bm x_9) \phi'(\bm X,\bm X',\tilde{\bm x}_4, \ldots,\tilde{\bm x}_9) \nonumber \\
          & ~~~~~~   \prod_{s=4}^5 p_{\bm X, \bm Y, \bm Z}(\bm x_s, \bm Y_1 + h_1 \bm u_s, \bm Z_1 + h_1 \bm v_s) \prod_{s=4}^5 p_{\bm X, \bm Y, \bm Z}(\tilde{\bm x}_s, \tilde{\bm Y}_1 + h_1 \tilde{\bm u}_s, \tilde{\bm Z}_1 + h_1 \tilde{\bm v}_s) \nonumber \\ 
          & ~~~~~~  \prod_{s=6}^9 p_{\bm X, \bm Z}( \bm x_s, \bm Z_1 + h _2 \bm u_s)  \prod_{s=6}^9 p_{\bm X, \bm Z}(\tilde{\bm x}_s, \tilde{\bm Z}_1 + h _2 \tilde{\bm u}_s) \Bigg\} \mathrm d \overline{\mathcal V} , \label{eq:U14gammapf}
\end{align} 
 where $\mathrm d \overline{\mathcal V} := \prod_{s=4}^9 \mathrm d\bm x_s \prod_{s=4}^9 \mathrm d \tilde{\bm x}_s \prod_{s=4}^5 \mathrm d\bm u_s \prod_{s=4}^5 \mathrm d\tilde{\bm u}_s \prod_{s=4}^9 \mathrm d\bm v_s \prod_{s=4}^9 \mathrm d\tilde{\bm v}_s$.      
Denote the integral \eqref{eq:U14gammapf} by $\gamma(\bm x,\bm x^\prime; \bm y_1,\bm z_1,\tilde{\bm y}_1,\tilde{\bm z}_1, h_1,h_2)$. Then taking expectation on both sides of \eqref{eq:U14gammapf} gives, 
        \begin{align*} 
        & \E[\vartheta_{1, 4}^{(2)}(\bm U, \bm U')^2] \nonumber \\ 
       & = \E\left[\frac{1}{h_1^{4(d_Y+d_Z)}} w_{\bm Y,\bm Z}(\bm Y_1,\bm Z_1) w_{\bm Y',\bm Z'}(\bm Y_1,\bm Z_1) w_{\bm Y,\bm Z}(\tilde{\bm Y}_1,\tilde{\bm Z}_1) w_{\bm Y',\bm Z'}(\tilde{\bm Y}_1,\tilde{\bm Z}_1)\gamma(\bm X,\bm X^\prime; \bm Y_1,\bm Z_1,\tilde{\bm Y}_1,\tilde{\bm Z}_1,h_1,h_2)\right]\\[2em]
       & = \frac{1}{h_1^{4(d_Y+d_Z)}} \int \Bigg\{ K\left(\frac{\|(\bm y,\bm z)-(\bm y_1,\bm z_1)\|}{h_1}\right) K\left(\frac{\|(\bm y',\bm z')-(\bm y_1,\bm z_1)\|}{h_1}\right) \\
      & ~~~~~~~~~~~~~~~~~~~~~~~~ K\left(\frac{\|(\bm y,\bm z)-(\tilde{\bm y}_1,\tilde{\bm z}_1)\|}{h_1}\right) K\left(\frac{\|(\bm y,\bm z)-(\tilde{\bm y}_1,\tilde{\bm z}_1)\|}{h_1}\right) \gamma(\bm x,\bm x^\prime; \bm y_1,\bm z_1,\tilde{\bm y}_1,\tilde{\bm z}_1 h_1,h_2)\\
      & ~~~~~~~~~~~~~~~~~~~~~~~~ p_{\bm X, \bm Y, \bm Z}(\bm x,\bm y,\bm z)p_{\bm X, \bm Y, \bm Z}(\bm x',\bm y',\bm z')p_{\bm X, \bm Y, \bm Z}(\bm x_1,\bm y_1,\bm z_1)p_{\bm X, \bm Y, \bm Z}(\tilde{\bm x}_1,\tilde{\bm y}_1,\tilde{\bm z}_1) \Bigg \} \mathrm d \mathcal T , 
      \end{align*} 
      where $\mathrm d \mathcal T = \mathrm d\bm x\mathrm d\bm y\mathrm d\bm z\mathrm d\bm x'\mathrm d\bm y'\mathrm d\bm z'\mathrm d \bm x_1\mathrm d\bm y_1\mathrm d\bm z_1\mathrm d\tilde{\bm x}_1\mathrm d\tilde{\bm y}_1\mathrm d\tilde{\bm z}_1$. By the change of variables: 
      $(\bm y_1,\bm z_1) = (\bm y, \bm z) + h_1 (\bm u_1, \bm v_1)$ and $(\tilde{\bm y}_1, \tilde{\bm z}_1) = (\tilde{\bm y}, \tilde{\bm z}) + h_1 (\tilde{\bm u}_1, \tilde{\bm v}_1)$ we get, 
      \begin{align*}
      & \E[\vartheta_{1, 4}^{(2)}(\bm U, \bm U')^2] \nonumber \\ 
      & = \frac{1}{h_1^{2(d_Y+d_Z)}} \int \Bigg \{ K\left(\|(\bm u_1,\bm v_1)\|\right) K\left(\frac{\|(\bm y',\bm z')-(\bm y,\bm z) - h_1(\bm u_1,\bm v_1)\|}{h_1}\right) \\
      & ~~~~~~~~ K\left(\|(\tilde{\bm v}_1,\tilde{\bm u}_1)\|\right) K\left(\frac{\|(\bm y',\bm z')-(\bm y,\bm z)-h_1(\tilde{\bm v}_1,\tilde{\bm u}_1)\|}{h_1}\right) \nonumber \\ 
      & ~~~~~~~~ \gamma(\bm x,\bm x^\prime;\bm y + h_1\bm u_1,\bm z+ h_1\bm v_1,\bm y+h_1\tilde{\bm u}_1,\bm z + h_1\tilde{\bm v}_1, h_1,h_2) \nonumber \\ 
      & ~~~~~~~~ p_{\bm X, \bm Y, \bm Z}(\bm x,\bm y,\bm z) p_{\bm X, \bm Y, \bm Z}(\bm x',\bm y',\bm z') p_{\bm X, \bm Y, \bm Z}(\bm x_1,\bm y + h_1 \bm v_1,\bm z+h_1\bm u_1) p_{\bm X, \bm Y, \bm Z}(\tilde{\bm x}_1,\bm y + h_1\tilde{\bm v}_1, \bm z + h_1\tilde{\bm u}_1) \Bigg \} \mathrm d \tilde{ \mathcal T} , 
      \end{align*}
      where $\mathrm d \tilde{ \mathcal T} = \mathrm d\bm x\mathrm d\bm y\mathrm d\bm z\mathrm d\bm x'\mathrm d\bm y'\mathrm d\bm z'\mathrm d \bm x_1\mathrm d\bm u_1\mathrm d\bm v_1\mathrm d\tilde{\bm x}_1\mathrm d\tilde{\bm u}_1\mathrm d\tilde{\bm v}_1$. Then by the change of variables: $(\bm y', \bm z') = (\bm y, \bm z) + h_1 (\bm u' , \bm v')$ we get, 
     \begin{align*} 
           &  \E[\vartheta_{1, 4}^{(2)}(\bm U, \bm U')^2] \nonumber \\ 
            & = \frac{1}{h_1^{(d_Y+d_Z)}} \int \Bigg \{ K\left(\|(\bm u_1,\bm v_1)\|\right) K\left(\|(\bm u',\bm v') - (\bm u_1,\bm v_1)\|\right)K\left(\|(\tilde{\bm u}_1,\tilde{\bm v}_1)\|\right) K\left(\|(\bm u',\bm v')-(\tilde{\bm u}_1,\tilde{\bm v}_1)\|\right) \\
      & ~~~~~~~~ \gamma(\bm x,\bm x^\prime; \bm y + h_1\bm u_1,\bm z+ h_1\bm v_1,\bm y+h_1\tilde{\bm u}_1,\bm z + h_1\tilde{\bm v}_1, h_1,h_2) p_{\bm X, \bm Y, \bm Z}(\bm x,\bm y,\bm z) \\ 
      & ~~~~~~~~ p_{\bm X, \bm Y, \bm Z}(\bm x',\bm y + h_1 \bm v', \bm z + h_1\bm u') p_{\bm X, \bm Y, \bm Z}(\bm x_1,\bm y + h_1 \bm v_1,\bm z+h_1\bm u_1) p_{\bm X, \bm Y, \bm Z}(\tilde{\bm x}_1,\bm y + h_1\tilde{\bm v}_1, \bm z + h_1\tilde{\bm u}_1) \Bigg \} \mathrm d \mathcal T' , \nonumber 
       \end{align*} 
       where $\mathrm d \mathcal T' = \mathrm d\bm x\mathrm d\bm y\mathrm d\bm z\mathrm d\bm x'\mathrm d\bm u'\mathrm d\bm v'\mathrm d \bm x_1\mathrm d\bm u_1\mathrm d\bm v_1\mathrm d\tilde{\bm x}_1\mathrm d\tilde{\bm u}_1\mathrm d\tilde{\bm v}_1$. 
                 Now, under Assumption \ref{YZ}, it can be shown that as $n \rightarrow \infty$, 
        \begin{align} 
           & h_1^{d_Y+d_Z}\E[\vartheta_{1, 4}^{(2)}(\bm U, \bm U')^2] \nonumber \\ 
           & \rightarrow \mathcal C_2(K) \int \left[Q_0^2(\bm x,\bm x^\prime) p(\bm z)^4 p(\bm y,\bm z)^2\right]^2 p(\bm x, \bm y,\bm z) p(\bm x',\bm y,\bm z) \mathrm d \bm x \mathrm d \bm x' \mathrm d \bm y \mathrm d \bm z = \sigma_2^2 \mathcal C_2(K) , \label{eq:U14limitpf}
        \end{align}
        where $Q_0$ is as defined in Theorem \ref{large-sam-dist-1}, and
        $$ \mathcal C_2(K) = \int \left[\int K\left(\|(\bm u_1,\bm v_1)\|\right) K\left(\|(\bm u, \bm v) - (\bm u_1,\bm v_1)\|\right)\mathrm d \bm u_1\mathrm d \bm v_1\right]^2 \mathrm d\bm u \mathrm d \bm v.$$
    
  Now, define     
   \begin{align}\label{eq:UH}
       \Gamma_n & := \frac{1}{n(n-1)}\sum_{1\leq i \ne j\leq n} \overline{\vartheta}_{1, 4}^{(2)} (\bm U_i,\bm U_j),
    \end{align}
    where 
    $$\overline{\vartheta}_{1, 4}^{(2)} (\bm U_i,\bm U_j) = \left\{ \vartheta_{1, 4}^{(2)}(\bm U_i,\bm U_j)-\E_{\bm U}[ \vartheta_{1, 4}^{(2)} (\bm U_i,\bm U)]-\E_{\bm U}[ \vartheta_{1, 4}^{(2)} (\bm U,\bm U_j) ]+\E_{\bm U,\bm U'}[ \vartheta_{1, 4}^{(2)} (\bm U,\bm U')]\right\} , $$
 with $\bm U, \bm U'$ i.i.d. with density $p_{\bm X, \bm Y, \bm Z}$. Note that $\Gamma_n$ in \eqref{eq:UH} is a degenerate $U$-statistics of degree 2. Hence, we can apply Theorem 1 from \cite{hall1984} to establish the convergence of $\Gamma_n$. 
   For this, note from \eqref{eq:varianceU} and \eqref{eq:U14limitpf} that 
    \begin{align*}
      n^2 h_{1}^{d_Y+d_Z} \var[ \Gamma_n] = (1+ o(1)) h_{1}^{d_Y+d_Z} \var[\overline{\vartheta}_{1, 4}^{(2)} (\bm U_1,\bm U_2)] = (1+ o(1)) h_{1}^{d_Y+d_Z} \E[\overline{\vartheta}_{1, 4}^{(2)}(\bm U_1,\bm U_2)^2]\rightarrow \sigma_2^2  \mathcal C_2(K) . 
       \end{align*}  
Now, define 
    $$G_n(\bm u,\bm u') := \E\left[\overline{\vartheta}_{1, 4}^{(2)}(\bm U_1, \bm u) \overline{\vartheta}_{1, 4}^{(2)}(\bm U_1, \bm u')\right].$$
Using similar calculations we can show that $h_{1}^{(d_Y+d_Z)}\E[G_n^2(\bm U_1,\bm U_2)]$ and $h_{1}^{3(d_Y+d_Z)}\E[\overline{\vartheta}_{1, 4}^{(2)}(\bm U_1, \bm U_2)^4]$ converges to $a_1$ and $a_2$, respectively, where $a_1, a_2>0$ are constants.  Therefore,
    $$\frac{\E[G_n^2(\bm U_1,\bm U_2)]}{\left[\E[\overline{\vartheta}_{1, 4}^{(2)}(\bm U_1,\bm U_2)^2]\right]^2} = \frac{h_{1}^{2(d_Y+d_Z)}\E[G_n^2(\bm U_1,\bm U_2)]}{\left[h_{1}^{(d_Y+d_Z)}\E[\overline{\vartheta}_{1, 4}^{(2)}(\bm U_1,\bm U_2)^2]\right]^2} = O(h_{1}^{(d_Y+d_Z)})\rightarrow 0,$$
    and 
    \begin{align*} 
    \frac{\E[\overline{\vartheta}_{1, 4}^{(2)}(\bm U_1,\bm U_2)^4]}{n\left[\E[\overline{\vartheta}_{1, 4}^{(2)}(\bm U_1,\bm U_2)^2]\right]^2} = \frac{h_{1}^{2(d_Y+d_Z)}\E[\overline{\vartheta}_{1, 4}^{(2)}(\bm U_1,\bm U_2)^4]}{n\left[h_{1}^{(d_Y+d_Z)}\E[\overline{\vartheta}_{1, 4}^{(2)}(\bm U_1,\bm U_2)^2]\right]^2} & = \frac{h_{1}^{3(d_Y+d_Z)}\E[\overline{\vartheta}_{1, 4}^{(2)}(\bm U_1,\bm U_2)^4]}{(nh_{1}^{(d_Y+d_Z)})\left[h_{1}^{(d_Y+d_Z)}\E[\overline{\vartheta}_{1, 4}^{(2)}(\bm U_1,\bm U_2)^2]\right]^2} \nonumber \\ 
    & = O\left(\frac{1}{nh_{1}^{(d_Y+d_Z)}} \right) \rightarrow 0, 
    \end{align*}
    as $n\rightarrow \infty$. Hence, from Theorem 1 from \cite{hall1984} we have,
    $$nh_{1}^{(d_Y+d_Z)/2} \Gamma_n \stackrel{D}{\rightarrow} N(0, \sigma_2^2 \mathcal C_2(K)) .$$ 
Also, by arguments as in \eqref{eq:H11expectation}, $$nh_{1}^{(d_Y+d_Z)/2} \frac{1}{n}\sum_{i=1}^n \E_{\bm U}[\vartheta_{1, 4}^{(2)} ( \bm U_i,\bm U)] \stackrel{P}{\rightarrow} 0\text{ and }nh_{1}^{(d_Y+d_Z)/2} \E_{\bm U_1,\bm U}[\vartheta_{1, 4}^{(2)} ( \bm U_1,\bm U)] \rightarrow 0.$$ Hence, the result in \eqref{eq:H24pf} follows. 
 \end{proof}

With \eqref{eq:H2pfsmall}, \eqref{eq:H24pf}, and \eqref{eq:H35pf} proved, the proof of Lemma \ref{lemma:hn} is complete. 
   \end{proof}

To complete the proof of Theorem \ref{large-sam-dist-2} (3), we need to show that 
\begin{align}\label{eq:Hc}
nh_{1}^{(d_Y+d_Z)/2} \Delta^{(c)} \stackrel{P} \rightarrow 0 , 
\end{align} 
for $c \geq 3$. Towards this, as in the proof Lemma \ref{lemma:Hsmall}, it can shown that, under $H_0$, $\E[\varphi_{n}^{(c)}(\bm U_1,\ldots, \bm U_c)] = O(h_1^4)$, for $c \geq 1$. 
 Hence, recalling \eqref{eq:hc}, 
\begin{align}\label{eq:Hcexpectation} 
 n h_{1}^{(d_Y+d_Z)/2} \E[\Delta_{n}^{(c)}] =  O(n h_{1}^{(d_Y+d_Z)/2 + 4}) \rightarrow 0.
\end{align}  
  Moreover, by \eqref{eq:var-order} we have 
\begin{align*}
\var[\Delta_{n}^{(c)}] = \frac{1}{{n\choose c}}  \sum_{a=1}^c (-1)^{c-a}{c\choose a} \var[\varphi_{n}^{(a)}(\bm U_1,\ldots, \bm U_a)] = O\left(\frac{1}{n^c h_{1}^{(c-1)(d_Y+d_Z)}}\right) . 
\end{align*} 
Hence, 
\begin{align}\label{eq:Hcvariance}
n^2 h_{1}^{(d_Y+d_Z)} \var[\Delta_{n}^{(c)}] = O\left(\frac{1}{n^{c-2} h_{1}^{(c-2)(d_Y+d_Z)}}\right) \rightarrow 0 , 
\end{align}
for $c \geq 3$. The result in \eqref{eq:Hc} follows from \eqref{eq:Hcexpectation} and \eqref{eq:Hcvariance}. Combining Lemma \ref{lemma:Hsmall}, Lemma \ref{lemma:hn}, and \eqref{eq:Hc} completes the proof of Theorem \ref{large-sam-dist-2} (3). \hfill $\Box$

\section{Proofs from Section \ref{sec:conditionalindependence}}
\label{proofs}

In this section we present the proofs of the results from Section \ref{sec:conditionalindependence}. The proof of Proposition \ref{CRT-test} 
is given in Section \ref{sec:rtestpf} and the proof of Proposition \ref{CPT-test} is given in Section \ref{sec:permutationpf}. In Section \ref{sec:LBpf} we prove Theorem \ref{LB-test}.

\subsection{ Proof of Proposition \ref{CRT-test} }
\label{sec:rtestpf}

    Note that by construction the observed dataset $\mathcal{D}_0$ and the resampled datasets $\mathcal{D}_1',\mathcal{D}_2',\ldots, \mathcal{D}_M'$ are exchangeable under $H_0$. Therefore, $\hat\zeta_{n}(\mathcal{D}_0)$ and $\{\hat\zeta_{n}(\mathcal{D}_i')\}_{1 \leq i \leq M}$ are exchangeable under $H_0$. Using this exchangeability property we have the following results: 
    \begin{enumerate}
        \item[(i)] Note that $(M+1)p_{CRT} = \sum_{i=1}^M \bm{1}\{ \hat\zeta_n(\mathcal D_0)> \hat\zeta_n(\mathcal D_i')\}+1$ is the rank of $\hat\zeta_n(\mathcal D_0)$ among the observations $\{\hat\zeta_n(\mathcal D_0),\hat\zeta_n(\mathcal D_1'),\ldots, \hat\zeta_n(\mathcal D_M')\}$. Due to exchangeability $(M+1)p_{CRT}$ is uniformly distributed over $\{1,2,\ldots, M+1\}$. Therefore,
        \begin{align*}
           \P(p_{CRT}\leq \alpha) & = \P((M+1)p_{CRT}\leq (M+1)\alpha) = \frac{ \lfloor (M+1)\alpha \rfloor}{M+1}.          
        \end{align*} 
        Hence, our test controls the Type I error uniformly under the CRT framework. 

        \item[(ii)]  Recall that, by construction, $\bm X_i' \indpt \bm Y_i|\bm Z_i$, for all $1 \leq i \leq n$. Hence, by Theorem \ref{consistency}, $\hat\zeta_{n}(\mathcal{D}_1') \stackrel{P} \rightarrow 0$, as $n \rightarrow \infty$. Theorem \ref{consistency} also tells us that under $H_1$, $\hat\zeta_{n}(\mathcal{D}_0) \stackrel{P} \rightarrow \zeta(\bm X, \bm Y | \bm Z)>0$. Therefore, $\bm{1}\{\hat\zeta_n(\mathcal D_i')>\hat\zeta_n(\mathcal D_0)\} \rightarrow 0$, as $n \rightarrow \infty$, for each $1 \leq i \leq M$. Hence, for any $M \geq 1$ fixed, 
           \begin{align*}
            p_{CRT} = \frac{1+\sum_{i=1}^M \bm{1}\{ \hat\zeta_n(\mathcal D_0)>\hat\zeta_n(\mathcal D_i')\} }{1+M} \stackrel{P}{\rightarrow} \frac{1}{1+M},
        \end{align*}
        as $n \rightarrow \infty$. Therefore, if $\alpha>1/(M+1)$,  $\P(p_{CRT}\leq \alpha) \rightarrow 1$ as $n \rightarrow \infty$. 
        
    \end{enumerate}

\subsection{ Proof of Proposition \ref{CPT-test} } 
\label{sec:permutationpf}

    Here also under $H_0$, the observed data $\mathcal{D}_0$ and $\mathcal{D}^{\pi_1},\mathcal{D}^{\pi_2},\ldots, \mathcal{D}^{\pi_M}$ are exchangeable (see \citet[Theorem 1]{berrett2019}). Hence, $\hat\zeta_{n}(\mathcal{D}_0)$ and $\{\hat\zeta_{n}(\mathcal{D}^{\pi_j})\}_{1 \leq j \leq M}$ are exchangeable. Now, using similar arguments as in the proof of Proposition \ref{CRT-test} we have our desired result.

\subsection{ Proof of Theorem \ref{LB-test} }
\label{sec:LBpf}

    Let us define $\mathcal{S} := \{(\bm Y_1,\bm Z_1),(\bm Y_2,\bm Z_2),\ldots, (\bm Y_n,\bm Z_n)\}$, $\mathcal{S}_{\bm Z} := \{\bm Z_1,\bm Z_2,\ldots,\bm Z_n\}$, $\mathcal{S}_{\bm X} := \{\bm X_1,\bm X_2,\ldots,\bm X_n\}$, and $\mathcal{S}_{\bm X}' := \{\bm X_1',\bm X_2',\ldots,\bm X_n'\}$. 
    
    \begin{lemma}\label{lm:totalvariation} Suppose the assumptions of Theorem \ref{LB-test} hold. Then under $H_0$, 
        \begin{align}
         \label{eq:tv-bound-final}
         \mathrm{TV}\left(\mathcal{L}\Big(\hat\zeta_{n}(\hat{\mathcal{D}}_1)\mid \mathcal{S}\Big),\mathcal{L}\Big(\hat\zeta_{n}(\mathcal{D}_0)\mid \mathcal{S}\Big)\right) \stackrel{P}{\rightarrow} 0 .
     \end{align}
    \end{lemma}
    
    \begin{proof}
    Then by the definition of the total variation distance (denoted as $\mathrm{TV}$) for any measurable function $f$ we have,
    \begin{align}
        & \mathrm{TV}\left(\mathcal L\big(f(\bm X_1,\ldots,\bm X_n)\big),\mathcal L\big(f(\bm X_1',\ldots, \bm X_n')\big)\right)\nonumber\\
        & = \sup_{A\in \mathcal B(\R)}\Big|\P(f(\bm X_1,\ldots,\bm X_n) \in A)-\P(f(\bm X_1',\ldots, \bm X_n')\in A)\Big|\nonumber\\
        & = \sup_{A\in \mathcal B(\R)}\Big|\P((\bm X_1,\ldots,\bm X_n) \in f^{-1} (A))-\P((\bm X_1',\ldots, \bm X_n')\in f^{-1}(A))\Big|\nonumber\\
        & \leq \sup_{A\in \mathcal B(\R^n)}\Big|\P((\bm X_1,\ldots,\bm X_n) \in A)-\P((\bm X_1',\ldots, \bm X_n')\in A)\Big|\nonumber\\
        & = \mathrm{TV}\big((\bm X_1,\ldots, \bm X_n),(\bm X_n',\ldots, \bm X_n' )\big),
        \label{eq:totalvariation}
    \end{align}
    where $\mathcal L\big(\bm V)$, for any random variable $\bm V$, denotes the distribution of $\bm V$.   
    Using \eqref{eq:totalvariation} we have
    \begin{equation*}
        \begin{split}
            & \mathrm{TV}\left(\mathcal{L}\Big(\hat\zeta_{n}(\hat{\mathcal{D}}_1)\mid \mathcal{S}\Big),\mathcal{L}\Big(\hat\zeta_{n}(\mathcal{D}_0)\mid \mathcal{S}\Big)\right)\leq \mathrm{TV} \Big(\mathcal{L}\big(\mathcal{S}_{\bm X}'\mid \mathcal{S}\big),\mathcal{L}\big(\mathcal{S}_{\bm X}\mid \mathcal{S}\big)\Big).
        \end{split}
    \end{equation*}
    However, under $H_0: \bm X\indpt \bm Y | \bm Z$ we have, $\mathcal{L}(\mathcal{S}_{\bm X}\mid \mathcal{S}) = \mathcal{L}(\mathcal{S}_{\bm X}\mid \mathcal{S}_{\bm Z})$ with density $\prod_{i=1}^n p(\bm x_i| \bm Z_i)$. Also, by the construction of the resampled data $\mathcal{L}(\mathcal{S}_{\bm X}'\mid \mathcal{S}) = \mathcal{L}(\mathcal{S}_{\bm X}'\mid \mathcal{S}_{\bm Z})$ with density $\E\left[\prod_{i=1}^n \hat p(\bm x_i| \bm Z_i)| \bm Z_1, \ldots,\bm Z_n\right]$ where $\hat p(\bm x\mid\bm z)$ is as in \eqref{eq:estimatepXZ}. Therefore, under $H_0$, 
    \begin{align}
    \label{eq:tv-bound-1}
            & \mathrm{TV}\left(\mathcal{L}\Big(\hat\zeta_{n}(\hat{\mathcal{D}}_1)\mid \mathcal{S}\Big),\mathcal{L}\Big(\hat\zeta_{n}(\mathcal{D}_0)\mid \mathcal{S}\Big)\right) \nonumber\\
            & \leq \mathrm{TV} \Big(\mathcal{L}\big(\mathcal{S}_{\bm X}'\mid \mathcal{S}\big),\mathcal{L}\big(\mathcal{S}_{\bm X}\mid \mathcal{S}\big)\Big)\nonumber\\
            & = \mathrm{TV} \Big(\mathcal{L}\big(\mathcal{S}_{\bm X}'\mid \mathcal{S}_{\bm Z}\big),\mathcal{L}\big(\mathcal{S}_{\bm X}\mid \mathcal{S}_{\bm Z}\big)\Big)\nonumber\\
            & = \frac{1}{2} \int \Big|\E\big[\prod_{i=1}^n \hat p(\bm x_i| \bm Z_i)| \bm Z_1,\bm Z_2,\ldots,\bm Z_n\big] - \prod_{i=1}^n p(\bm x_i| \bm Z_i)\Big| \prod_{i=1}^n\,\mathrm d\bm x_i.
    \end{align}
    Note that conditioned on $\bm Z_1,\bm Z_2,\ldots,\bm Z_n$ we have,
    \begin{align*}
     \E\left[\prod_{i=1}^n \hat p(\bm x_i| \bm Z_i) | \bm Z_1,\ldots, \bm Z_n\right] = \sum_{\bm i\in [n]^n} \prod_{j=1}^n \gamma_{i_j}( \bm Z_j )  \E\left[ \prod_{j=1}^n \frac{1}{h_0^{d_X}} \phi(\bm x_{j}, \bm X_{i_j}, h_0 \bm I_{d_X})| \bm Z_1,\ldots, \bm Z_n\right] . 
    \end{align*}
    Observe that $$\E \left[ \prod_{j=1}^n \frac{1}{h_0^{d_X}} \phi(\bm x_{j}, \bm X_{i_j}, h_0 \bm I_{d_X})| \bm Z_1,\bm Z_2,\ldots,\bm Z_n\right]$$ is itself a valid density function, for any $\bm i = (i_1,\ldots, i_n) \in [n]^n$. 
   Using $\sum_{\bm i\in [n]^n} \prod_{j=1}^n \gamma_{i_j}( \bm Z_j ) = \prod_{j=1}^n \sum_{i=1}^n \gamma_i(\bm Z_j) = 1$ gives, 
    \begin{align}
    \label{eq:tv-bound-2}
        & \int \left|\E\left[\prod_{i=1}^n \hat p(\bm x_i| \bm Z_i)| \bm Z_1,\ldots,\bm Z_n \right] - \prod_{i=1}^n p(\bm x_i| \bm Z_i) \right| \prod_{i=1}^n\, \mathrm d \bm x_i\nonumber\\
        & = \int \left|\sum_{ \bm i \in [n]^n } \prod_{j=1}^n \gamma_{i_j}( \bm Z_j )  \E\left[ \prod_{j=1}^n \frac{1}{h_0^{d_X}} \phi(\bm x_{j},\bm X_{i_j}, h_0 \bm I_{d_X})| \bm Z_1,\ldots,\bm  Z_n\right]- \prod_{i=1}^n p(\bm x_i| \bm Z_i) \right| \prod_{i=1}^n\, \mathrm d\bm x_i\nonumber\\
        & \leq \prod_{i=1}^n \gamma_{i}(\bm Z_i) \int \left| \E\left[ \prod_{i=1}^n \frac{1}{h_0^{d_X}} \phi(\bm x_{i}, \bm X_{i}, h_0 \bm I_{d_X})| \bm Z_1,\ldots,\bm Z_n\right]-\prod_{i=1}^n p(\bm x_i| \bm Z_i) \right|\prod_{i=1}^n\, \mathrm d\bm x_i + U_n, 
    \end{align}
     where 
    \begin{align*} 
     U_n = \sum_{\substack{\pi: [n] \rightarrow [n] \\ \pi \not= \mathrm{id} }} \prod_{i=1}^n \gamma_{\pi(i)}( \bm Z_i ) \int \Big| \E\big[ \prod_{i=1}^n \frac{1}{h_0^{d_X}} \phi(\bm x_{i}, \bm X_{\pi(i)}, h_0 \bm I_{d_X})| \bm Z_1,\ldots,\bm Z_n\big]-\prod_{i=1}^n p(\bm x_i| \bm Z_i)\Big|\prod_{i=1}^n\, \mathrm d\bm x_i , 
     \end{align*} 
     with $\mathrm{id}: [n] \rightarrow [n]$ denoting the identity map, that is, $\mathrm{id}(i) = i$, for $1 \leq i \leq n$.   

      For any $\pi: [n] \rightarrow [n]$ with $\pi \not= \mathrm{id}$ note that,
     \begin{align*}
         & \int \left| \E\left[ \prod_{i=1}^n \frac{1}{h_0^{d_X}} \phi(\bm x_{i}, \bm X_{\pi(i)}, h_0 \bm I_{d_X})| \bm Z_1,\ldots,\bm Z_n\right]-\prod_{i=1}^n p(\bm x_i| \bm Z_i) \right| \prod_{i=1}^n\, \mathrm d\bm x_i \nonumber \\ 
         & = 2\mathrm{TV}(\mathcal{L}(\bm V_1,\ldots,\bm V_n), \mathcal{L}(\bm X_1,\ldots, \bm X_n)),
     \end{align*}
     for some random vector $(\bm V_1,\ldots, \bm V_n)$ having joint density function $\E[ \prod_{i=1}^n \frac{1}{h_0^{d_X}} \phi(\bm x_{i}, \bm X_{\pi(i)}, h_0 \bm I_{d_X})| \bm Z_1,\ldots,\bm Z_n]$. Bounding the TV distance by 1, we then have     
     \begin{align}
         \label{eq:tv-bound-3}
         U_n = O_P \left(  \sum_{\substack{\pi: [n] \rightarrow [n] \\ \pi \not= \mathrm{id} }} \prod_{i=1}^n \gamma_{\pi(i)}( \bm Z_i )  \right) = o_P(1) , 
             \end{align}
         as $n \rightarrow \infty$, since $\prod_{s=1}^n \gamma_{s}(\bm Z_s) \stackrel{P} \rightarrow 1$.  
     Also, note that
     \begin{align} 
         \E\left[ \prod_{i=1}^n \frac{1}{h_0^{d_X}} \phi(\bm x_{i}, \bm X_{i}, h_0 \bm I_{d_X})| \bm Z_1,\ldots,\bm Z_n\right] & = \int \prod_{i=1}^n \frac{1}{h_0^{d_X}} \phi\left(\Big\|\frac{\bm x_{i}-\bm w_{i}}{h_0}\Big\|\right) \prod_{i=1}^n p(\bm w_i| \bm Z_i) \prod_{i=1}^n \mathrm d \bm w_i\nonumber\\
         & = \int \prod_{i=1}^n \phi(\bm a_i) \prod_{i=1}^n p(\bm x_i + h_0\bm a_i| \bm Z_i) \prod_{i=1}^n \mathrm d \bm a_i \nonumber \\ 
         &= \prod_{i=1}^n \left( \int  \phi(\bm a) p(\bm x_i + h_0\bm a| \bm Z_i)  \mathrm d \bm a \right) . \nonumber 
     \end{align}
     Hence, by a telescoping argument, 
     \begin{align}
            A_n := &  \int \left|\E\left[ \prod_{i=1}^n \frac{1}{h_0^{d_X}} \phi(\bm x_{i}, \bm X_{i}, h_0 I_{d_X})| \bm Z_1,\ldots,\bm Z_n\right] - \prod_{i=1}^n p(\bm x_i| \bm Z_i)\right| \prod_{i=1}^n \mathrm d \bm x_i \nonumber\\ 
            & = \int \left| \prod_{i=1}^n \left( \int  \phi(\bm a) p(\bm x_i + h_0\bm a| \bm Z_i)  \mathrm d \bm a \right) - \prod_{i=1}^n p(\bm x_i| \bm Z_i)\right| \prod_{i=1}^n \mathrm d \bm x_i  \nonumber\\             
            & \leq \sum_{i=1}^n \left|\int \phi(\bm a) p(\bm x_i + h_0\bm a | \bm Z_i) \mathrm d \bm a - p(\bm x_i | \bm Z_i)\right| \mathrm d \bm x_i . \nonumber 
            \end{align} 
            Taking expectation with respect to $\bm Z_1,\ldots,\bm Z_n$ gives,          
            \begin{align}
                   \label{eq:tv-bound-4}
            \E[A_n] &  = n  \E \left[ \left|\int \phi(\bm a) p(\bm x + h_0\bm a | \bm Z_1) \mathrm d \bm a - p(\bm x | \bm Z_1)\right| \mathrm d \bm x \right] \nonumber\\
            & = n \int \left|\int \phi(\bm a) \frac{p_{\bm X,\bm Z}(\bm x + h_0\bm a,\bm z)}{p_{\bm Z}(\bm z)} \mathrm d \bm a - \frac{p_{\bm X,\bm Z}(\bm x, \bm z)}{p_{\bm Z}(\bm z )}\right| p_{\bm Z}(\bm z)\mathrm d \bm x \mathrm d \bm z \nonumber\\
            & = n \int \left|\int \phi(\bm a) p_{\bm X,\bm Z}(\bm x + h_0\bm a,\bm z) \mathrm d \bm a - p_{\bm X,\bm Z}(\bm x, \bm z)\right| \mathrm d \bm x \mathrm d \bm z \nonumber\\
            & = O(nh_0^2) ,     
          \end{align} 
            since $p_{\bm X,\bm Z}$ is nice in $\bm X$. This implies, $A_n \stackrel{P} \rightarrow 0$.  
     Now combining \eqref{eq:tv-bound-1},\eqref{eq:tv-bound-2},\eqref{eq:tv-bound-3}, and \eqref{eq:tv-bound-4} the result in Lemma \ref{lm:totalvariation} follows. 
         \end{proof}

     Now we prove the desired properties of the Type I error and power of the local wild bootstrap test $\phi_{LWB} = \bm{1}\{ p_{LWB}<\alpha \}$. 

    \begin{enumerate}
        \item[(a)] Before proving the result in Theorem \ref{LB-test} (a) let us first note that for two random variables $T$ and $W$, 
        \begin{align}
        \label{eq:totalvariationfact}
            \P(T>x) = \int_{\{ t>x\} } \mathrm d\mathcal L(T)(t) & = \int_{\{t>x\}} \mathrm d\Big(\mathcal L(W) + \mathcal L(T)-\mathcal{L}(W)\Big)(t)\nonumber\\
            & = \int_{\{t>x\}} \mathrm d\mathcal L(W)(t) + \int_{\{t>x\}} \mathrm d\Big(\mathcal L(T)-\mathcal{L}(W)\Big)(t)\nonumber\\
            & \leq \P(W>x) + \sup_A \left| \mathcal L(T)(A)-\mathcal{L}(W)(A)\right|\nonumber\\
            & = \P(W>x) + \mathrm{TV}(\mathcal L(T), \mathcal L(W)),
        \end{align}
        for any $x\in \R$. Using \eqref{eq:totalvariationfact} for $\zeta_n(\hat{\mathcal{D}}_1)$ and $\zeta_n(\mathcal D_0)$ conditioning on $\mathcal S$ we have,
        \begin{align*}
          \P(\hat\zeta_{n}(\mathcal{D}_0)>x\mid \mathcal S) &\leq \P(\hat\zeta_{n}(\hat{\mathcal{D}}_1)>x \mid \mathcal S) + \mathrm{TV}\left(\mathcal{L}\Big(\hat\zeta_{n}(\hat{\mathcal{D}}_1)\mid \mathcal{S}\Big),\mathcal{L}\Big(\hat\zeta_{n}(\mathcal{D}_0)\mid \mathcal{S}\Big)\right) . 
        \end{align*}
       Taking expectation on both sides, applying Lemma \ref{lm:totalvariation} and the Dominated Convergence Theorem gives       
       \begin{align}
        \label{eq:stat-restat-relation}
            \limsup_{n\to\infty}\P(\hat\zeta_n(\mathcal D_0)>x) &  \leq \limsup_{n\to\infty} \P(\hat\zeta_n(\hat{\mathcal D}_1)>x) ,
        \end{align} 
   for any $x\in\R$. Now, define $c_n^*(\alpha)$ as the upper $\alpha$-th quantile of $\hat\zeta_n(\hat{\mathcal D}_1)$ conditioned on $\mathcal D_0$, that is, $\P(\hat\zeta_n(\hat{\mathcal D}_1) >c_{n}^*(\alpha)\mid \mathcal D_0) \leq \alpha$. Also, define $p^*:=\P(\hat\zeta_{n}(\hat{\mathcal{D}}_1)\geq \hat\zeta_{n}(\mathcal{D}_0) \mid \mathcal{D}_0)$. Let
        $$F(t) = \P\{\hat\zeta_{n}(\hat{\mathcal{D}}_1)\leq t\mid \mathcal{D}_0\} \quad \text{ and } \quad F_M(t) = \frac{1}{M}\left\{\sum_{i=1}^M \bm{1}\{ \hat\zeta_{n}(\hat{\mathcal{D}}_i)\leq t \} \right\}.$$
        Note that $F$ and $F_M$ are distribution functions conditioned on the observed data $\mathcal{D}_0$. Then,
        \begin{equation*}
          \begin{split}
             |p^*-p_{LWB}| & = \left|\P(\hat\zeta_{n}(\hat{\mathcal{D}}_1)\geq \hat\zeta_{n}(\mathcal{D}_0)\mid \mathcal{D}_0) -\frac{1}{M+1}\Big\{\sum_{i=1}^M \bm{1}\{ \hat\zeta_{n}(\hat{\mathcal{D}}_i)\geq \hat\zeta_{n}(\mathcal{D}_0) \} + 1\Big\}\right|\\
             & = \left|\P(\hat\zeta_{n}(\hat{\mathcal{D}})< \hat\zeta_{n}(\mathcal{D}_0)\mid \mathcal{D}_0) - \frac{1}{M+1}\Big\{\sum_{i=1}^M \bm{1}\{ \hat\zeta_{n}(\hat{\mathcal{D}}_i)< \hat\zeta_{n}(\mathcal{D}_0)\} \Big\}\right|\\
             &  = \left|F\Big(\hat\zeta_{n}(\mathcal{D}_0)-\Big)-\frac{M}{M+1}F_M\Big(\hat\zeta_{n}(\mathcal{D}_0)-\Big)\right|\\
             & \leq \left|F\Big(\hat\zeta_{n}(\mathcal{D}_0)-\Big)-F_M\Big(\hat\zeta_{n}(\mathcal{D}_0)-\Big)\right| + \Big|\frac{F_M\big(\hat\zeta_{n}(\mathcal{D}_0)-\big)}{M+1}\Big|\\ & \leq \sup_{t\in\R}|F(t)-F_M(t)|+\frac{1}{M+1}.
         \end{split}
        \end{equation*}        
        Using the Dvoretzky-Keifer-Wolfwitz inequality now gives, 
        $$\P(|p_{LWB}-p^*|>\varepsilon\mid \mathcal{D}_0) \leq 2\exp\left\{-2M\left(\varepsilon-\frac{1}{M+1}\right)^2 \right \},$$
        for any $\varepsilon>1/(M+1)$. Now, fix a large $M$ and take any $\varepsilon>1/(M+1)$. Then,
        \begin{align*}
            \P(p_{LWB}\leq \alpha\mid\mathcal{D}_0) & = \P(p_{LWB}\leq \alpha, |p_{LWB}-p^*|>\varepsilon\mid\mathcal{D}_0) + \P(p_{LWB}\leq \alpha, |p_{LWB}-p^*|\leq\varepsilon\mid\mathcal{D}_0) \\
            &\leq \P(|p_{LWB}-p^*|>\varepsilon \mid\mathcal{D}_0) + \bm 1\left\{p^*\leq\alpha+\varepsilon\right\} \\
            & \leq 2\exp \left\{-2M\left(\varepsilon-\frac{1}{M+1}\right)^2\right\} + \bm 1\{ \zeta_{n}(\mathcal{D}_0)>c_{n}^*(\alpha+\varepsilon)\} .
        \end{align*}
        Taking expectations on both sides gives, 
        $$\P(p_{LWB}\leq \alpha) \leq 2\exp\left\{-2M\left(\varepsilon-\frac{1}{M+1}\right)^2\right\} + \P(\zeta_{n}(\mathcal{D}_0)>c_{n}^*(\alpha+\varepsilon)).$$  
        Now, taking limit as $n\to\infty$ and $M\to\infty$ we get,
        \begin{align*}
        \limsup_{n,M\rightarrow\infty} \P(p_{LWB}\leq \alpha) & \leq \limsup_{n \rightarrow \infty} \P(\hat\zeta_{n}(\mathcal{D}_0)>c_{n}^*(\alpha+\varepsilon) ) \\
        & \leq \limsup_{n \to\infty} \P(\hat\zeta_n(\hat{\mathcal D}_1) >c_{n}^*(\alpha+\varepsilon)) \tag*{\text{(using \eqref{eq:stat-restat-relation})}}\\
        & = \limsup_{n \to\infty} \E\left[\P(\hat\zeta_n(\hat{\mathcal D}_1) >c_{n}^*(\alpha+\varepsilon)\mid \mathcal D_0) \right]\\
        & \leq \limsup_{n \to\infty} \E\left[\alpha + \varepsilon \right]\tag*{\text{(by the definition of $c_n^*(\alpha+\varepsilon)$)}}\\
        &  = \alpha +  \varepsilon. 
        \end{align*}
        Hence, since $\varepsilon$ is arbitrary, under $H_0$, $\limsup_{n,M\to\infty}\P_{H_0}(p_{LWB}\leq \alpha ) \leq \alpha$. This completes the proof of the first part of Theorem \ref{LB-test}.   
        
        \item[(b)] Under $H_1:\bm X\not\indpt \bm Y | \bm Z$, using similar arguments as in the proof of  Lemma \ref{lm:totalvariation} we can show that 
      \begin{align}
         \label{eq:tv-bound-final}
         \mathrm{TV}\left(\mathcal{L}\Big(\hat\zeta_{n}(\hat{\mathcal{D}}_1)\mid \mathcal{S}\Big),\mathcal{L}\Big(\hat\zeta_{n}(\mathcal{D}_1')\mid \mathcal{S}\Big)\right)\stackrel{P}{\rightarrow} 0,
     \end{align}
     where $\mathcal D_1'$ is the resampled data obtained in the conditional randomization test. Therefore, even under $H_1$, $\hat\zeta_n(\hat{\mathcal{D}}_1) \stackrel{P} \rightarrow 0$, as $n\rightarrow\infty$, by Theorem \ref{consistency}. Moreover, $\hat\zeta_{n}(\mathcal{D}_0)$ converges to $\zeta(\bm X,\bm Y | \bm Z)$ which is positive under $H_1$. Therefore, under $H_1$, $p_{LWB}$ converges in probability to $1/(M+1)$, as $n \rightarrow \infty$. Hence, if $\alpha>1/(M+1)$ or $M>(1-\alpha)/\alpha$, $\limsup_{n\to\infty}\P_{H_1}(p_{LWB}\leq \alpha ) =1$, that is, the local wild bootstrap test is consistent. 
    \end{enumerate}

\subsection{Verifying the Weight Condition} 
\label{sec:weightpf}

Let $\gamma_s(\bm z) = \kappa_s(\bm z)$ as defined in \eqref{eq:kernelweights}. In the following lemma shows how to choose the bandwidth $h_2'$ such that the condition $\prod_{s=1}^n \gamma_s(\bm Z_s)\stackrel{P}{\rightarrow} 1$ holds. 

\begin{lemma} 
Let $\gamma_s(\bm z) = \kappa_s(\bm z)$ as defined in \eqref{eq:kernelweights}. Also, suppose that the Assumptions \ref{YZ}, \ref{K}, \ref{h} hold. If $h_2^\prime  = o(n^{-2/d_Z})$ we have,
    $$\prod_{s=1}^n \gamma_s(\bm Z_s)\stackrel{P}{\rightarrow} 1 , $$
    as $n\rightarrow\infty$. 
    \label{lemma:auxillary-bandwidth}
\end{lemma}

\begin{proof} First, recalling \eqref{eq:kernelweights}, note that 
$\{ \gamma_1(\bm Z_1)<1 \}$ implies that there exists $\bm Z_j$, with $j \not=1$, such that $K(\frac{\|\bm Z_1-\bm Z_j\|}{h_2^\prime})>0$. Now, since the kernel $K$ has compact support (by Assumption \ref{K}), this implies $\|\bm Z_j-\bm Z_1\|< Ch_2^\prime$ for some constant  $C>0$. Therefore, by the union bound,
    \begin{align*}
        \P(\gamma_1(\bm Z_1)<1) & \leq \sum_{j=2}^n \P\left(\|\bm Z_j-\bm Z_1\|< Ch_2^\prime\right)\\
        & = (n-1) \P\left(\|\bm Z_2-\bm Z_1\|< Ch_2^\prime\right)  \tag*{(since $\bm Z_2, \ldots, \bm Z_n$ are i.i.d.)}\\
        & = (n-1) \int_{\|\bm z_2-\bm z_1\|< Ch_2^\prime} p_{\bm Z}(\bm z_1) p_{\bm Z}(\bm z_2) \mathrm d\bm z_1 \mathrm d \bm z_2\\
        & \lesssim  (n-1) \int  \mathrm{Vol}(B(\bm z_1, Ch_2^\prime)) p_{\bm Z}(\bm z_1) \mathrm d \bm z_1 \tag*{(since $p_{\bm Z}$ is bounded by Assumption \ref{YZ})} \\
        & \lesssim (n-1)(h_2^\prime)^{d_Z}. 
    \end{align*}
    Here, $\mathrm{Vol}(B(\bm z_1, Ch_2^\prime))$ denotes the volume of the ball of radius $Ch_2^\prime$ which is centered at $\bm z_1$ and the constants in the $\lesssim$ bounds depend on $d_Z$. Now, fix any arbitrary $\epsilon>0$ and note that,
    \begin{align*}
        \P\left(\prod_{s=1}^n \gamma_s(\bm Z_s)<1-\epsilon\right) & \leq \sum_{s=1}^n \P\left(\gamma_s(\bm Z_s)<(1-\epsilon)^{1/n}\right)\\
        & = n \P\left(\gamma_1(\bm Z_1)<(1-\epsilon)^{1/n}\right) \\
        & \leq n \P\left(\gamma_1(\bm Z_1)<1\right) \\
        & \lesssim  n(n-1)(h_2^\prime)^{d_Z}\\
        & \rightarrow 0 ,   
    \end{align*} 
    since $h_2^\prime = o(n^{-2/d_Z})$. This completes the proof. 
\end{proof}

\section{Proofs from Section \ref{sec:empiricalprocess} } 
\label{sec:technicalempiricalprocesspf}

This section is organized as follows: In Appendix \ref{sec:technicalpf} we collect the proofs of some technical lemmas which will be used in the analysis of the conditional empirical process. The proof of Proposition \ref{prop:main-1} is given in Section \ref{sec:empiricalprocesspf}. Proposition \ref{prop:main-8} is proved in Section \ref{sec:Vpf}.

\subsection{Auxiliary Lemmas}
\label{sec:technicalpf}

We begin with the following result: 

\begin{lemma} Suppose $(\bm X_1, \bm Z_1), (\bm X_2, \bm Z_2), \ldots, (\bm X_r, \bm Z_r)$ are i.i.d. samples from $p_{\bm X, \bm Z}$. Further, suppose the following conditions hold:   
\begin{itemize}
\item $p_{\bm X,\bm Z}$ satisfies assumption \ref{YZ}; 

\item the kernel $K$ satisfies assumption \ref{K}; 

\item  $h_{2}$ is a bandwidth satisfying the conditions in assumption \ref{h}. 

\end{itemize}
Fix an integer $r \geq 1$. Then for a measurable, symmetric function $\psi: \mathbb (\mathbb R^{d_X})^r \rightarrow \mathbb R$ such that $$\mathbb E[\psi^2( \bm X_1, \bm X_2, \ldots, \bm X_r ) | \{\bm Z_s = \bm z\}_{1 \leq s \leq r} ] < \infty,$$ we have 
\begin{align}\label{eq:hz}
\lim_{h_{2} \rightarrow 0} \frac{1}{h_{2}^{ r d_Z } } \mathbb E \left[  \prod_{s=1}^r  K \left(\frac{\| \bm z - \bm Z_s \|}{h_{2}}\right) \psi(\bm X_1, \ldots, \bm X_r)  \right] = p_{\bm Z} (\bm z)^r \E \left[\psi(\bm X_1, \ldots, \bm X_r)| \{\bm Z_s = \bm z\}_{1 \leq s \leq r} \right]. 
\end{align}
Further, if
\begin{align}
    \sup_{ \bm z\in \R^{d_Z} }\mathbb E[\psi^2( \bm X_1, \bm X_2, \ldots, \bm X_r ) | \{\bm Z_s = \bm z\}_{1 \leq s \leq r} ] < \infty,
    \label{eq:assumption-finiteness}
\end{align} holds, we have,   
\begin{align} \label{eq:hZ}
& \lim_{h_{2} \rightarrow 0} \frac{1}{h_{2}^{ r d_Z } } \int \mathbb E \left[  \prod_{s=1}^r  K\left(\frac{\| \bm z - \bm Z_s \|}{h_{2}}\right) \psi(\bm X_1, \ldots, \bm X_r)  \right] p_{\bm Z}(\bm z)\mathrm d \bm z \nonumber \\ 
& = \int p_{\bm Z} (\bm z)^{r+1} \E \left[\psi(\bm X_1, \ldots, \bm X_r)| \{\bm Z_s = \bm z\}_{1 \leq s \leq r} \right] \mathrm d \bm z. 
\end{align} 
\label{lm:aux-1}
\end{lemma} 

\begin{proof} 
Denote $\underline{\bm X} = (\bm X_1, \ldots, \bm X_r)$ and $\underline{\bm x} = (\bm x_1, \ldots, \bm x_r)$. Then observe that  
\begin{align}\label{eq:Kfunction1}
 \frac{1}{h_{2}^{ r d_Z }} \mathbb E \left[ \prod_{s=1}^r K\left(\frac{\| \bm z - \bm Z_s \|}{h_{2}}\right) \psi( \underline{ \bm X} ) \right] & = \frac{1}{h_{2}^{ r d_Z }} \int \int \prod_{s=1}^r K\left(\frac{\| \bm z - \tilde{\bm z}_s \|}{h_{2}} \right) \psi(\underline{\bm x}) \prod_{s=1}^r p_{\bm X, \bm Z} (\bm x_s, \tilde{\bm z}_s) \mathrm d \bm x_s \mathrm d \tilde{\bm z}_s \nonumber \\ 
 & = \int \int \prod_{s=1}^r  K\left( \| \bm z_s \| \right) \psi( \underline{\bm x}) \prod_{s=1}^r p_{ \bm X, \bm Z} ( \bm x_s, \bm z + h_{2} \bm z_s) \mathrm d \bm x_s \mathrm d \bm z_s , 
\end{align} 
with the change of variable $\tilde{\bm z}_s = \bm z+ h_{2} \bm z_s$, for $1 \leq s \leq r$. Denote $$p^{(r)}( \underline{\bm x}, \bm z )  := \prod_{s=1}^r p_{\bm X, \bm Z} (\bm x_s, \bm z)  \text{ and } p^{(r)}(  \underline{\bm x}, \bm z + h \underline{ \bm z}  )   := \prod_{s=1}^r p_{\bm X, \bm Z} (\bm x_s, \bm z + h_{2} \bm z_s) , $$ 
where $\underline{\bm z} = (\bm z_1, \ldots, \bm z_r)$. Also, define $Q(\underline{\bm z}) := \prod_{s=1}^r K(\|{\bm z_s}\|)$ and 
\begin{align}
    R(\bm x_s,\bm z;\bm z_s, h_2) = \frac{1}{h_{2}^{2}} \left\{ \frac{p_{\bm X,\bm Z}(\bm x_s,\bm z + h_{2} \bm z_s)}{p_{\bm X,\bm Z}(\bm x_s,\bm z)} -1 - h_{2} ~\bm z_s  \frac{\frac{\partial}{\partial \bm z}p_{\bm X,\bm Z}(\bm x_s,\bm z)}{p_{\bm X,\bm Z}(\bm x_s,\bm z)} \right\}.
    \label{eq:density-residual}
\end{align}
Then, using the identity, 
$$p_{\bm X,\bm Z}(\bm x_s, \bm z + h_{2} \bm z_s) = p_{\bm X,\bm Z}(\bm x_s, \bm z ) + h_{2} \bm z_s  \frac{\partial}{\partial \bm z}p_{\bm X,\bm Z}(\bm x_s,\bm z) + h_{2}^2 p_{\bm X,\bm Z}(\bm x_s,\bm z) R(\bm x_s,\bm z; \bm z_s, h_2) , $$ 
and noting that 
$\int \int K (\bm z_s) 
\bm z_s  \frac{\partial}{\partial \bm z}p_{\bm X,\bm Z}(\bm x_s,\bm z) \mathrm d \bm z_s = 0$, 
since the kernel is symmetric about zero by Assumption \ref{K}, we have  
\begin{align}\label{eq:Khfunction2}
      A_n^{\bm z} & := \int \int Q (\underline{\bm z}) \psi(\underline{\bm x}) p^{(r)}(  \underline{\bm x}, \bm z + h \underline{ \bm z}  )   \prod_{s=1}^r \mathrm d \bm x_s \mathrm d \bm z_s - \int\int Q(\underline{\bm z}) \psi( \underline{\bm x}) p^{(r)}( \underline{\bm x}, \bm z )  \prod_{s=1}^r \mathrm d \bm x_s \mathrm d \bm z_s  \nonumber \\
             & = \sum_{k=1}^r h_{2}^{2 k} \sum_{\substack{S\subseteq [r]:\\ |S| = k}} A_n^{\bm z}(S), 
             \end{align} 
        where 
        \begin{align}\label{eq:AnS}
        |A_n^{\bm z}(S)| & := \left|\int \int Q (\underline{\bm z}) \psi(\underline{\bm x}) \prod_{s\in S^c} p_{\bm X, \bm Z} (\bm x_s, \bm z) \prod_{s\in S} \Big( R(\bm x_s,\bm z;\bm z_s, h_2) p_{\bm X,\bm Z}(\bm x_s,\bm z)\Big) \prod_{s=1}^r \mathrm d \bm x_s \mathrm d \bm z_s\right| \nonumber \\ 
        & \leq \int Q (\underline{\bm z}) \left( \int  \big|\psi(\underline{\bm x})\big| \prod_{s\in S^c} p_{\bm X, \bm Z} (\bm x_s, \bm z) \prod_{s\in S} \Big( \big|R(\bm x_s,\bm z;\bm z_s, h_2)\big| p_{\bm X,\bm Z}(\bm x_s,\bm z)\Big) \prod_{s=1}^r \mathrm d \bm x_s \right) \prod_{s=1}^r  \mathrm d \bm z_s \nonumber \\ 
        & \leq  \int Q (\underline{\bm z}) \left ( \int  \psi(\underline{\bm x})^2 p^{(r)}( \underline{\bm x}, \bm z ) \prod_{s=1}^r \mathrm d \bm x_s \right)^{\frac{1}{2}} \left( \int \prod_{s\in S} R(\bm x_s,\bm z;\bm z_s, h_2)^2 p^{(r)}( \underline{\bm x}, \bm z )  \prod_{s=1}^r \mathrm d \bm x_s \right)^{\frac{1}{2}} \prod_{s=1}^r \mathrm d \bm z_s . 
        \end{align} 
        Now, note that 
        \begin{align*}
            & \limsup_{n\rightarrow\infty} \int \prod_{s\in S} R(\bm x_s,\bm z;\bm z_s, h_2)^2 p^{(r)}( \underline{\bm x}, \bm z )  \prod_{s=1}^r \mathrm d \bm x_s\\
            &  = \limsup_{n\rightarrow\infty} \int \prod_{s\in S} \left(\|\bm z_s\|^2 R\left(\bm x_s,\bm z;\frac{\bm z_s}{\|\bm z_s\|}, h_2\|\bm z_s\| \right)^2 p^{(r)}( \underline{\bm x}, \bm z) \right)  \prod_{s=1}^r \mathrm d \bm x_s\tag*{(by \eqref{eq:density-residual})}\\
            & \leq \big(p(\bm z)\big)^{r-|S|} \prod_{s\in S} \|\bm z_s\|^4  \left( \sup_{\bm u:\|\bm u\|=1} \limsup_{n\rightarrow\infty} \int R(\bm x,\bm z;\bm u, h_2\|\bm z_s\|)^2 p(\bm x, \bm z )  \mathrm d \bm x\right)^{|S|} \lesssim \prod_{s\in S} \|\bm z_s\|^4. 
        \end{align*}
        Also, by Assumption \ref{K}, $$\int Q(\underline{\bm z}) \prod_{s\in S} \|\bm z_s\|^2 \prod_{s=1}^r \mathrm d \bm z_s = \left(\prod_{s\in S^c}  \int K( \| \bm z_s \| ) \mathrm d \bm z_s\right)\left( \prod_{s\in S}  \int \|\bm z_s\|^2 K( \| \bm z_s \| ) \mathrm d \bm z_s\right) = (d_Z\tau)^{|S|}.$$ 
                Hence, by Assumption \ref{YZ}, the finiteness of $\E[\psi^2(\underline{\bm X}) | \{ \bm Z_s = \bm z\}_{1 \leq s \leq r} ]<\infty$, we have, 
$$|A_n^{\bm z}(S)|\lesssim p_{\bm Z}(\bm z)^{r/2} \int Q (\underline{\bm z}) \prod_{s\in S} \|\bm z_s\|^2 \prod_{s=1}^r  \mathrm d \bm z_s \lesssim 1. $$ 
Hence, from \eqref{eq:Khfunction2}, we get $A_n=O(h_{2}^2)$. Therefore, from \eqref{eq:Kfunction1},  
\begin{align}
 \frac{1}{h_{2}^{ r d_Z }} \mathbb E \left[ \prod_{s=1}^r K\left(\frac{\| \bm z - \bm Z_s \|}{h_{2}}\right) \psi( \underline{ \bm X} ) \right] & = \int \int \prod_{s=1}^r  K\left( \| \bm z_s \| \right) \psi( \underline{\bm x}) \prod_{s=1}^r p_{ \bm X, \bm Z} ( \bm x_s, \bm z) \mathrm d \bm x_s \mathrm d \bm z_s + O(h_{2}^2) \nonumber \\ 
& =  \int \psi( \underline{\bm x}) \prod_{s=1}^r p_{ \bm X, \bm Z} ( \bm x_s, \bm z ) \mathrm d \bm x_s  + O(h_{2}^2)  \nonumber \\ 
& = p_{\bm Z} (\bm z)^r \E \left[\psi(\bm X_1, \ldots, \bm X_r)| \{\bm Z_s = \bm z\}_{1 \leq s \leq r} \right] + O(h_{2}^2) . \nonumber  
\end{align} 
This completes the proof of \eqref{eq:hz}.

To show \eqref{eq:hZ}, note from the arguments in \eqref{eq:Khfunction2} and \eqref{eq:AnS} and the assumption \eqref{eq:assumption-finiteness} we get $\int A_n^{\bm z} p_{\bm Z}(\bm z) \mathrm d \bm z = O(h_{2}^2)$. The result in \eqref{eq:hZ} then follows by taking expectation with respect to $\bm Z$ on both sides of \eqref{eq:hz}. 
\end{proof}

\begin{rem}
    Note that for Lemma \ref{lm:aux-1} to hold it is sufficient that $\mathcal{C}_K:= \int K(\|\bm u\|)\mathrm d\bm u <\infty$. In this case, from the proof of Lemma \ref{lm:aux-1} it follows that
       \begin{align}
        \lim_{h_{2} \rightarrow 0} \frac{1}{h_{2}^{ r d_Z } } \mathbb E \left[  \prod_{s=1}^r  K \left(\frac{\| \bm z - \bm Z_s \|}{h_{2}}\right) \psi(\bm X_1, \ldots, \bm X_r)  \right] = \mathcal{C}_K^r p_{\bm Z} (\bm z)^r \E \left[\psi(\bm X_1, \ldots, \bm X_r) | \{\bm Z_s = \bm z\}_{1 \leq s \leq r} \right]. 
        \label{eq:modified-aux-1}
    \end{align} 
\end{rem}

Using the above lemma we can derive the limit of the expectation of a homogeneous sum of $r$-fold kernel products.

\begin{lemma} \label{lm:aux-2} 
Suppose the assumptions the Lemma \ref{lm:aux-1} hold. Then 
\begin{align}\label{eq:hzsum}
\lim_{n \rightarrow \infty} \frac{1}{(nh_{2}^{ d_Z })^r } \mathbb E \left[ \sum_{\bm i\in [n]^r}  \prod_{s=1}^r  K\left(\frac{\| \bm z - \bm Z_{i_s} \|}{h_{2}}\right) \psi(\underline{\bm X}_{\bm i})  \right] = p_{\bm Z} (\bm z)^r \E\left[\psi(\bm X_1, \ldots, \bm X_r) | \{\bm Z_s = \bm z\}_{1 \leq s \leq r} \right], 
\end{align}
where $\bm i = (i_1,i_2,\ldots, i_r)$ and $\underline{\bm X}_{\bm i} = (\bm X_{i_1},\bm X_{i_2},\ldots, \bm X_{i_r})$.  
Moreover,  
\begin{align}\label{eq:hZsum}
& \lim_{n \rightarrow \infty} \frac{1}{(nh_{2}^{ d_Z })^r } \int \mathbb E \left[ \sum_{i_1,\ldots, i_r =1}^n  \prod_{s=1}^r  K\left(\frac{\| \bm z - \bm Z_{i_s} \|}{h_{2}}\right) \psi(\bm X_{i_1}, \ldots, \bm X_{i_r})  \right] p_{\bm Z}(\bm z) \mathrm d \bm z \nonumber \\ 
& = \int p_{\bm Z} (\bm z)^{r+1} \E\left[\psi(\bm X_1, \ldots, \bm X_r) | \{\bm Z_s = \bm z\}_{1 \leq s \leq r} \right] \mathrm d \bm z. 
\end{align} 
\end{lemma} 

\begin{proof} Denote by 
$\underline{\bm i}$ the distinct elements of $\bm i \in [n]^r$. Now,
    \begin{align} \label{eq:Kzsum}
       & \frac{1}{(nh_{2}^{ d_Z })^r } \mathbb E \left[ \sum_{\bm i\in [n]^r}  \prod_{s=1}^r  K\left(\frac{\| \bm z - \bm Z_{i_s} \|}{h_{2}}\right) \psi(\underline{\bm X}_{\bm i})  \right] \nonumber\\
       & = \frac{1}{(nh_{2}^{ d_Z })^r } \mathbb E \left[ \sum_{k=1}^r \sum_{\bm i \in [n]^r: |\underline{\bm i}| =  k}  \prod_{s=1}^r  K\left(\frac{\| \bm z - \bm Z_{i_s} \|}{h_{2}}\right) \psi(\underline{\bm X}_{\bm i})  \right]\nonumber\\
       & = \frac{1}{(nh_{2}^{ d_Z })^r }  \sum_{k=1}^r \sum_{\bm i \in [n]^r: |\underline{\bm i}| = k} \mathbb E \left[ \prod_{s=1}^r  K\left(\frac{\| \bm z - \bm Z_{i_s} \|}{h_{2}}\right) \psi(\underline{\bm X}_{\bm i})  \right]\nonumber\\
       & = \frac{1}{(nh_{2}^{ d_Z })^r }   \sum_{k=1}^r \sum_{\bm i \in [n]^r: |\underline{\bm i}| = k}   \mathbb E \left[ \prod_{ j \in \underline{\bm i}}  \overline{K}_j \left(\frac{\| \bm z - \bm Z_{j} \|}{h_{2}}\right) \psi(\underline{\bm X}_{\ell_{\underline{\bm i}}})  \right] , 
       \end{align}   
       where for a given $\bm i\in [n]^r$, 
       \begin{align}\label{eq:Kj}
       \overline{K}_j\left(\frac{\| \bm z - \bm Z_{j} \|}{h_{2}}\right) := \prod_{s: i_s = j} K\left(\frac{\| \bm z - \bm Z_{j} \|}{h_{2}}\right)
       \end{align} 
       and $ \ell_{\underline{\bm i}}=\bigcup_{j=1}^n\{s: i_s = j\}\in [n]^r$ is the vector obtained after sorting $\bm i$. Now, fix $1 \leq k \leq r$ and consider, 
      \begin{align}  \label{eq:Tk}
      T_{k, n} & := \sum_{\bm i \in [n]^r: |\underline{\bm i}| = k}   \mathbb E \left[ \prod_{ j \in \underline{\bm i}}  \overline{K}_j \left(\frac{\| \bm z - \bm Z_{j} \|}{h_{2}}\right) \psi(\underline{\bm X}_{\ell_{\underline{\bm i}}})  \right]    \\
       & = h_{2}^{k d_Z} \sum_{\bm i \in [n]^r: |\underline{\bm i}| = k}  \int \prod_{j=1}^{k} \overline{K}_j \left( \|\bm u_j \| \right) \psi( \bm x_{\ell_{\underline{\bm i}}}) \prod_{j=1}^{k} p_{\bm X,\bm Z}(\bm x_j, \bm z + h_{2} \bm u_j) \prod_{j=1}^{k} \mathrm d \bm x_j \mathrm d \bm u_j 
   \tag*{ (after the substitution $\bm z_j = \bm z + h_{2} \bm u_j$, for $1 \leq j \leq k$) } \nonumber \\
 &  = O(n^k h_{2}^{k d_Z}) , \label{eq:Tkn}
      \end{align}
 by arguments as in Lemma \ref{lm:aux-1}. This implies, 
   \begin{align*}  
    & \frac{1}{(nh_{2}^{ d_Z })^r } \mathbb E \left[ \sum_{\bm i\in [n]^r}  \prod_{s=1}^r  K\left(\frac{\| \bm z - \bm Z_{i_s} \|}{h_{2}}\right) \psi(\underline{\bm X}_{\bm i})  \right]  \nonumber \\ 
      & = \frac{ T_{r, n} }{(nh_{2}^{ d_Z })^r } +  O\left(\frac{1}{nh_{2}^{ d_Z }}\right) \tag*{ (by \eqref{eq:Kzsum} and \eqref{eq:Tkn}) } \nonumber \\ 
       & = \frac{1}{(nh_{2}^{ d_Z })^r }  (n)_r h_{2}^{r d_Z} \int \prod_{j=1}^{r}  K\left( \|\bm u_j \| \right) \psi(\bm x_{1}, \ldots, \bm x_{r}) \prod_{j=1}^{r} p_{\bm X,\bm Z}(\bm x_j,\bm z + h_{2} \bm u_j) \prod_{j=1}^r \mathrm d \bm x_j \mathrm d \bm u_j + O\left(\frac{1}{nh_{2}^{ d_Z }}\right) \tag*{ (where $(n)_r := n(n-1) \cdots (n-r+1)$) } \nonumber\\
       & \rightarrow \int \prod_{j=1}^{r}  K\left(\bm u_j\right) \psi(\bm x_{1}, \ldots, \bm x_{r}) \prod_{j=1}^{r} p_{\bm X,\bm Z}(\bm x_j, \bm z) \prod_{j=1}^r \mathrm d \bm x_j \mathrm d \bm u_j
       \tag*{\text{(by \eqref{eq:hZ} and \eqref{eq:Kfunction1})}}\\
       & = p_{\bm Z}(\bm z)^r \E[\psi(\bm X_{1}, \ldots, \bm X_{r})| \{ \bm Z_s = \bm z \}_{1 \leq s \leq r} ] . \nonumber \\ 
    \end{align*}
    This proves \eqref{eq:hzsum}.

The result in \eqref{eq:hZsum} can be proved similarly, by repeating the previous arguments and using \eqref{eq:hZ}. 
\end{proof}

Note that the limits obtained in Lemma \ref{lm:aux-1} is zero, when the function $\psi:(\R^{d_X})^r\rightarrow \R$ is {\it completely degenerate}, that is, 
\begin{align}\label{eq:phiX}
\E[\psi(\bm x_1,\ldots,\bm x_{r-1},\bm X_r)| \bm Z_r = \bm z] = 0, 
\end{align} 
almost surely for $\bm x_1,\ldots,\bm x_{r-1} \in \R^{d_X}$ and $\bm z \in \R^{d_Z}$. In this case one has to normalize the $r$-fold kernel product differently to obtain a non-trivial limit.

\begin{lemma} 
Suppose the assumptions of Lemma \ref{lm:aux-1} hold. Further, assume that the function $\psi:(\R^{d_X})^r\rightarrow \R$ is {\it completely degenerate} as defined in \eqref{eq:phiX}. Then the following hold: 
    \begin{align*}
      & \lim_{n \rightarrow \infty} \frac{1}{(nh_{2}^{ d_Z })^{r/2} } \mathbb E \left[ \sum_{\bm i \in [n]^r}  \prod_{s=1}^r  K\left(\frac{\| \bm z - \bm Z_{i_s} \|}{h_{2}}\right) \psi(\underline{\bm X}_{\bm i})  \right]\\
      & = 
\left\{
\begin{array}{cc}
      \displaystyle{\left(\int K^2(\|\bm u\|)\mathrm d \bm u\right)^{r/2}\left(\left\lfloor \frac{r}{2} \right \rfloor \right) ! \int \psi(\bm x_1,\bm x_1, \ldots, \bm x_{\lfloor r/2\rfloor},\bm x_{\lfloor r/2\rfloor}) \prod_{s=1}^{\lfloor r/2\rfloor} p_{\bm X,\bm Z}(\bm x_s, \bm z) \prod_{s=1}^{\lfloor r/2 \rfloor} \mathrm d \bm x_s},  &   \text{ if } r \text{ is even, }   \\
  0 &    \text{ if } r \text{ is odd. } 
\end{array}
\right. 
    \end{align*}
        \label{lm:aux-4}
\end{lemma}

\begin{proof}
As in the proof of Lemma \ref{lm:aux-2}, denote by 
$\underline{\bm i}$ the distinct elements of $\bm i \in [n]^r$. Then from \eqref{eq:Kzsum} we have, 
    \begin{align*}
        \mathbb E \left[ \sum_{\bm i\in [n]^r}  \prod_{s=1}^r  K\left(\frac{\| \bm z - \bm Z_{i_s} \|}{h_{2}}\right) \psi(\underline{\bm X}_{\bm i})  \right]  & =   \sum_{k=1}^r T_{k, n}, 
   \end{align*}  
   with $T_{k, n}$ as defined in \eqref{eq:Tk}. 
For any $k < \lfloor r/2 \rfloor$, by \eqref{eq:Tkn}, 
    $$\lim_{n\rightarrow \infty} \frac{T_{k, n}}{(nh_{2}^{ d_Z })^{r/2}} \lesssim \lim_{n\rightarrow \infty} \frac{(nh_{2}^{ d_Z })^k}{(nh_{2}^{ d_Z })^{r/2}} = 0. $$
Therefore, we only need to consider the case $k\geq \lfloor r/2 \rfloor$. Towards this, suppose $\bm i \in [n]^r$ such that $|\underline{\bm i}| = k >\lfloor r/2 \rfloor$. Then by the pigeonhole principle, there exists an index  which occurs exactly once in $\bm i$. Let $\nu_1(\bm i)$ be the number of indices in $\bm i$ that repeats exactly once and $\nu_2(\bm i)$ are those that repeats multiple times in $\bm i$. Clearly, $k = \nu_1(\bm i) + \nu_2(\bm i)$, for $\bm i \in [n]^r$ such that $|\underline{\bm i}| = k$. Now, recalling \eqref{eq:Tk}, let  
\begin{align}\label{eq:Tksum}
T_{k, n} = \sum_{\bm i \in [n]^r: |\underline{\bm i}| = k}   \mathbb E \left[ \prod_{ j \in \underline{\bm i}}  \overline{K}_j \left(\frac{\| \bm z - \bm Z_{j} \|}{h_{2}}\right) \psi(\underline{\bm X}_{\ell_{\underline{\bm i}}})  \right] = \sum_{\substack{0 \leq a , b \leq k \\  a + b = k}} T_{a, b, n}, 
 \end{align} 
   where 
   $$ T_{a, b, n} := \sum_{\substack{\bm i \in [n]^r \\ \nu_1(\bm i) = a, \nu_2(\bm i) = b  }}   \mathbb E \left[ \prod_{ j \in \underline{\bm i}}  \overline{K}_j \left(\frac{\| \bm z - \bm Z_{j} \|}{h_{2}}\right) \psi(\underline{\bm X}_{\ell_{\underline{\bm i}}})  \right] . $$ Then, by the degeneracy of $\psi$ all terms in the expansion \eqref{eq:Khfunction2} with $|S| \leq a-1$ are zero. Hence, $T_{a, b, n} = O((nh_{2}^{ d_Z })^{a + b} h_{2}^{2 a})$. Also, note that $r \geq a+2b$, hence $a +b -r/2 \leq a/2$. 
  This implies, for $k >\lfloor r/2 \rfloor$, 
   \begin{align}\label{eq:Tkab}
         \frac{T_{a, b, n}}{(nh_{2}^{ d_Z } )^{r/2}} = O\Big((nh_{2}^{ d_Z })^{a +b -r/2}h_{2}^{2 a}\Big) = O\Big((nh_{2}^{d_Z+4})^{a/2}\Big) \rightarrow 0  ,
    \end{align} 
    since $a \geq 1$ when $k >\lfloor r/2 \rfloor$ and $nh_{2}^{d_Z+4} \rightarrow 0$ (by Assumption \ref{K}). Hence, for $k > \lfloor r/2 \rfloor$, by \eqref{eq:Tksum}, 
    \begin{align}\label{eq:Tr2kn}
    \lim_{n\rightarrow \infty} \frac{T_{k, n}}{(nh_{2}^{ d_Z })^{r/2}}  = 0. 
    \end{align}

    Now, let us consider the case $k=\lfloor r/2 \rfloor$. If $r$ is odd, then by arguments as in \eqref{eq:Tkn}, 
    \begin{align*}
        \lim_{n\rightarrow \infty} \frac{T_{\lfloor r/2 \rfloor, n}}{(nh_{2}^{ d_Z })^{r/2}} = \lim_{n \rightarrow \infty} O\Big((nh_{2}^{d_Z})^{\lfloor r/2\rfloor-r/2}\Big) = 0. \end{align*}
    Therefore, suppose $r$ is even and $k=\lfloor r/2 \rfloor$. In this case, if $a \geq 1$, then from \eqref{eq:Tkab} we get, 
            \begin{align*}
    \lim_{n \rightarrow \infty} \frac{T_{a, b, n}}{(nh_{2}^{ d_Z })^{r/2}}  =  \lim_{n \rightarrow \infty} O\big((nh_{2}^{d_Z + 4})^{a/2}\big) = 0 ,   
    \end{align*}
     Hence, recalling \eqref{eq:Tksum}
     \begin{align}\label{eq:Tr2n}
       \frac{T_{\lfloor r/2 \rfloor, n}}{(nh_{2}^{ d_Z })^{r/2}} = \frac{T_{0, \lfloor r/2 \rfloor, n}}{(nh_{2}^{ d_Z })^{r/2}} + o(1). 
     \end{align}
     Also, by the change of variable $\bm z_s = \bm z + h_2 \bm u_s$, 
            \begin{align*}
       &  \frac{T_{0, \lfloor r/2 \rfloor, n}}{(nh_{2}^{ d_Z })^{r/2}} \nonumber \\ 
        & = \frac{{n\choose \lfloor r/2\rfloor}}{n^{\lfloor r/2\rfloor}} \left( \left\lfloor \frac{r}{2} \right \rfloor \right)! \int \prod_{s=1}^{r/2} K (\|\bm u_s\| )^2 \psi(\bm x_1,\bm x_1, \ldots, \bm x_{\lfloor r/2\rfloor},\bm x_{\lfloor r/2\rfloor}) \prod_{s = 1}^{\lfloor r/2 \rfloor} p_{\bm X,\bm Z}(\bm x_s, \bm z + h_{2} \bm u_s) \prod_{s = 1}^{\lfloor r/2 \rfloor} \mathrm d \bm x_s \mathrm d \bm u_s \nonumber \\ 
        &\rightarrow \left( \left\lfloor \frac{r}{2} \right \rfloor \right)! \left( \int K (\|\bm u\| )^2 \mathrm d \bm u \right)^{\left\lfloor r/2 \right \rfloor} \int \prod_{s = 1}^{r/2}  \psi(\bm x_1,\bm x_1, \ldots, \bm x_{\lfloor r/2\rfloor},\bm x_{\lfloor r/2\rfloor}) \prod_{s = 1}^{\lfloor r/2 \rfloor} p_{\bm X,\bm Z}(\bm x_s, \bm z ) \prod_{s = 1}^{\lfloor r/2 \rfloor} \mathrm d \bm x_s  .
    \end{align*}  
    Combining the above with \eqref{eq:Tr2kn} and \eqref{eq:Tr2n}, the proof of Lemma \ref{lm:aux-4} is completed. 
     \end{proof}

\begin{rem} \label{remark:srz}
Note that if $\psi$ is not symmetric, then Lemma \ref{lm:aux-4} holds with $\psi'$ instead of $\psi$, where $\psi'$ is defined as
    \begin{align}
        \psi'(\bm x_1,\ldots, \bm x_r) = \frac{1}{r!} \sum_{\tau \in S_r} \psi(\underline{\bm x}_{\tau} ) , 
        \label{eq:symver}
    \end{align}
        where $S_r$ is the set of permutations of $\{1, 2, \ldots, r\}$ and $\underline{\bm x}_{\tau} = ( \bm x_{\tau(1)}, \bm x_{\tau(2)}, \ldots \bm x_{\tau(r)} )$, for $\tau \in S_r$. 
\end{rem}

To derive the asymptotic properties of the conditional 2-sample $V$-statistic \eqref{eq:thetaestimate}, we need the 2-sample analogous of the previous lemmas. We begin with the 2-sample version of Lemma \ref{lm:aux-1}.

\begin{lemma} Fix $r_{1}, r_{2} \geq 1$ and suppose $(\bm X, \bm Y, \bm Z), (\bm X_1, \bm Y_1, \bm Z_1), \ldots, (\bm X_{r_1+r_2}, \bm Y_{r_1+r_2}, \bm Z_{r_1+r_2})$ are i.i.d. samples from $p_{\bm X,\bm Y, \bm Z}$. Further, suppose the following conditions hold:   
\begin{itemize}
\item $p_{\bm X, \bm Y, \bm Z}$ satisfies assumption \ref{YZ}; 

\item the kernel $K$ satisfies assumption \ref{K}; 

\item  $h_{1}$ and $h_{2}$ are bandwidths satisfying the conditions in assumption \ref{h}. 

\end{itemize} 
Consider a measurable function $\psi: (\mathbb R^{d_X})^{r_{1}}\times (\mathbb R^{d_X})^{r_{2}}  \rightarrow \mathbb R$ which is symmetric in its first $r_{1}$ and last $r_{2}$ coordinates and 
$$\sup_{ \bm y \in \mathbb R^{d_Y}, \bm z \in \mathbb R^{d_Z} }\mathbb E[\psi^2( \underline{\bm X}_{[r_{1}+r_{2}]}) \mid \{ \bm Y_s = \bm y \}_{ 1  \leq s \leq r_{1} }, \{ \bm Z_s = \bm z \}_{1 \leq s \leq r_{1}+r_{2}}  ] < \infty,$$
where $\underline{\bm X}_{[r_{1}+r_{2}]} = ( \bm X_1, \ldots, \bm X_{r_{1}+r_{2}})$. Then the following holds, as $h_1, h_{2} \rightarrow 0$, 
   \begin{align}\label{eq:Kyz}
       & \lim_{h_1, h_2 \rightarrow 0} \frac{1}{h_{1}^{r_{1} (d_Y+d_Z)}  h_{2}^{ r_{2} d_Z } } \mathbb E \left[  \prod_{s=1}^{r_{1}}  K\left(\frac{\| (\bm y,\bm z)  - (\bm Y_s,\bm Z_s)  \|}{h_{1}}\right) \prod_{s=r_{1}+1}^{r_{1}+r_{2}}  K\left(\frac{\| \bm z - \bm Z_s \|}{h_{2}}\right)    \psi( \underline{\bm X}_{[r_{1}+r_{2}]} )  \right] \nonumber \\
       & =  p_{\bm Y,\bm Z} (\bm y,\bm z)^{r_{1}} p_{\bm Z} (\bm z)^{r_{2}} \E \left[\psi( \underline{\bm X}_{[r_{1}+r_{2}]} )  \mid \{ \bm Y_s = \bm y_s \}_{ 1  \leq s \leq r_{1} }, \{ \bm Z_s = \bm z_s \}_{1 \leq s \leq r_{1}+r_{2}}  \right] . 
   \end{align} 
Further,  
   \begin{align}\label{eq:Kpyz}
     \displaystyle  &  \lim_{h_1, h_2 \rightarrow 0} \frac{ 1 }{h_{1}^{r_{1} (d_Y+d_Z)}  h_{2}^{ r_{2} d_Z } }  \E \left[ \prod_{s=1}^{r_{1}}  K\left( \frac{ \left\| (\bm Y, \bm Z)   - (\bm Y_s, \bm Z_s)  \right \|}{h_{1}}\right)  \prod_{s=r_{1}+1}^{r_{1}+r_{2}}  K\left(\frac{\| \bm Z - \bm Z_s \|}{h_{2}}\right) \psi( \underline{\bm X}_{[r_{1}+r_{2}]} )  \right] 
     \nonumber \\
       & = \int  p_{\bm Y,\bm Z} (\bm y,\bm z)^{r_{1}+1} p_{\bm Z} (\bm z)^{r_{2}} \mathbb E[\psi^2( \underline{\bm X}_{[r_{1}+r_{2}]}) \mid \{ \bm Y_s = \bm y \}_{ 1  \leq s \leq r_{1} }, \{ \bm Z_s = \bm z \}_{1 \leq s \leq r_{1}+r_{2}}  ] \mathrm d\bm y \mathrm d \bm z  .
   \end{align} 
    \label{lm:aux-3.5}
\end{lemma}

\begin{proof}
    The proof follows similar to Lemma \ref{lm:aux-1}. The details are omitted. 
      \end{proof}

\begin{rem} \label{remark:h} Note that if the function $\psi$ is conditionally degenerate of order 1 (as in Definition \ref{def:Vphi}), then the limits in \eqref{eq:Kyz} and \eqref{eq:Kpyz} are zero. In this case, by an expansion similar to \eqref{eq:Khfunction2} one can show  that 
\begin{align*}
& \frac{1}{h_{1}^{r_{1} (d_Y+d_Z)}  h_{2}^{ r_{2} d_Z } } \mathbb E \left[  \prod_{s=1}^{r_{1}}  K\left(\frac{\| (\bm y,\bm z)  - (\bm Y_s,\bm Z_s)  \|}{h_{1}}\right) \prod_{s=r_{1}+1}^{r_{1}+r_{2}}  K\left(\frac{\| \bm z - \bm Z_s \|}{h_{2}}\right)    \psi( \underline{\bm X}_{[r_{1}+r_{2}]} )  \right]  \nonumber \\ 
& = O(h_1^2 h_2^2) = O(h_1^4) , 
\end{align*}
since $h_2/h_1 \rightarrow 0$. Similarly, 
\begin{align*}
& \frac{1}{h_{1}^{r_{1} (d_Y+d_Z)}  h_{2}^{ r_{2} d_Z } } \mathbb E \left[  \prod_{s=1}^{r_{1}}  K\left(\frac{\| (\bm Y, \bm Z)  - (\bm Y_s,\bm Z_s)  \|}{h_{1}}\right) \prod_{s=r_{1}+1}^{r_{1}+r_{2}}  K\left(\frac{\| \bm Z - \bm Z_s \|}{h_{2}}\right)    \psi( \underline{\bm X}_{[r_{1}+r_{2}]} )  \right] \nonumber \\  
& = O(h_1^2 h_2^2) = O(h_1^4) .  
\end{align*} 
\end{rem}

Now, suppose the function $\psi:(\R^{d_X})^{r_{1}}\times (\R^{d_X})^{r_{2}}\rightarrow \R$, which is symmetric in its first $r_{1}$ and last $r_{2}$ arguments, is completely degenerate, as in Definition \ref{def:Vphi}. Then we have following  2-sample version of Lemma \ref{lm:aux-4}.

\begin{lemma} 
Suppose the assumptions of Lemma \ref{lm:aux-3.5} hold. Further, assume  that the function $\psi:(\R^{d_X})^{r_{1}}\times (\R^{d_X})^{r_{2}}\rightarrow \R$ is completely degenerate. Define 
$$T_{r_{1}, r_{2}, n}^{\bm y, \bm z} : = \mathbb E \left[ \sum_{\bm i \in [n]^{r_{1}}} \sum_{\bm j\in [n]^{r_{2}}}  \prod_{s=1}^{r_{1}}K\left(\frac{\| (\bm y,\bm z)  - (\bm Y_{i_s}, \bm Z_{i_s})  \|}{h_{1}}\right) \prod_{s=r_{1}+1}^{r_{1}+r_{2}} K\left(\frac{\| \bm z - \bm Z_{j_s} \|}{h_{2}}\right)  \psi( \underline{\bm X}_{(\bm i,\bm j)})  \right],$$
where  $\bm i = (i_1,i_2,\ldots, i_{r_{1}})$, $\bm j = (j_1, j_2,\ldots, j_{r_{2}})$, and 
$$\underline{\bm X}_{(\bm i, \bm j)} = (\bm X_{i_1},\bm X_{i_2},\ldots, \bm X_{i_{r_{1}}}; \bm X_{j_1},\bm X_{j_2},\ldots, \bm X_{j_{r_{2}}}).$$
Then the following hold:  
\begin{itemize}

\item If $r_{1} = 2 \ell_1$ and $r_{2} = 2 \ell_2$ are both even,  
    \begin{align*}
      & \lim_{n \rightarrow \infty} \frac{T_{r_{1}, r_{2}, n}^{\bm y, \bm z}}{ (nh_{1}^{ d_Y+d_Z })^{r_{1}/2} (nh_{2}^{ d_Z })^{r_{2}/2} }  \\ 
            & =C_1^{\ell_1} C_2^{\ell_2}
\ell_1 ! \ell_2  ! \int \psi^*(\bm x_1,\ldots,\bm x_{\ell_1 + \ell_2} ) \prod_{i=1}^{\ell_1} p_{\bm X,\bm Y,\bm Z}(\bm x_i,\bm y, \bm z)  \prod_{i= \ell_1+ 1}^{ \ell_1 + \ell_2 } p_{\bm X,\bm Z}(\bm x_i, \bm z)   \prod_{i=1}^{\ell_1 + \ell_2 } \mathrm d \bm x_i , 
    \end{align*}
    where $\psi^*(\bm x_1,\ldots, \bm x_{\ell_1+ \ell_2}) := \psi(\bm x_1,\bm x_1, \ldots,\bm x_{\ell_1},\bm x_{\ell_1}; \bm x_{\ell_1+1},\bm x_{\ell_1+1}, \ldots,\bm x_{\ell_1+\ell_2},\bm x_{\ell_1+\ell_2} ),$  
    $C_1 := \int_{\R^{d_Y+d_Z}} K^2(\|\bm u\| \mathrm d \bm u )$, and $C_2 := \int_{\R^{d_Z}} K^2(\|\bm v\|)\mathrm d \bm v$. 
 \item   If either one of $r_{2}$ or $r_{1}$ is odd, 
 \begin{align*}
      \lim_{n \rightarrow \infty} \frac{T_{r_{1}, r_{2}, n}^{\bm y, \bm z}}{ (nh_{1}^{ d_Y+d_Z })^{r_{1}/2} (nh_{2}^{ d_Z })^{r_{2}/2} } = 0. 
      \end{align*}
%
%

\end{itemize}
    \label{lm:aux-5}
\end{lemma}

\begin{proof}
   The proof is similar to Lemma \ref{lm:aux-4}. We provide a sketch, omitting some details: Denote by 
$\underline{\bm i}$ and $\underline{\bm j}$ the distinct elements of $\bm i \in [n]^{r_{1}}$ and $\bm j \in [n]^{r_{2}}$. Then 
    \begin{align} \label{eq:Kxyn}
   T_{r_{1}, r_{2}, n}^{\bm y, \bm z}  
      & =  \sum_{\substack{0 \leq a_1, b_1 \leq r_{1} \\ a_1 + b_1 \leq r_{1}}} \sum_{\substack{0 \leq a_2, b_2 \leq r_{2} \\ a_2 + b_2 \leq r_{2}}} T_{a_1, b_1, a_2, b_2, n} , 
        \end{align} 
       where 
       \begin{align*}
       & T_{a_1, b_1, a_2, b_2, n} \nonumber \\ 
       & := \sum_{\substack{\bm i \in [n]^{r_{1}} \\ \nu_1(\bm i) = a_1, \nu_2(\bm i) = b_1 }}  \sum_{\substack{\bm j \in [n]^{r_{2}} \\ \nu_1(\bm j) = a_2, \nu_2(\bm j) = b_2 }} \mathbb E \left[ \prod_{ s \in \underline{\bm i }}  \overline{K}_{s} \left(\frac{\| (\bm y, \bm z)  - ( \bm Y_{s}, \bm Z_{s})  \|}{h_{1}}\right) \prod_{ t \in \underline{\bm j }}  \overline{K}_{t} \left(\frac{\| \bm z - \bm Z_{t} \|}{h_{2}}\right)   \psi(\underline{\bm X}_{(\ell_{\underline{\bm i}}, \ell_{\underline{\bm j}})})  \right] , 
       \end{align*}
       with notation as in the proof of Lemma \ref{lm:aux-4}. 
       First note that when $a_1+b_1 \ne r_{1}/2$ or $a_2 + b_2 \ne r_{2}/2$, then as in the proof of Lemma \ref{lm:aux-4} we can show that 
        \begin{align*}
      \lim_{n \rightarrow \infty} \frac{ T_{a_1, b_1, a_2, b_2, n} }{ (nh_{1}^{ d_Y+d_Z })^{r_{1}/2} (nh_{2}^{ d_Z })^{r_{2}/2} } = 0 . 
      \end{align*} 
       Hence, suppose $r_{1} = 2 \ell_1$ and $r_{2} = 2 \ell_2$ are even and $a_1+b_1 = \ell_1$ and $a_2 + b_2 = \ell_2$. Then, by arguments similar to \eqref{eq:Tr2n} it can be shown that, when $a_1 \geq 1$ or $a_2 \geq 1$, 
           \begin{align*}
      \lim_{n \rightarrow \infty} \frac{ T_{a_1, b_1, a_2, b_2, n} }{ (nh_{1}^{ d_Y+d_Z })^{r_{1}/2} (nh_{2}^{ d_Z })^{r_{2}/2} } = 0 , 
      \end{align*} 
      This implies, from \eqref{eq:Kxyn},            
      \begin{align*} 
   \frac{T_{r_{1}, r_{2}, n}^{\bm y, \bm z} }{ (nh_{1}^{ d_Y+d_Z })^{r_{1}/2} (nh_{2}^{ d_Z })^{r_{2}/2} }  = \frac{ T_{0, r_{1}/2, 0, r_{2}/2, n} }{ (nh_{1}^{ d_Y+d_Z })^{r_{1}/2} (nh_{2}^{ d_Z })^{r_{2}/2}  } + o(1) . 
        \end{align*} 
            The result in Lemma \ref{lm:aux-5} then now follows from Lemma \ref{lm:aux-3.5}. 
   \end{proof}

\begin{rem}\label{remark:sryz}
For a function $\psi$ that is not necessarily symmetric, the result in Lemma \ref{lm:aux-5} holds with $\psi'$ instead of $\psi$, where $\psi'$ is the symmetrization of $\psi$ as defined in \eqref{eq:symver-2}.   
\end{rem}

\subsection{Proof of Proposition \ref{prop:main-1}}
\label{sec:empiricalprocesspf}

We begin the following simple fact about kernel density estimation.

\begin{lemma}
    Suppose the assumptions of Proposition \ref{prop:main-1} hold. If $p_{\bm Z}(\bm z)>0$, then (recalling \eqref{eq:kw})
    \begin{align}\label{eq:wpz}
    \frac{1}{nh_{2}^{ d_Z }} w(\bm z) = 
   \frac{1}{nh_{2}^{ d_Z }}\sum_{i=1}^n K\Big(\frac{\|\bm z-\bm Z_i\|}{h_{2}}\Big) \stackrel{L_2}{\rightarrow} p_{\bm Z}(\bm z) , 
   \end{align}
    as $n\rightarrow\infty$. Similarly, if $p_{\bm Y, \bm Z}(\bm y, \bm z)>0$, then 
    \begin{align}\label{eq:wpyz}
    \frac{1}{nh_{1}^{ (d_Y+d_Z) }} w( \bm y, \bm z) = \frac{1}{nh_{1}^{ (d_Y+d_Z) }} \sum_{i=1}^n K\Big(\frac{\|(\bm y, \bm z) -(\bm Y_i, \bm Z_i) \|}{h_{1}}\Big) \stackrel{L_2}{\rightarrow} p_{\bm Y, \bm Z}(\bm y, \bm z).
    \end{align}
    \label{lem:main-0}
\end{lemma}

\begin{proof}
    Note that by Lemma \ref{lm:aux-1},
    \begin{align*}
        & \E\left[\frac{1}{nh_{2}^{ d_Z }}\sum_{i=1}^n K\Big(\frac{\|\bm z-\bm Z_i\|}{h_{2}}\Big)\right] = \frac{1}{h_{2}^{d_Z}} \E\left[K\Big(\frac{\|\bm z-\bm Z_1\|}{h_{2}}\Big)\right] \rightarrow p_{\bm Z}(\bm z).
    \end{align*} 
    Also, by Lemma \ref{lm:aux-1},
    \begin{align*}
        \E\left[\frac{1}{nh_{2}^{ d_Z }}\sum_{i=1}^n K\Big(\frac{\|\bm z-\bm Z_i\|}{h_{2}}\Big)\right]^2 =&  \frac{n(n-1)}{n^2h_{2}^{2d_Z}} \E\left[K\Big(\frac{\|\bm z-\bm Z_1\|}{h_{2}}\Big) K\Big(\frac{\|\bm z-\bm Z_2\|}{h_{2}}\Big)\right] + \frac{n}{n^2h^{2d_Z}} \E\left[K^2\Big(\frac{\|\bm z-\bm Z_1\|}{h_{2}}\Big)\right]\\
        & \rightarrow p_{\bm Z}^2(\bm z).
    \end{align*}
    Combining these we get $\var[\frac{1}{nh_{2}^{ d_Z }}\sum_{i=1}^n K(\frac{\|\bm z-\bm Z_i\|}{h_{2}})] \rightarrow 0$, as $n\rightarrow\infty$. This proves \eqref{eq:wpz}. The result in \eqref{eq:wpyz} can be proved similarly.  
    \end{proof}

Now, we proceed with the proof of Proposition \ref{prop:main-1}. For any $f\in\mathcal{F}_{\bm z}$, we can write 
 \begin{equation*}
    \begin{split}
    S_n^{\bm z}(f) & :=  \sqrt{nh_{2}^{ d_Z } }\left(\int f(\bm x) \mathrm d\Tilde{\mathbb{P}}_{\bm X| \bm Z = \bm z}( \mathrm d \bm x)-\E_{\mathbb{P}_{\bm X| \bm Z = \bm z}} [ f(\bm X) ] \right) \nonumber \\ 
    & = \sqrt{nh_{2}^{ d_Z }}\left(\frac{\sum_{i=1}^n w_{\bm Z_i}( \bm z ) f({\bm X}_i)}{\sum_{i=1}^n w_{\bm Z_i}( \bm z ) }-\E_{\mathbb{P}_{\bm X| \bm Z =\bm z}} [f({\bm X})] \right) \nonumber \\ 
     & = \sqrt{nh_{2}^{ d_Z }}\left(\frac{\frac{1}{nh_{2}^{ d_Z }}\sum_{i=1}^n w_{\bm Z_i}( \bm z ) f({\bm X}_i)}{\frac{1}{nh_{2}^{ d_Z }}\sum_{i=1}^n w_{\bm Z_i}( \bm z ) }-\E_{\mathbb{P}_{\bm X| \bm Z =\bm z}} [f({\bm X})]\right)\\
     & = \frac{ \left(\frac{1}{\sqrt{nh_{2}^{ d_Z }}}\sum_{i=1}^n w_{\bm Z_i}( \bm z ) f({\bm X}_i) - \frac{1}{\sqrt{nh_{2}^{ d_Z }}}\sum_{i=1}^n w_{\bm Z_i}( \bm z ) \E_{\mathbb{P}_{\bm X| \bm Z =\bm z}} [f({\bm X})]\right)}{ \left(\frac{1}{nh_{2}^{ d_Z }}\sum_{i=1}^n w_{\bm Z_i}( \bm z )  \right) } . 
    \end{split}
\end{equation*}
Similarly, for $g \in \mathcal F_{\bm y, \bm z}$ we have, 
\begin{equation*}
    \begin{split}
    S_n^{\bm y, \bm z}(g) & := \sqrt{nh_{1}^{ d_Y + d_Z } }\left(\int f(\bm x) \mathrm d\Tilde{\mathbb{P}}_{\bm X| \bm Y=\bm y, \bm Z = \bm z}( \mathrm d \bm x)-\E_{\mathbb{P}_{\bm X| \bm Y= y, \bm Z = \bm z}} [ f(\bm X) ] \right) \nonumber \\ 
    & = \frac{ \left(\frac{1}{\sqrt{nh_{1}^{ d_Y + d_Z }}}\sum_{i=1}^n w_{ \bm Y_i, \bm Z_i}( \bm y, \bm z ) g({\bm X}_i) - \frac{1}{\sqrt{nh_{1}^{ d_Y + d_Z }}}\sum_{i=1}^n w_{ \bm Y_i, \bm Z_i}( \bm y, \bm z ) \E_{\mathbb{P}_{\bm X| \bm Y = \bm y, \bm Z =\bm z}} [g({\bm X})]\right)}{ \left(\frac{1}{nh_{1}^{ d_Y+d_Z }}\sum_{i=1}^n w_{\bm Y_i, \bm Z_i}( \bm y, \bm z )  \right) } . 
    \end{split}
\end{equation*} 

Hence, given a finite collection of functions $f_1, f_2, \ldots, f_{K_1} \in \mathcal F_{\bm z}$ and a finite collection of functions $g_1, g_2, \ldots, g_{K_2} \in \mathcal F_{\bm y, \bm z}$, 
to obtain the distribution of $\{\{S_n^{\bm z}(f_s): 1 \leq s \leq K_1 \}, \{S_n^{\bm y, \bm z}(g_t): 1 \leq t \leq K_2 \}\}, $ 
we need to study the joint asymptotic behavior of 
\begin{align*}
& \Bigg\{ \left(\frac{1}{\sqrt{nh_{2}^{ d_Z }}}\sum_{i=1}^n w_{\bm Z_i}( \bm z ) , \frac{1}{\sqrt{nh_{2}^{ d_Z }}}\sum_{i=1}^n w_{\bm Z_i}( \bm z ) f_s({\bm X}_i) , 1 \leq s \leq K_1 \right), \nonumber \\
& ~~~~~~~~ \left( \frac{1}{\sqrt{nh_{1}^{ d_Y+ d_Z }}}\sum_{i=1}^n w_{\bm Y_i, \bm Z_i}( \bm y, \bm z ), 
\frac{1}{\sqrt{nh_{1}^{ d_Y + d_Z }}}\sum_{i=1}^n w_{ \bm Y_i, \bm Z_i}( \bm y, \bm z ) g_t({\bm X}_i): 1 \leq t \leq K_2 \right) \Bigg\} . 
\end{align*} 
Due to the Cramer-Wold device, it is enough to derive the limiting distribution of 
$$\frac{1}{\sqrt{nh_{2}^{ d_Z }}}\sum_{i=1}^n w_{\bm Z_i}( \bm z ) \left( \sum_{s=1}^{K_1} a_s+ \sum_{s=1}^{K_1} a_s' f_s({\bm X}_i) \right) + 
\frac{1}{\sqrt{nh_{1}^{ d_Y+ d_Z }}}\sum_{i=1}^n w_{\bm Y_i, \bm Z_i}( \bm y, \bm z ) \left( \sum_{t=1}^{K_2} b_t  + \sum_{t=1}^{K_2} b_t' g_t ({\bm X}_i) \right),$$
for every $a_1, a_2, \ldots, a_{K_1}, a_1', a_2' \ldots, a_{K_1}' \in \R$ and $b_1, b_2, \ldots, b_{K_2}, b_1', b_2' \ldots, b_{K_2}' \in \R$. This means that it suffices to obtain the limiting distribution of 
$$\left(\frac{1}{\sqrt{nh_{2}^{ d_Z }}}\sum_{i=1}^n w_{\bm Z_i}( \bm z )  f({\bm X}_i), \frac{1}{\sqrt{nh_{1}^{ d_Y+ d_Z }}}\sum_{i=1}^n w_{\bm Y_i, \bm Z_i}( \bm y, \bm z ) g(\bm X_i)\right) , $$ 
for any $f \in \F^{\bm z}$ and $g \in \mathcal F^{\bm y, \bm z}$. For this, consider, for $\alpha_1, \alpha_2 \in \R$, 
    \begin{align*}      
     \frac{1}{\sqrt{n}} \sum_{i=1}^n \Big(\alpha_1 \frac{1}{\sqrt{h_{2}^{d_Z}}} w_{\bm Z_i}(\bm z) f(\bm X_i) + \alpha_2 \frac{1}{\sqrt{h_{1}^{d_Y+d_Z}}} w_{\bm Y_i,\bm Z_i}(\bm y,\bm z) g(\bm X_i)\Big)  = \frac{1}{\sqrt{n}} \sum_{i=1}^n V_{n, i} . 
   \end{align*}
     Note that $\{ V_{n, i} - \E[V_{n, i}]: 1\leq i\leq n\}$ is a triangular array of random variables which are row-wise i.i.d. with variance $\sigma_n^2 = \var[V_{n, 1}]$. To compute the limit of $\sigma_n^2$, first note by that Lemma \ref{lm:aux-3.5}, as $n \rightarrow \infty$, 
     $$\E\left[ \frac{w_{\bm Z_1}(\bm z) f(\bm X_1)}{\sqrt{h_{2}^{d_Z}}}  \right] \rightarrow 0 \text{ and } \E\left[\frac{w_{\bm Y_1,\bm Z_1}(\bm y,\bm z) g(\bm X_1)}{\sqrt{h_{1}^{d_Y+d_Z}}} \right] \rightarrow 0.$$
     Also, by that Lemma \ref{lm:aux-3.5}, as $n \rightarrow \infty$, 
     $$\E\left[ \frac{w_{\bm Z_1}(\bm z)^2 f(\bm X_1)^2}{h_{2}^{d_Z}}  \right] \rightarrow \Big(\int K^2(\| \bm u \|)\mathrm d \bm u\Big) p_{\bm Z}(\bm z) \E_{\P_{\bm X| \bm Z =\bm z}}[f(\bm X)^2]$$ and 
      $$\E\left[ \frac{w_{\bm Y_1, \bm Z_1}(\bm y, \bm z)^2 g(\bm X_1)^2}{h_{1}^{d_Y + d_Z}}  \right] \rightarrow \Big(\int K^2(\|(\bm u, \bm v) \|)\mathrm d \bm u \mathrm d \bm v \Big) p_{\bm Y, \bm Z}(\bm y, \bm z) \E_{\P_{\bm X| \bm Y = \bm y, \bm Z =\bm z}}[g(\bm X)^2]  . $$ 
        Further, note that 
     \begin{align*}
         & \E\left[\frac{w_{\bm Z_1}(\bm z) f(\bm X_1)}{\sqrt{h_{2}^{d_Z}}} \cdot \frac{w_{\bm Y_1,\bm Z_1}(\bm y,\bm z) g(\bm X_1)}{\sqrt{h_{1}^{d_Y+d_Z}}}  \right] \\ 
         & = \sqrt{\frac{1}{h_1^{d_Y+d_Z} h_2^{d_Z}}} \int K\left(\frac{\|\bm z- \bm z_1\|}{h_2}\right) K \left( \frac{\|(\bm y, \bm z) - (\bm y_1, \bm z_1) \|}{h_1} \right) f(\bm x)g(\bm x) p_{\bm X,\bm Y,\bm Z}(\bm x,\bm y_1, \bm z_1) \mathrm d\bm x\mathrm d \bm y_1 \mathrm d \bm z_1 \nonumber \\ 
         & = \sqrt{h_1^{d_Y+d_Z} h_2^{d_Z} } \int K(\|\bm u\|) K(\|(\bm u, \bm v)\|) f(\bm x)g(\bm x) p_{\bm X,\bm Y,\bm Z}(\bm x,\bm y+h_1 \bm u, \bm z + h_2 \bm v) \mathrm d\bm x\mathrm d \bm u \mathrm d \bm v \tag*{ (by the change of variables $\bm y_1 = \bm y +  h_1 \bm u$ and $\bm z_1 = \bm z +  h_2 \bm v$)} \\ 
                  & \rightarrow 0 , 
     \end{align*} 
    as $n\rightarrow \infty$, under Assumption \ref{YZ} and \ref{h}.  
    Hence, by Lemma \ref{lm:aux-3.5} we have,
     \begin{align}
     \sigma_n^2 & \rightarrow \sigma^2 \nonumber \\ 
     & := \Big(\int K^2(\|\bm u\|)\mathrm d \bm u\Big) \alpha_1^2 p_{\bm Z}(\bm z) \E_{\P_{\bm X| \bm Z =\bm z}}[f(\bm X)^2] + \Big(\int K^2(\| ( \bm u , \bm v) \|)\mathrm d \bm u \mathrm d \bm v\Big) \alpha_2^2 p_{\bm Y, \bm Z}(\bm y,\bm z) \E_{{\mathbb P}_{\bm X | \bm Y=\bm y, \bm Z = \bm z}}[g(\bm X)^2] . \nonumber 
     \end{align}
          Then by Lindeberg's central limit theorem, 
     $$\frac{1}{\sqrt{n}} \sum_{i=1}^n \Big( V_{n, i} - \E[V_{n, i}]\Big) \stackrel{D}{\rightarrow} N(0,\sigma^2).$$
     Also, under assumption \ref{h} we have $nh_{1}^{d_Y+d_Z+4}\rightarrow 0$ and $nh_{2}^{d_Z+4}\rightarrow 0$, as $n\rightarrow\infty$. Therefore, by equation \eqref{eq:Khfunction2}
     \begin{align*}
         & \sqrt{nh_{2}^{d_Z}}\left(\E\left[\frac{1}{h_{2}^{d_Z}} w_{\bm Z_1}(\bm z) f(\bm X)\right] - \int f(\bm x) p_{\bm X,\bm Z}(\bm x, \bm z)\mathrm d \bm x\right) \rightarrow 0 ,  \\ 
         &\sqrt{nh_{1}^{d_Y+d_Z}}\left(\E\left[\frac{1}{h_{1}^{d_Y+d_Z}} w_{\bm Y_1,\bm Z_1}(\bm y, \bm z) g(\bm X)\right] - \int g(\bm x) p_{\bm X,\bm Y,\bm Z}(\bm x,\bm y, \bm z)\mathrm d \bm x\right) \rightarrow 0.
     \end{align*} 
Moreover, by Lemma \ref{lem:main-0}, $$\frac{1}{nh_{2}^{ d_Z }}\sum_{i=1}^n w_{\bm Z_i}( \bm z ) \stackrel{P}\rightarrow p_{\bm Z}(\bm z) \text{ and } \frac{1}{nh_{1}^{ d_Y+d_Z }}\sum_{i=1}^n w_{\bm Y_i, \bm Z_i}( \bm y, \bm z ) \stackrel{P}\rightarrow p_{\bm Y, \bm Z}( \bm y, \bm z).$$  Hence, using the Cramer-Wold device and Slutsky's theorem we get the result in Proposition \ref{prop:main-1}. 
\hfill $\Box$

\subsection{Proof of Proposition \ref{prop:main-8}}
\label{sec:Vpf}

We begin with the following result.

\begin{lemma} 
Suppose the assumptions of Lemma \ref{lm:aux-3.5} hold. Further, assume  that the function $\psi:(\R^{d_X})^{r_{1}}\times (\R^{d_X})^{r_{2}}\rightarrow \R$ is completely degenerate as in Definition \ref{def:Vphi}. If $p_{\bm Z}(\bm z)>0$ and $p_{\bm Y,\bm Z}(\bm y,\bm z)>0$, then  
    \begin{align}
    & \int \phi(\bm x_1, \ldots, \bm x_{r_{1}}; \bm{x}'_1,\ldots, \bm{x}'_{r_{2}})\prod_{i=1}^{r_{1}} \mathrm d \hat {\mathbb G}(\bm x_i) \prod_{j=1}^{r_{2}} \mathrm d \hat {\mathbb G}'(\bm{x}'_j) \nonumber \\ 
    & \stackrel{D} \rightarrow \int \phi({\bm x_1, \ldots, \bm x_{r_{1}}; \bm{x}'_1\ldots, \bm{x}'_{r_{2}}})\prod_{i=1}^{r_{1}} \mathrm d\mathbb{G}(\bm x_i) \prod_{j=1}^{r_{1}} \mathrm d\mathbb{G}'(\bm{x}'_j),
    \end{align}
    as $n \rightarrow \infty$, where 
    \begin{align}\label{eq:G}
    \hat{\mathbb G} = \sqrt{nh_{1}^{d_Y+d_Z}}(\Tilde{\mathbb P}_{\bm X | \bm Y=\bm y, \bm Z = \bm z} - \mathbb{P}_{\bm X| \bm Y= \bm y, \bm Z = \bm z})  \text{ and } \hat{\mathbb G}'= \sqrt{nh_{2}^{ d_Z }}(\Tilde{\mathbb{P}}_{\bm X| \bm Z = \bm z} - \mathbb{P}_{\bm X| \bm Z = \bm z})  , 
    \end{align}
    are the conditional empirical processes as defined in \eqref{eq:conditional-emp-processz} and \eqref{eq:conditional-emp-processyz}, respectively; and 
    $$\mathbb{G} = \mathbb{G}_{\mathbb{P}_{\bm X| \bm Y = \bm y, \bm Z = \bm z}} \text{ and } \mathbb{G}' =  \mathbb{G}_{\mathbb{P}_{\bm X| \bm Z = \bm z}} , $$ are as defined in Proposition \ref{prop:main-1}.        \label{lem:main-6}
\end{lemma}

The proof of Lemma \ref{lem:main-6} is given in Appendix \ref{sec:Gpf}. We will first use this result to prove Proposition \ref{prop:main-8}. For this, let us define,
\begin{align*}
    V_n(\phi) & = \int \phi(\bm x_{1}, \ldots, \bm x_{r_{1}}; \bm{x}'_{1},\ldots, \bm{x}'_{r_{2}})  \prod_{s=1}^{r_{1}}\mathrm d \Tilde{\P}_{\bm X\mid (\bm Y,\bm Z)  = (\bm y,\bm z) }(\bm x _s) \prod_{s=r_{1}+1}^{r_{1}+r_{2}}\mathrm d \Tilde{\P}_{\bm X| \bm Z = \bm z}(\bm{x}'_s)  ,
\end{align*} 
Now let us define,
\begin{align*}
    \psi(t,s) = \int \phi(\bm x_{1}, \ldots, \bm x_{r_{1}}; \bm{x}'_{1},\ldots, \bm{x}'_{r_{2}}) \prod_{s=1}^{r_{1}} \mathrm d \Delta_1(\bm x_s) \prod_{s=r_{1}+1}^{r_{1}+r_{2}} \Delta_2(\bm{x}'_s),
\end{align*}
where $$\Delta_1 = {\mathbb P}_{\bm X | \bm Y=\bm y, \bm Z = \bm z}+s\sqrt{nh_{1}^{d_Y+d_Z}}\big(\Tilde{\P}_{\bm X| \bm Y = \bm y, \bm Z=\bm z} - {\mathbb P}_{\bm X | \bm Y=\bm y, \bm Z = \bm z}\big)$$ 
and $\Delta_2 = \P_{\bm X| \bm Z =\bm z}+t\sqrt{nh_{2}^{d_Z}}\big(\Tilde{\P}_{\bm X|\bm Z=\bm z}-\P_{\bm X| \bm Z =\bm z}\big)$. 
Note that $\psi(t,s)$ is a bivariate polynomial in $(t,s)$ with coefficient of $t^is^j$ being 
\begin{align*}
    {r_{1}\choose i}{r_{2}\choose j} \int \phi_{i,j}(\bm x_1,\ldots,\bm x_i; \bm{x}'_1,\ldots, \bm{x}'_{j}) \prod_{s=1}^{i} \mathrm d \hat{\mathbb G}(\bm x_s) \prod_{s=1}^{j} \mathrm d \hat{\mathbb G}'(\bm{x}'_s) , 
\end{align*}
where $\phi_{i,j}$ is the $(i, j)$-th Hoeffding's projection of the symmetrized function $\phi'$ (recall \eqref{eq:symver-2}). 
Hence, by Lemma \ref{lem:main-6}, 
$$\int \phi_{i,j}(\bm x_1,\ldots,\bm x_i; \bm{x}'_1,\ldots, \bm{x}'_{j}) \prod_{s=1}^{i} \mathrm d \hat{\mathbb G}(\bm x_s) \prod_{s=1}^{j} \mathrm d \hat{\mathbb G}'(\bm{x}'_s) \rightarrow \int \phi_{i,j}(\bm x_1,\ldots,\bm x_i; \bm{x}'_1,\ldots, \bm{x}'_{j}) \prod_{s=1}^{i} \mathrm d \mathbb G(\bm x_s) \prod_{s=1}^{j} \mathrm d  \mathbb G'(\bm{x}'_s) .
$$
In particular, this implies the coefficients of $t^is^j$ in $\psi(s,t)$ are tight, for all $i+j\geq 1$. Hence, noting that $\psi(0,0) = \theta$,  
\begin{align*}
V_n(\phi) = \psi \left( \frac{1}{(nh_{1})^{(d_Y+d_Z)/2}} , \frac{1}{(nh_{2})^{d_Z/2}} \right) = \psi(0,0) + R_n = \theta + R_n, 
\end{align*} 
where $R_n = O_P\big((nh_{1}^{d_Y+d_Z})^{-1/2}\big)$. This completes the proof of Proposition \ref{prop:main-8} (1).

Next, suppose $\phi$ is conditionally degenerate of order $k \geq 0$. (Note that $k=0$ means $\phi$ is conditionally non-degenerate). Then from Definition \ref{def:Vphi}, the $(i, j)$-th Hoeffding's projection $\phi_{i,j}= 0$, for all $i+j\leq k$. Then by a Taylor expansion,  
\begin{align*}
    & V_n(\phi) \nonumber \\ 
    &= \theta + \sum_{\substack{(i, j) \\ i+j=k+1}} \frac{{r_{1}\choose i}{r_{2}\choose j}}{n^{(k+1)/2} h_{1}^{ i (d_Y+d_Z)/2} h_{2}^{ j d_Z/2} } \int \phi_{i,j} (\bm x_1,\ldots,\bm x_i; \bm{x}'_1,\ldots, \bm{x}'_{j}) \prod_{s=1}^{i} \mathrm d \hat{\mathbb G}(\bm x_s) \prod_{s=1}^{j} \mathrm d \hat{\mathbb G}'(\bm{x}'_s) + R_n,
\end{align*}
where $$R_n=O_P\left( \frac{1}{(nh_{1}^{d_Y+d_Z})^{(k+2)/2}} + \frac{1}{(nh_{2}^{d_Z})^{(k+2)/2}} \right),$$ by Lemma \ref{lem:main-6}. Since by Assumption \ref{h}, $h_{1}^{d_Y+d_Z}/h_{2}^{d_Z}\rightarrow0$, we have,
\begin{align}
    & \big(nh_{1}^{(d_Y+d_Z)}\big)^{(k+1)/2} (V_n(\phi) -\theta) \nonumber \\
    & =  \sum_{i+j=k+1} \frac{ {r_{1}\choose i}{r_{2}\choose j} h_{1}^{(k+1-i) (d_Y+d_Z)/2} }{ h_{2}^{ (k+1-i) d_Z/2}  }\int \phi_{i,j} (\bm x_1,\ldots,\bm x_i; \bm{x}'_1,\ldots, \bm{x}'_{j}) \prod_{s=1}^{i} \mathrm d \hat{\mathbb G}(\bm x_s) \prod_{s=1}^{j} \mathrm d \hat{\mathbb G}'(\bm{x}'_s) + o_P(1) \nonumber \\
    & \stackrel{D}{\rightarrow} {r_{1}\choose k+1} \int \phi_{k+1, 0} (\bm{x}_1,\ldots, \bm{x}_{k+1}) \prod_{i=1}^{k+1} \mathrm d {\mathbb G}(\bm{x}_i). \label{eq:dYZtheta}
\end{align}
The result in Proposition \ref{prop:main-8} (2) follows by substituting $k=0$ in \eqref{eq:dYZtheta}. Also, Proposition \ref{prop:main-8} (3) follows from \eqref{eq:dYZtheta} by noting that $\theta=0$, for $k \geq 1$. 
\hfill $\Box$

\subsubsection{Proof of Lemma \ref{lem:main-6} }
\label{sec:Gpf}

Let $f_0=1, f_1,f_2,\ldots$ be an orthogonal basis of $L_2(\mathbb{R}^{d_X}, \mathbb{P}_{ \bm X| \bm Y=\bm y, \bm Z=\bm z})$ and $g_0=1,g_1,g_2,\ldots$ be the same for $L_2(\mathbb{R}^{d_X}, \mathbb{P}_{\bm X| \bm Z = \bm z})$. Since $\phi$ is completely degenerate, it can represented as 
\begin{align}\label{eq:phifg}
\phi = \sum_{i_1,\ldots,i_{r_{1}}, j_1,\ldots, j_{r_{2}} \geq 1}\langle \phi,f_{i_1}\times\cdots\times f_{i_{r_{1}}} \times g_{j_1}\times\cdots\times g_{j_{r_{2}}}\rangle f_{i_1}\times\cdots\times f_{i_{r_{1}}} \times g_{j_1}\times\cdots\times g_{j_{r_{2}}}, 
\end{align}
where the equality holds in the $L_2$ sense. We define the truncated version of $\phi$ as follows: For $L \geq 1$, 
\begin{align}\label{eq:phiL}
    \phi_L = \sum_{1\leq i_1,\ldots,i_{r_{1}}, j_1,\ldots,j_{r_{2}} \leq L}\langle \phi,f_{i_1}\times\cdots\times f_{i_{r_{1}}} \times g_{j_1}\times\cdots\times g_{j_{r_{2}}}\rangle f_{i_1}\times\cdots\times f_{i_{r_{1}}} \times g_{j_1}\times\cdots\times g_{j_{r_{2}}}.
\end{align} 
Note that by the orthogonality of $\{f_s\}_{s\geq 1}$ and $\{g_s\}_{s\geq 1}$ we have,
    \begin{align}\label{eq:Pphi}
        & \| \phi-\phi_L \|_2^2 \nonumber \\ 
        & = \int [ \phi(\bm x_1,\ldots,\bm x_{r_{1}};\bm{x}'_1,\ldots,\bm{x}'_{r_{2}})-\phi_L(\bm x_1,\ldots,\bm x_{r_{1}};\bm{x}'_1,\ldots, \bm{x}'_{r_{2}}) ]^2 \prod_{i=1}^{r_{1}}\mathrm d \bm {\mathbb P}_{\bm X | \bm Y=\bm y, \bm Z = \bm z}(\bm{x}_i) \prod_{i=1}^{r_{2}} \mathrm d \bm \P_{\bm X| \bm Z =\bm z}(\bm{x}'_i)  \nonumber \\
        & \rightarrow 0,
    \end{align}
    as $L\rightarrow\infty$. 
    
    The proof of Lemma \ref{lem:main-6} proceeds in 3 steps: First we show that the integral of $\phi_L$ with respect to the conditional empirical processes $\hat{\mathbb G}$ and $\hat{\mathbb G}'$ converges to its limiting counterpart with $\mathbb G$ and $\mathbb G'$, as $n \rightarrow \infty$. To this end, using Proposition \ref{prop:main-1} and the  continuous mapping theorem note that 
    \begin{align*}
        \left\{ \prod_{s=1}^{r_{1}} \int f_{i_s}(\bm u) \mathrm d \hat{\mathbb G} (\bm u) \prod_{s=1}^{r_{2}} \int g_{j_s}(\bm u) \mathrm d \hat{\mathbb G}' (\bm u) \right\} \stackrel{D}{\rightarrow} \left\{ \prod_{s=1}^{r_{1}} \int f_{i_s}(\bm u) \mathrm d {\mathbb G} (\bm u) \prod_{s=1}^{r_{2}}\int g_{j_s}(\bm u) \mathrm d {\mathbb G}' (\bm u) \right\},
    \end{align*}
    where $\hat{\mathbb G}$ and $\hat{\mathbb G}'$ are the conditional empirical processes and $\mathbb G$ and ${\mathbb G}'$ the corresponding Gaussian processes as defined in the statement Lemma \ref{lem:main-6}. Now, recalling \eqref{eq:phiL} and applying the continuous mapping theorem again gives, for every $L \geq 1$, 
    \begin{align}
        & \int \phi_L(\bm x_1,\ldots,\bm x_{r_{1}};\bm{x}'_1,\ldots, \bm{x}'_{r_{2}}) \prod_{i=1}^{r_{1}} \mathrm d \hat{\mathbb G} (\bm x_i) \prod_{i=1}^{r_{2}} \mathrm d \hat{\mathbb G}' (\bm{x}'_i) \nonumber \\
        & = \sum_{1\leq i_1,\ldots,i_{r_{1}}, j_1,\ldots,j_{r_{2}} \leq L}\langle \phi,f_{i_1}\times\cdots\times f_{i_{r_{1}}} \times g_{j_1}\times\cdots\times g_{j_{r_{2}}}\rangle \prod_{s=1}^{r_{1}} \int f_{i_s}(\bm u) \mathrm d \mathbb {\hat G}(\bm u)\times \prod_{s=1}^{r_{2}}\int g_{j_s}(\bm u)\mathrm d\mathbb {\hat G}' (\bm u) \nonumber \\
         & \stackrel{D}{\rightarrow} \sum_{1\leq i_1,\ldots,i_{r_{1}}, j_1,\ldots,j_{r_{2}} \leq L}\langle \phi,f_{i_1}\times\cdots\times f_{i_{r_{1}}} \times g_{j_1}\times\cdots\times g_{j_{r_{2}}}\rangle \prod_{s=1}^{r_{1}} \int f_{i_s}(\bm u) \mathrm d \mathbb {G}(\bm u)\times \prod_{s=1}^{r_{2}}\int g_{j_s}(\bm u)\mathrm d\mathbb { G}' (\bm u) \nonumber \\
        & = \int \phi_L(\bm x_1,\ldots,\bm x_{r_{1}};\bm{x}'_1,\ldots, \bm{x}'_{r_{2}}) \prod_{i=1}^{r_{1}} \mathrm d {\mathbb G} (\bm x_i) \prod_{i=1}^{r_{2}} \mathrm d {\mathbb G}' (\bm{x}'_i), 
        \label{eq:phin}
    \end{align}
    as $n\rightarrow\infty$. 
    
   The second step is to show that the difference $\psi_L := \phi-\phi_L$ under $\mathbb{G}$ and $\mathbb{G'}$ is negligible in $L_2$, as $L \rightarrow \infty$. Towards this, for a $L_2$ function $f$ define,   
    \begin{align}\label{eq:Vf}
    V(f) := \int f(\bm x_1,\ldots,\bm x_{r_{1}};\bm{x}'_1,\ldots, \bm{x}'_{r_{2}}) \prod_{i=1}^{r_{1}} \mathrm d {\mathbb G} (\bm x_i) \prod_{i=1}^{r_{2}} \mathrm d {\mathbb G}' (\bm{x}'_i).
    \end{align}
    Since $\{\int f_s (\bm u) \mathrm d \mathbb G(\bm u) \}_{1 \leq s \leq r_{1}}$ and $\{\int g_s (\bm u) \mathrm d \mathbb G'(\bm u) \}_{1 \leq s \leq r_{2}}$ are two independent sequence of mean zero Gaussian random variables with variance $c_K({\bm z})$ and $c_K({\bm y,\bm z})$, respectively, for any $L \geq 1$, 
    \begin{align*}
        \E\left[V(\psi_L)\right]^2 & \leq C_{r_{1}, r_{2}, K} \sum_{i_1,\ldots,i_{r_{1}}, j_1,\ldots,j_{r_{2}} \geq 1} \big|\langle \psi_L, f_{i_1}\times\cdots\times f_{i_{r_{1}}} \times g_{j_1}\times\cdots\times g_{j_{r_{2}}}\rangle\big|^2 \nonumber \\
        & \leq C_{r_{1}, r_{2}, K} \| \psi_L \|_2^2, 
    \end{align*}
    for some constant $C_{r_{1}, r_{2}, K} > 0$ depending on the kernel $K$ and $r_{1}, r_{2}$. Therefore, using \eqref{eq:Pphi},  
    \begin{align}\label{eq:phiVL}
        \lim_{L\rightarrow \infty} \E\left[V(\psi_L)^2 \right]=0.
    \end{align}

    Finally, we show that $\psi_L := \phi-\phi_L$ is negligible in $L_2$ under $\hat{\mathbb{G}}$ and $\hat{\mathbb{G}}'$. For this, similar to \eqref{eq:Vf},  define 
        $$V_n(f) := \int f(\bm x_1,\ldots,\bm x_{r_{1}};\bm{x}'_1,\ldots, \bm{x}'_{r_{2}}) \prod_{i=1}^{r_{1}} \mathrm d \hat{\mathbb G} (\bm x_i) \prod_{i=1}^{r_{2}} \mathrm d \hat{\mathbb G}' (\bm{x}'_i).$$ 
       Since $\phi$ is completely conditionally degenerate, $\psi_L$ is also  completely conditionally degenerate. Therefore, 
       recalling the definition of $\hat{\mathbb G}$ and $\hat{\mathbb G}'$ from \eqref{eq:G} gives,     
       \begin{align*}
       & V_n( \psi_L ) \\ 
       & =  (nh_{1})^{r_{1}(d_Y+d_Z)/2} (nh_{2})^{ r_{2} d_Z/2}  \int \psi_L(\bm x_1,\ldots,\bm x_{r_{1}};\bm{x}'_1,\ldots, \bm{x}'_{r_{2}}) \prod_{s=1}^{r_{1}} \mathrm d\Tilde{\mathbb P}_{\bm X | \bm Y=\bm y, \bm Z = \bm z} (\bm{x}_i)  \prod_{s=r_{1}+1}^{r_{1}+r_{2}} \mathrm d \Tilde{\mathbb{P}}_{\bm X| \bm Z = \bm z} (\bm{x}'_i) \\  
       & = (nh_{1})^{r_{1}(d_Y+d_Z)/2} (nh_{2})^{ r_{2} d_Z/2} \left[   \frac{ \sum_{  \bm i \in [n]^{r_{1}} } \sum_{  \bm j \in [n]^{r_{2}} } \prod_{s=1}^{r_{1}} w_{(\bm Y_{i_s}, \bm Z_{i_s})}(\bm y, \bm z) \prod_{s=r_{1}+1}^{r_{1}+r_{2}} w_{\bm Z_{j_s}}(\bm z)  \psi_L(\bm X_{\bm i} ; \bm X_{\bm j})   }{ w( \bm y, \bm z)^{ r_{1} } w(\bm z)^{ r_{2} } } \right] . 
       \end{align*}  
By Lemma \ref{lem:main-0},  $\frac{1}{n h_{2}^{d_Z}}w(\bm z) $  and  $\frac{1}{n h_{1}^{d_Y+d_Z}}w(\bm y, \bm z) $  converges in $L_2$ to $p_{\bm Z}(\bm z)$ and $p_{\bm Y, \bm Z}(\bm y, \bm z)$. This implies, defining 
$$T_{n, L} := \frac{1}{ (nh_{1})^{r_{1}(d_Y+d_Z)/2} (nh_{2})^{ r_{2} d_Z/2} } \sum_{  \bm i \in [n]^{r_{1}} } \sum_{  \bm j \in [n]^{r_{2}} } \prod_{s=1}^{r_{1}} w_{(\bm Y_{i_s}, \bm Z_{i_s})}(\bm y, \bm z) \prod_{s=r_{1}+1}^{r_{1}+r_{2}} w_{\bm Z_{j_s}}(\bm z)  \psi_L(\bm X_{\bm i} ; \bm X_{\bm j}) , $$
we have $T_{n, L}/V_n(\psi_L)$ converges in $L_2$ to $p_{\bm Y, \bm Z}(\bm y, \bm z)^{r_{1}} p_{\bm Z}(\bm z)^{r_{2}}$. Therefore, it suffices to shows that $\E[T_{n, L}^2] \rightarrow 0$. Note that
\begin{align*}
    T_{n, L}^2 & = \frac{\sum_{\bm i \in [n]^{2r_{1}}} \sum_{\bm j \in [n]^{2r_{2}}} \prod_{s=1}^{2r_{1}} w_{(\bm Y_{i_s}, \bm Z_{i_s})}(\bm y, \bm z)  \prod_{s=1}^{2r_{2}} w_{\bm Z_{j_s}}(\bm z)  \tilde\psi_L(\bm X_{\bm i}; \bm X_{\bm j}) }{ (nh_{1})^{r_{1}(d_Y+d_Z)} (nh_{2})^{ r_{2} d_Z}  } , 
\end{align*}
where $\tilde\psi_L$ is a function in $2 r_{1} + 2 r_{2}$ variables defined below: 
    \begin{align*}
        & \tilde \psi_L(\bm x_1,\ldots,\bm x_{2r_{2}}, \bm x_{2r_{1}+1}',\ldots,\bm x_{2(r_{1}+r_{2})}')\\
        & = (\phi-\phi_L)(\bm x_1,\ldots, \bm x_{r_{2}}; \bm x_{2r_{1}+1}',\ldots, \bm x_{2r_{1}+r_{2}}' )(\phi-\phi_L)(\bm x_{r_{1}+1},\ldots, \bm x_{2r_{2}}; \bm x_{2r_{1}+r_{1}+1}',\ldots, \bm x_{2r_{2}+2r_{1}}').
    \end{align*} 
 Hence, using Lemma \ref{lm:aux-5}, 
    \begin{align}\label{eq:TnL}
        &  \lim_{n\rightarrow \infty} \E[ T_{n, L}^2 ] \nonumber \\
        & = (r_{1})!(r_{2}) ! \int \tilde \psi_L'(\bm x_1,\bm x_1,\ldots,\bm x_{r_{1}+r_{2}},\bm x_{r_{1}+r_{2}}) \prod_{s= 1 }^{r_{1}} p_{\bm X,\bm Y,\bm Z}(\bm x_s,\bm y, \bm z) \prod_{s=r_{1}+1}^{r_{1}+r_{2}} p_{\bm X,\bm Z}(\bm x_s, \bm z)  \prod_{s=1}^{ (r_{1}+r_{2}) } \mathrm d \bm x_s \nonumber \\ 
        & : = T_L , 
    \end{align}
    where $\psi_L'$ is the symmetrized version (as in \eqref{eq:symver-2}). 
    Now, using the basis expansions \eqref{eq:phifg} and \eqref{eq:phiL} and the Cauchy-Schwarz inequality it follows that 
    \begin{align*}
    T_L & \leq (r_{1})!(r_{2})! p_{\bm Y,\bm Z}(\bm y,\bm z)^{r_{1}} p_{\bm Z}(\bm z)^{r_{2}} \mathbb E[ \psi_L^2( \underline{\bm X}_{[r_{1}+r_{2}]}) \mid \{ \bm Y_s = \bm y \}_{ 1  \leq s \leq r_{1} }, \{ \bm Z_s = \bm z \}_{1 \leq s \leq r_{1}+r_{2}}  ] \nonumber \\ 
    & \rightarrow 0 , 
  \end{align*}
as $L \rightarrow \infty$, by \eqref{eq:phiyz}. Hence, from \eqref{eq:TnL} and recalling that $T_{n, L}/V_n(\psi_L)$ converges in $L_2$ to $p_{\bm Y, \bm Z}(\bm y, \bm z)^{r_{1}} p_{\bm Z}(\bm z)^{r_{2}} $, we have 
\begin{align}\label{eq:phinL}
\lim_{L\rightarrow \infty} \lim_{n\rightarrow\infty}\E[V_n(\psi_L)^2]\rightarrow 0. 
\end{align}
Combining \eqref{eq:phin}, \eqref{eq:phiVL}, and \eqref{eq:phinL}, the result in Lemma \ref{lem:main-6} follows.

\section{Additional Results} 
\label{sec:addpf}

\subsection{A Sufficient Condition for (\ref{eq:condition})} 
\label{sec:conditionpf}

Here, we derive a sufficient condition for a  density $f(\bm x, \bm y)$ to be nice (recall Definition \ref{definition:condition}). 

\begin{lemma}\label{lm:condition}
Suppose $(\bm X, \bm Y)$ is a random variable in $\R^{d_X+d_Y}$ with joint density $f(\bm x, \bm y)$ and marginal densities $f_{\bm X}(\bm x)$ and $f_{\bm Y}(\bm y)$. Moreover, suppose an appropriate interchange of integrals and derivatives is allowed (see \eqref{eq:fxyintegral}) and for every $\bm y,\bm u \in \R^{d_Y}$, $|R_n(\bm x,\bm y;\bm u)|\leq M(\bm x,\bm y;\bm u)$ where $\int M^2(\bm x,\bm y; \bm u) f(\bm x,\bm y)\mathrm d\bm x$ is finite. Then \eqref{eq:condition} holds if $f_{\bm Y}(\bm y)$, 
$$\sup_{\bm u: \|\bm u\|=1}\var\left[\Big(\bm u \frac{\partial}{\partial \bm y} \log f_{\bm Y|\bm X}(\bm Y|\bm X)^\top\Big)^2\Big|\bm Y=\bm y\right], \text{ and }\sup_{\bm u:\|\bm u\|=1}\var\left[\bm u \frac{\partial^2}{\partial \bm y^2} \log f_{\bm Y|\bm X}(\bm Y|\bm X)\bm u^\top\Big|\bm Y=\bm y\right]$$ are all bounded in $\bm y$. 
\end{lemma}

\begin{proof} 
By Taylor's theorem, 
$$f(\bm x,\bm y+\delta_n\bm u) = f(\bm x,\bm y) + \delta_n \bm u \frac{\partial}{\partial \bm y} f(\bm x,\bm y)^\top + \frac{\delta_n^2}{2} \bm u \frac{\partial^2}{\partial \bm y^2} f(\bm x, \bm y_n^\prime) \bm u^\top,$$
for some intermediate $\bm y_n^\prime$ such that $\| \bm y - \bm y_n^\prime\| \leq \delta_n$. Then, as $n\rightarrow\infty$, we have,
\begin{align*}
    R_n(\bm x,\bm y;\bm u)  = \frac{1}{2}\frac{\bm u \frac{\partial^2}{\partial \bm y^2} f(\bm x, \bm y_n^\prime) \bm u^\top}{f(\bm x,\bm y)} & \rightarrow \frac{1}{2}\frac{\bm u \frac{\partial^2}{\partial \bm y^2} f(\bm x, \bm y) \bm u^\top}{f(\bm x,\bm y)} \\
    & = \frac{1}{2}\frac{\bm u \frac{\partial^2}{\partial \bm y^2} f_{\bm Y|\bm X}(\bm y|\bm x) \bm u^\top}{f_{\bm Y|\bm X}(\bm y|\bm x)}\\ 
    & = \frac{1}{2}\left\{\Big(\bm u \frac{\partial}{\partial \bm y} \log f_{\bm Y|\bm X}(\bm y|\bm x)^\top\Big)^2 + \bm u \frac{\partial^2}{\partial \bm y^2} \log f_{\bm Y|\bm X}(\bm y|\bm x)\bm u^\top\right\}.
\end{align*}
If the interchange of integrals and derivatives are allowed, it is easy to see that,
\begin{align}\label{eq:fxyintegral}
    \int \Big(\bm u \frac{\partial}{\partial \bm y} \log f_{\bm Y|\bm X}(\bm y|\bm x)^\top\Big)^2 f(\bm x,\bm y)\mathrm d \bm x = -\int \bm u \frac{\partial^2}{\partial \bm y^2} \log f_{\bm Y|\bm X}(\bm y|\bm x)\bm u^\top f(\bm x,\bm y) \mathrm d \bm x. 
\end{align}
This implies, 
$$\int \left\{ \Big(\bm u \frac{\partial}{\partial \bm y} \log f_{\bm Y|\bm X}(\bm y|\bm x)^\top\Big)^2 + \bm u \frac{\partial^2}{\partial \bm y^2} \log f_{\bm Y|\bm X}(\bm y|\bm x)\bm u^\top \right\} f(\bm x,\bm y) \mathrm d \bm x  = 0.$$
Now, using $|R_n(\bm x, \bm y;\bm u)| \leq M(\bm x,\bm y;\bm u)$ where $\int M^2(\bm x, \bm y;\bm u) f(\bm x,\bm y)\mathrm d\bm x$ is finite and the Reverse Fatou's Lemma we get 
\begin{align}
    & \limsup_{n\rightarrow\infty} \int \big(R_n(\bm x,\bm y;\bm u)\big)^2 f(\bm x,\bm y)\mathrm d\bm x \nonumber\\
    & \leq \int \limsup_{n\rightarrow\infty} \big(R_n(\bm x,\bm y;\bm u)\big)^2 f(\bm x,\bm y)\mathrm d\bm x \nonumber \\
    & = \int \frac{1}{4}\left\{\Big(\bm u \frac{\partial}{\partial \bm y} \log f_{\bm Y|\bm X}(\bm y|\bm x)^\top\Big)^2 + \bm u \frac{\partial^2}{\partial \bm y^2} \log f_{\bm Y|\bm X}(\bm y|\bm x)\bm u^\top\right\}^2 f(\bm x,\bm y)\mathrm d\bm x \nonumber\\
    & = \frac{1}{4} \var\left[\Big(\bm u \frac{\partial}{\partial \bm y} \log f_{\bm Y|\bm X}(\bm Y|\bm X)^\top\Big)^2 + \bm u \frac{\partial^2}{\partial \bm y^2} \log f_{\bm Y|\bm X}(\bm Y|\bm X)\bm u^\top\Big| \bm Y=\bm y\right] f_{\bm Y}(\bm y) . \nonumber  
\end{align}
Hence, \eqref{eq:condition} holds if $f_{\bm Y}(\bm y)$, 
$$\sup_{\bm u: \|\bm u\|=1}\var\left[\Big(\bm u \frac{\partial}{\partial \bm y} \log f_{\bm Y|\bm X}(\bm Y|\bm X)^\top\Big)^2\Big|\bm Y=\bm y\right], \text{ and }\sup_{\bm u:\|\bm u\|=1}\var\left[\bm u \frac{\partial^2}{\partial \bm y^2} \log f_{\bm Y|\bm X}(\bm Y|\bm X)\bm u^\top\Big|\bm Y=\bm y\right]$$ are bounded in $\bm y$. 
\end{proof}

\subsection{Normalized Conditional Ball Divergence} 
\label{sec:RXYZpf}

In this section we prove the following result about the {\it normalized conditional Ball Divergence} (recall \eqref{eq:RXYZ}): 
    \begin{align*}
         \mathcal{R}(\bm X, \bm Y| \bm Z) :=  \frac{\E\left [\Theta^2(\P_{\bm X| \bm Y, \bm Z}, \P_{\bm X| \bm Z}) ) \right]}{\mathbb E\left [\int \Theta^2(\P_{\delta_{\bm x}}, \P_{\bm X| \bm Z}) \mathrm d \P_{\bm X|\bm Y,\bm Z} (\bm x)   \right] } . 
    \end{align*}

\begin{prop}
\label{prop:normalized-cBD}
   The normalized conditional Ball Divergence has the following properties: 
     \begin{enumerate}
         \item[\textnormal{(1)}] $\mathcal{R}(\bm X, \bm Y| \bm Z) \in [0, 1]$, 
         \item[\textnormal{(2)}] $\mathcal{R}(\bm X, \bm Y| \bm Z) = 0$ if and only if ${\bm X\indpt \bm Y | \bm Z}$, and
         \item[\textnormal{(3)}] $\mathcal{R}(\bm X, \bm Y| \bm Z) = 1$ if and only if $\bm X$ is a measurable function of $\bm Y$ and $\bm Z$ almost surely. 
     \end{enumerate}
\end{prop}

\begin{proof} 
    Recalling \eqref{eq:uv} note that 
    \begin{align}\label{eq:thetaAB}
        \Theta^2(F,G) & = A + B , 
    \end{align} 
    where 
\begin{align}\label{eq:FGAB}     
    A & := \int (F-G)^2(\bar{B}(\bm u,\rho(\bm v,\bm u)))\mathrm d F(\bm u)\mathrm d F(\bm v) , \nonumber \\ 
    B & := \int (F-G)^2(\bar{B}(\bm u,\rho(\bm v,\bm u)))\mathrm d G(\bm u)\mathrm d G(\bm v) . 
    \end{align} 
    Now, for any fixed $\bm u,\bm v\in \R^d$, define $Q(\bm w) = \bm 1\{ \bm w \in \bar{B}(\bm u,\rho(\bm v,\bm u)) \} - G\big(\bar{B}(\bm u,\rho(\bm v,\bm u))\big)$. Note that, for $\bm W \sim F$, 
    $$\mathbb E[Q(\bm W)] = \mathbb E \left[ \bm 1\{ \bm W \in \bar{B}(\bm u,\rho(\bm v,\bm u)) \} - G\big(\bar{B}(\bm u,\rho(\bm v,\bm u))\big) \right]  = (F-G)(\bar{B}(\bm u,\rho(\bm v,\bm u)).$$ 
    Hence, by Jensen's inequality, 
    \begin{align}\label{eq:W}
    (F-G)^2(\bar{B}(\bm u,\rho(\bm v,\bm u)) =  \mathbb E[Q(\bm W)]^2 & \leq \mathbb E [Q(\bm W)^2] \nonumber \\  
    & = \mathbb E \left[ \left( \bm 1\{ \bm W \in \bar{B}(\bm u,\rho(\bm v,\bm u)) \} - G\big(\bar{B}(\bm u,\rho(\bm v,\bm u))\big) \right)^2 \right] \nonumber \\ 
    & = \int \left( \bm 1\{ \bm w \in \bar{B}(\bm u,\rho(\bm v,\bm u)) \} - G\big(\bar{B}(\bm u,\rho(\bm v,\bm u))\big) \right)^2 \mathrm d F(\bm w) \nonumber \\ 
    & = \int (\mathbb P_{\delta_{\bm w}} - G)^2\big(\bar{B}(\bm u,\rho(\bm v,\bm u))\big) \mathrm d F(\bm w) . 
    \end{align}     
  Now, combining \eqref{eq:FGAB} and \eqref{eq:W} and applying Fubini's theorem gives, 
    \begin{align}\label{eq:A}
        A & \leq \int \left(\int (\P_{\delta_{\bm w}} - G)^2 (\bar{B}(\bm u,\rho(\bm v,\bm u))) \mathrm d F(\bm u)\mathrm d F(\bm v) \right) \mathrm d F(\bm w) . 
        \end{align}
        Similarly, we get 
    \begin{align}\label{eq:B}
        B & \leq \int \left(\int (\P_{\delta_{\bm w}} - G)^2 (\bar{B}(\bm u,\rho(\bm v,\bm u))) \mathrm d G(\bm u)\mathrm d G(\bm v) \right) \mathrm d F(\bm w).
    \end{align}
   Combining \eqref{eq:A}, \eqref{eq:B}, and recalling the definition of $\Theta^2(F,G)$ from \eqref{eq:thetaAB} then gives, 
    \begin{align}
        \label{eq:BD-ub}
        \Theta^2(F,G) \leq \int \Theta^2(\P_{\delta_{\bm w}}, G) \mathrm d F(\bm w).
    \end{align}
    Setting $F = \P_{\bm X|\bm Y,\bm Z}$ and $G = \P_{\bm X|\bm Z}$ in \eqref{eq:BD-ub} gives, 
    \begin{align}
      \label{eq:cBD-ub}
      \Theta^2(\P_{\bm X|\bm Y,\bm Z},\P_{\bm X|\bm Z}) \leq \int \Theta^2(\P_{\delta_{\bm x}}, \P_{\bm X|\bm Z}) \mathrm d \P_{\bm X|\bm Y,\bm Z}(\bm x).  
    \end{align}
    Therefore, taking expectations on both sides with respect to $(\bm Y, \bm Z)$ shows $\mathcal{R}(\bm X,\bm Y|\bm Z) \in [0,1]$. This proves Proposition \ref{prop:normalized-cBD} (1). 
    
    It is easy to see that $\mathcal{R}(\bm X,\bm Y|\bm Z) = 0$ if and only if $\Theta^2(\P_{\bm X|\bm Y,\bm Z}, \P_{\bm X|\bm Z}) = 0$ almost everywhere, that is, if and only if $\bm X\indpt \bm Y|\bm Z$ (by Proposition \ref{nice-theo-prop}). This shows Proposition \ref{prop:normalized-cBD} (2). 
    
    Next, note that $\mathcal{R}(\bm X,\bm Y|\bm Z) = 1$ if and only if  \eqref{eq:cBD-ub} holds with an equality. This implies, by the equality condition of Jensen's inequality, for almost every $\bm u, \bm v$, 
    $$Q(\bm x) = \bm 1\{ \bm x \in \bar{B}(\bm u,\rho(\bm v,\bm u)) \} - \P_{\bm X | \bm Z}\big(\bar{B}(\bm u,\rho(\bm v,\bm u))\big) \text{ is constant almost everywhere } \P_{\bm X|\bm Y,\bm Z}.$$ 
  This means $\P_{\bm X|\bm Y,\bm Z}$ is a degenerate distribution or, in other words, $\bm X$ is a measurable function of $(\bm Y,\bm Z)$. This proves Proposition \ref{prop:normalized-cBD} (3).  
\end{proof}

\end{document}